\documentclass{pset}
\usepackage{enumerate}
\usepackage{xr}
\title{Generically Stable Measures and Distal Regularity in Continuous Logic}
\author{Aaron Anderson}

\begin{document}

\maketitle

\begin{abstract}
    We develop a theory of generically stable and smooth Keisler measures in NIP metric theories, generalizing the case of classical logic.
    Using smooth extensions, we verify that fundamental properties of (Borel)-definable measures and the Morley product hold in the NIP metric setting.
    With these results, we prove that as in discrete logic, generic stability can be defined equivalently through
    definability properties, statistical properties, or behavior under the Morley product.
    We also examine weakly orthogonal Keisler measures, characterizing weak orthogonality in terms of various analytic regularity properties.
    
    We then examine Keisler measures in distal metric theories, proving that as in discrete logic, distality is characterized by all generically stable measures being smooth, or by all pairs of generically stable measures being weakly orthogonal.
    We then use this, together with our results on weak orthogonality and a cutting lemma, to find analytic versions of distal regularity and the strong Erd\H{o}s-Hajnal property.
\end{abstract}

\section{Introduction}
This article continues the program begun in \cite{anderson1} of studying distal theories in continuous logic.
In that paper, we characterized distal metric structures in terms of the behavior of their indiscernible sequences and a continuous version of strong honest definitions, generalizing \cite{distal_simon} and \cite{cs2}.
It is just as fundamental to define distal structures as those structures where all generically stable Keisler measures are smooth.

Keisler measures, as a real-valued generalization of types, lend themselves naturally to continuous logic.
Despite this, while many properties of types such as definability, finite satisfiability, and generic stability have been generalized both to Keisler measures\cite{HPS} and to types in continuous logic\cite{StabGrps}\cite{random_gs}\cite{khanaki},
the literature is comparatively lacking in simultaneous generalizations to measures in continuous logic.
Thus before we can examine distal metric structures from a Keisler measure perspective, we must generalize these properties, extending the theory of Keisler measures over metric structures from papers such as \cite{random09}, \cite{randomVC}, and \cite{csr}.

Once we understand generically stable Keisler measures in continuous logic, and prove that the Keisler measure definition of distality is equivalent to all other definitions for metric structures,
we may use these measures for combinatorial applications of distality.
We develop continuous logic versions of the distal regularity lemma and (definable) strong Erd\H{o}s-Hajnal property of \cite{distal_reg}.
A forthcoming paper with Ben Yaacov will provide several examples of metric structures to which these results apply\cite{anderson3}.

This contributes to a growing subject of ``tame regularity'' in the analytic setting.
Analytic regularity lemmas replace the graphs of Szemer\'edi's original regularity lemma with real-valued functions, which are decomposed into structured, pseudorandom, and error parts\cite{ls1}.
Under a tameness assumption, such as the function being definable in an NIP\cite{ls2}, $n$-dependent\cite{ct}, or stable\cite{csr} metric structure, this decomposition can be simplified.

The distal analytic regularity lemma, Theorem \ref{thm_distal_reg}, implies that for every $\varepsilon > 0$,
any definable predicate $\phi(x_1,\dots, x_n)$ in a distal structure can be expressed as a the sum of a structured part of bounded complexity
and a particularly well-behaved error part, which is bounded in magnitude by $\varepsilon$ everywhere except on a structured set of small measure.

This in turn implies an analytic version of the strong Erd\H{o}s-Hajnal property:
We say that a predicate $\phi(x_1,\dots,x_n)$ has the \emph{strong Erd\H{o}s-Hajnal property} in some structure $M$
when for every $\varepsilon > 0$, there exists $\delta > 0$ such that for any finite sets $A_i \subseteq M^{x_i}$,
there are subsets $B_i \subseteq A_i$ such that $|B_i| \geq \delta |A_i|$ and for all $b, b' \in B_1 \times \dots \times B_n$,
$|\phi(b) - \phi(b')| \leq \varepsilon$.
Just as \cite[Theorem 3.1]{distal_reg} proves in the discrete case, we show that in continuous logic, distality is equivalent to every definable predicate 
having a definable version of the strong Erd\H{o}s-Hajnal property, where the counting measures on the sets $A_i$ can be replaced with generically stable Keisler measures,
and the sets $B_i$ can be defined uniformly.

\subsection*{Overview and Results}
Section \ref{sec_measures} lays out the basic theory of Keisler measures in continuous logic.
These can be understood either as regular Borel measures on the space of types, or equivalently, as certain linear functionals on the space of definable predicates\cite{csr}.
Most importantly for studying distality, we study weak orthogonality and smooth measures, following the approach of \cite{simon_distal_reg}. 
We characterize weakly orthogonal measures as those where the following equivalent conditions hold:
\begin{cor}[Corollary \ref{cor_ortho}]
    Let $x_1,\dots,x_n$ be variable tuples, and let $\mu_i \in \mathfrak{M}_{x_i}(M)$ be Keisler measures on $x_i$ for each $i$.
    The measures $\mu_i$ are weakly orthogonal, meaning that there is a unique measure $\omega \in \mathfrak{M}_{x_1,\dots,x_n}(M)$ on $(x_1,\dots,x_n)$ extending the product measure of $\mu_1, \dots, \mu_n$,
    if and only if for every $M$-definable predicate $\phi(x_1,\dots,x_n)$ and every $\varepsilon > 0$, there exist $M$-definable predicates $\psi^-(x_1,\dots,x_n),\psi^+(x_1,\dots,x_n)$,
where $\psi^\pm(x_1,\dots,x_n)$ are each of the form $\sum_{j = 1}^m \prod_{i = 1}^n\theta_{ij}^{\pm}(x_i)$, such that
\begin{itemize}
    \item For all $(x_1,\dots,x_n)$, $\psi^-(x_1,\dots,x_n) \leq \phi(x_1,\dots,x_n) \leq \psi^+(x_1,\dots,x_n)$.
    \item For any product measure $\omega$ of $\mu_1,\dots,\mu_n$, $\int_{S_{x_1\dots x_n}(M)}(\psi^+ - \psi^-)\,d\omega \leq \varepsilon$.
\end{itemize}
\end{cor}
From this perspective, we consider smooth measures - all measures such that there is a small model $M$ such that $\mu|_M$ has a unique global extension.
We characterize them also as the measures that are weakly orthogonal to all measures, or equivalently all types.
We also examine invariant and (Borel)-definable measures, extending the careful work in \cite{cgh} on Morley products in NIP to continuous logic.
This approach revolves around the fact that any measure in an NIP theory admits a smooth extension, which we verify for continuous logic in Lemma \ref{lem_smooth_extensions}.
We are then able to use smooth extensions of measures as we would use realizations of types.

In Section \ref{sec_gs}, we turn to generically stable measures, finding many equivalent continuous logic characterizations of these versatile measures,
culminating with a generalization of \cite[Theorem 3.2]{HPS} to continuous logic:
\begin{thm}[Thm \ref{thm_gen_stable}]
    Assume $T$ is an NIP metric theory.
    For any small model $M \subseteq \mathcal{U}$, if $\mu$ is a global $M$-invariant measure, the following are equivalent:
    \begin{enumerate}[(i)]
        \item $\mu$ is a frequency interpretation measure (fim) over $M$ (see Definition \ref{defn_measure})
        \item $\mu$ is a finitely approximated measure (fam) over $M$ (see Definition \ref{defn_measure})
        \item $\mu$ is definable over and approximately realized in $M$ (see Definition \ref{defn_measure})
        \item $\mu(x) \otimes \mu(y) = \mu(y) \otimes \mu(x)$ (see Definition \ref{defn_product})
        \item $\mu^{(\omega)}(x_0,x_1,\dots)|_{M}$ is totally indiscernible (see Definition \ref{defn_product}).
    \end{enumerate}
\end{thm}
This connects the topological properties of generically stable measures (definability and approximate realizability) and the behavior of the Morley product to the property of being a frequency interpretation measure (fim).
Classically, these are measures against which formulas obey a version of the VC-theorem.
We show that in continuous logic, definable predicates and generically stable measures satisfy various properties that were shown in \cite[Section 2]{anderson1} for definable predicates and finitely-supported measures.
This includes a Glivenko-Cantelli property analogous to the VC-Theorem (our definition of fim) as well as bounds on the sizes of $\varepsilon$-approximations (Corollary \ref{cor_keisler_approx}) and $\varepsilon$-nets (Theorem \ref{thm_keisler_net}).

Before approaching distal regularity directly, we connect weak orthogonality of measures to regularity properties in Section \ref{sec_wo}.
In \cite{simon_distal_reg}, the distal regularity lemma are proven using weak orthogonality.
Before assuming distality, we develop the nomenclature for expressing this regularity lemma and the (definable) strong Erd\H{o}s-Hajnal property in continuous logic,
generalizing \cite[Theorems 3.1 and 5.8]{distal_reg} in the discrete case,
and we are able to prove non-uniform versions of these regularity lemmas for any weakly orthogonal measures:
\begin{thm}[Theorem \ref{thm_ortho_reg}]
    Let $\mu_1 \in \mathfrak{M}_{x_1}(M),\dots,\mu_n \in \mathfrak{M}_{x_n}(M)$.
    The following are equivalent:
    \begin{itemize}
        \item The measures $\mu_1,\dots,\mu_n$ are weakly orthogonal.
        \item For each $M$-definable predicate $\phi(x_1,\dots,x_n)$ and each $\varepsilon, \delta > 0$, there is some $C$ such that $\phi$ admits a definable $(\varepsilon,\delta)$-distal regularity partition (see Definitions \ref{defn_constructible} and \ref{defn_distal_reg_part})
        \item For each $M$-definable predicate $\phi(x_1,\dots,x_n)$ and each $\varepsilon, \delta > 0$, there is some $C$ such that $\phi$ admits a constructible $(\varepsilon,\delta)$-distal regularity partition (see Definitions \ref{defn_constructible} and \ref{defn_distal_reg_part})
        \item For each $M$-definable predicate $\phi(x_1,\dots,x_n)$ and each $\varepsilon > \gamma \geq 0$, there is some $\delta > 0$ such that for any product measure $\omega$ of continuous localizations of $\mu_1,\dots,\mu_n$,
            if $\int_{S_{x_1\dots x_n}(M)}\phi\,d\omega \geq \varepsilon$, then there are $M$-definable predicates $\psi_i(x_i)$ 
            such that $\phi(a_1,\dots,a_n) \geq \gamma$ whenever $\psi_i(a_i) > 0$ for each $i$, and $\int_{S_{x_i}(M)}\psi_i(x_i)\,d\mu_i \geq \delta$ for each $i$.
        \item For each $M$-definable predicate $\phi(x_1,\dots,x_n)$ and each $\varepsilon > 0$, $\phi$ has the definable $\varepsilon$-SEH with respect to any continuous localizations of $\mu_1,\dots,\mu_n$ (see Definition \ref{defn_SEH}).
    \end{itemize}
    Furthermore, if these hold, then the $(\varepsilon,\delta)$-distal regularity partitions can be chosen to be grid partitions of size $O(\delta^{-C})$ for some constant $C$ depending on $\phi,\varepsilon,\mu_1,\dots,\mu_n$. 
\end{thm}

Having explored Keisler measures in NIP metric structures, we turn to distality in Section \ref{sec_distal_measure}.
First we characterize distal metric theories in terms of Keisler measures:
\begin{thm}[Theorem \ref{thm_distal_measure}]
    The following are equivalent:
    \begin{itemize}
        \item The theory $T$ is distal
        \item Every generically stable measure is smooth
        \item All pairs of generically stable measures are weakly orthogonal.
    \end{itemize}
\end{thm}
We then apply distality to the regularity results of Section \ref{sec_wo},
showing that the regularity lemmas hold uniformly, getting a continuous logic version of the distal regularity lemma from \cite{distal_reg}:
\begin{thm}[Theorem \ref{thm_distal_reg}]
    Assume $T$ is distal.
    For each definable predicate $\phi(x_1,\dots,x_n;y)$ and $\varepsilon > 0$, there exist predicates $\psi_i(x_i;z_i)$, which can be chosen to be either definable or constructible, and a constant $C$ such that
    if $\mu_1 \in \mathfrak{M}_{x_1}(M),\dots,\mu_n \in \mathfrak{M}_{x_n}(M)$ are such that for $i < n$, $\mu_i$ is generically stable, $b \in M^y$, and $\delta > 0$, the following all hold:
    The predicate $\prod_{i = 1}^n \psi_i(x_i;z_i)$ defines a $(\varepsilon,\delta)$-distal regularity grid partition for $\phi(x_1,\dots,x_n;b)$ of size $O(\delta^{-C})$.
\end{thm}

\subsection*{Acknowledgements}
We thank Artem Chernikov for advising and support throughout this project, Ita\"i Ben Yaacov for advising while the author was in Lyon, and James Hanson for several helpful ideas and conversations about continuous logic.
The author was partially supported by the Chateaubriand fellowship, the UCLA Logic Center, the UCLA Dissertation Year Fellowship,
and NSF grants DMS-1651321 and DMS-2246598
during the writing process.

\section{Keisler Measures in Metric Theories}\label{sec_measures}
In this section, we translate some of the theory of Keisler measures to continuous logic, building on the definitions in \cite{random09}, \cite{randomVC}, and \cite{csr}.
Throughout, $T$ will be a complete metric theory, $M$ will be metric structure modelling $T$ contained in a monster model $\mathcal{U}$,
and for any $A \subseteq M$, $S_x(A)$ will be the space of types in variables $x$ with parameters in $A$.
For a general reference on metric structures, see \cite{mtfms}.

\begin{defn}
A \emph{Keisler measure} on $S_x(M)$ is a regular Borel probability measure on $S_x(M)$. We denote the space of such measures $\mathfrak{M}_x(M)$.
\end{defn}
It is noted in \cite{csr} that these are in bijection with \emph{Keisler functionals}, that is, positive linear functionals on $\mathcal{C}(S_x(M),\R)$ with $\norm{f} = 1$.
We give the space $\mathfrak{M}_x(M)$ of Keisler measures the weak$^*$ topology as positive linear functionals,
the coarsest topology such that every definable predicate $\phi(x)$ with parameters in $M$, $\mu \mapsto \int_{S_x(M)} \phi(x)\,d\mu$ is continuous.
This generalizes the \emph{compact Hausdorff topology} used for Keisler measures in classical logic in \cite{gannon}.
We also see that for every definable predicate $\phi(x)$ with parameters in $M$, as $\phi(x)$ is the uniform limit of a sequence of formulas,
$\mu \mapsto \int_{S_x(M)} \phi(x)\,d\mu$ is the uniform limit of a sequence of integrals of formulas, each of which is a continuous function, and is thus continuous.

We now present continuous analogs for several key properties that global Keisler measures (measures in $\mathfrak{M}_x(\mathcal{U})$) can have.

\begin{defn}\label{defn_measure}
Let $\mu$ be a global Keisler measure, and let $A \subseteq \mathcal{U}$ be a small set, and $M \preceq \mathcal{U}$ a small model.
\begin{itemize}
    \item We say $\mu$ is $A$-\emph{invariant} when for any tuples $a \equiv_A b$ in $\mathcal{U}^y$, and any formula $\phi(x;y) \in \mathcal{L}(A)$, $\int\phi(x;a)\,d\mu = \int\phi(x;b)\,d\mu$. Equivalently, any automorphism of $\mathcal{U}$ fixing $A$ preserves $\mu$.
    \item If $\mu$ is $A$-invariant, define the map $F^\phi_{\mu,A} : S_y(A) \to [0,1]$ by $F^\phi_{\mu,A}(p) = \int \phi(x;b)\,d\mu$ for $b \models p$.
    \item We say $\mu$ is $A$-\emph{Borel definable} when it is $A$-\emph{invariant} and for all $\phi(x;y) \in \mathcal{L}(A)$, the map $F^\phi_{\mu,A}$ is Borel.
    \item We say $\mu$ is $A$-\emph{definable} when it is $A$-\emph{invariant} and for all $\phi(x;y) \in \mathcal{L}(A)$, the map $F^\phi_{\mu,A}$ is continuous (and thus a definable predicate).
    \item We say $\mu$ is \emph{approximately realized in} $A$ when $\mu$ is in the topological closure of the convex hull of the Dirac measures at types of points in $A$. This corresponds to finite satisfiability.
    \item Keeping discrete notation, we call a definable, approximately realized measure \emph{dfs} (for definable, finitely satisfiable).
    \item We say $\mu$ is \emph{finitely approximated in} $M$ when for every $\varphi(x;y) \in \mathcal{L}(M)$ and every $\varepsilon > 0$, there exists a tuple $(a_1,\dots,a_n) \in (M^x)^n$ which is a $\varepsilon$-approximation for the family $\{\varphi(x;b) : b \in \mathcal{U}^y\}$ with respect to $\mu$. We abbreviate this property as \emph{fam}.
    \item We say $\mu$ is a \emph{frequency interpretation measure over} $M$ when for every $\varphi(x;y) \in \mathcal{L}(M)$, there is a family of formulas $(\theta_n(x_1,\dots,x_n) : n \in \omega)$ with parameters in $M$ such that $\lim_{n \to \infty} \mu^{(n)}(\theta_n(x_1,\dots,x_n)) = 1$, and for every $\varepsilon > 0$, for large enough $n$, any $\bar a \in (\mathcal{U}^x)^n$ satisfying $\theta_n(\bar a)$ is a $\varepsilon$-approximation to $\varphi(x;y)$ with respect to $\mu$. We abbreviate this property as \emph{fim}.
    \item We say $\mu$ is \emph{smooth} over $M$ when for every $N$ with $M \preceq N$, there exists a unique extension $\mu' \in \mathfrak{M}_x(N)$ of $\mu|_M$.
\end{itemize}
\end{defn}

Note that if $\mu$ is $A$-invariant, then $F^\phi_{\mu,A}$ can also be defined for $\phi(x;y)$ a definable predicate.
Any definable predicate is a uniform limit of formulas, so not only will $F^\phi_{\mu,A}$ be well-defined, but it will be the uniform limit of functions of the form $F^{\psi}_{\mu,A}$ where $\psi$ is a formula.
Thus if $\mu$ is $A$-Borel definable, the function $F^\phi_{\mu,A}$ will be Borel for $\phi$ a definable predicate, and if $\mu$ is $A$-definable, $F^\phi_{\mu,A}$ will be continuous.
While we will often prove results for Borel definable measures for full generality, we will eventually show that in the NIP context, these are the same as invariant measures (see Lemma \ref{lem_inv_borel}).

We will need to be able to consider sequences which are indiscernible with respect to $\mu$ in a certain sense,
for which we will need the following definitions.
\begin{defn}
    Let $\mathcal{L}_{\mathbb{E}}$ be an extension of the language $\mathcal{L}$ to add a relation symbol $\mathbb{E}_\psi(y)$ for each restricted formula $\psi(x;y)$,
    with $\mathbb{E}_\psi(y)$ having the same Lipschitz constant as $\psi$.
    
    If $M$ is a model and $\mu \in \mathfrak{M}_x(M)$, let $(M;\mu)$ be the $\mathcal{L}_{\mathbb{E}}$-structure so that for all $b \in M^y$,
    $\mathbb{E}_\psi(b) = \int_{S_x(M)}\psi(x;b)\,d\mu$.
    \end{defn}
    The metric structure $(M;\mu)$ is valid because the integral of a $C$-Lipschitz function is also $C$-Lipschitz.
    Then by density and the fact that uniform limits commute with integrals,
    for any $\mathcal{L}$-definable predicate $\psi(x;y)$, we can define a $\mathcal{L}_{\mathbb{E}}$-definable predicate
    $\mathbb{E}_\psi(y)$, which we interpret as $\int_{S_x(M)}\psi(x;y)\,d\mu$.

    \begin{lem}[{Generalizes \cite[Prop 7.5]{nip_guide}}]\label{lem_measure_consistent}
        Let $\mu \in \mathfrak{M}_x(\mathcal{U})$ be a global measure, $(b_i : i < \omega)$ an indiscernible sequence.
        Let $\phi(x;y)$ be a formula, and let $0 < s < r$ be such that
        $$\int_{S_x(\mathcal{U})}(\phi(x;b_i))\,d\mu \geq r$$ for all $i < \omega$.
        Then the partial type $\{\phi(x;b_i) \geq s : i < \omega\}$ is consistent.
    \end{lem}
    \begin{proof}
    We can use Ramsey and compactness to extract an $\mathcal{L}_{\mathbb{E}}$-indiscernible in an elementary extension of $(M;\mu)$ satisfying the EM-type of $(b_i : i < \omega)$.
    In particular, for every formula $\psi(x;y_1,\dots,y_n)$ (not just the restricted ones), $\int_{S_x(M')}\psi(x;b_{i_1},\dots,b_{i
    n})\,d\mu'$ takes the same value for all $i_1< \dots < i_n \in \N$.
    Thus we can assume that the sequence $(b_i : i < \omega)$ was already indiscernible in this extended language.
    
    Assume for contradiction that $\{\phi(x;b_i) \geq s : i < \omega\}$ is inconsistent.
    Thus for some $N$,
    $$\min_{i = 0}^N \left(\phi(x;b_i)\dot{-}s\right) = 0$$
    indentically, and in particular,
    $$\int_{S_x(M)}\min_{i = 0}^N \left(\phi(x;b_i)\dot{-}s\right)\, d\mu = 0.$$
    Let $N$ be the minimal such value, and let $t = \int_{S_x(M)}\min_{i = 0}^{N - 1} \left(\phi(x;b_i)\dot{-}s\right)\,d\mu'$.
    Note that $t > 0$, as
    $$t =
    \int_{S_x(M)}\min_{i = 0}^{N - 1} \left(\phi(x;b_i)\dot{-}s\right)\,d\mu'
    \geq \int_{S_x(M)}\phi(x;b_0)\,d\mu' - s  = r - s > 0.$$
    
    Then define
    $$\psi_k(x) = \min_{i = 0}^{N - 1} \left(\phi(x;b_i)\dot{-}s\right),$$
    and observe that
    $$\min(\psi_0(x),\psi_1(x)) = \min_{i = 0}^{2N-1} \left(\phi(x;b_i)\dot{-}s\right) \leq \min_{i = 0}^{N} \left(\phi(x;b_i)\dot{-}s\right) = 0,$$
    so by indiscernibility, for all $i < j$,
    $\int_{S_x(M)}\min(\psi_i(x),\psi_j(x))\,d\mu = 0.$
    Thus for any $i_1 < \dots < i_m$,
    \begin{align*}
        \int_{S_x(M)}\max_{1 \leq j \leq m}\psi_{i_j}(x)\,d\mu
        & = \int_{S_x(M)}\sum_{j = 1}^m\psi_{i_j}(x)\,d\mu\\
        & = mt \\
    \end{align*}
    and for $m > \frac{1}{t}$, this gives $\int_{S_x(M)}\max_{1 \leq j \leq m}\psi_{i_j}(x)\,d\mu' > 1$,
    a contradiction because $\max_{1 \leq j \leq m}\psi_{i_j}(x)$ is bounded above by 1.
    \end{proof}

\begin{defn}\label{defn_support}
For any measure $\mu \in \mathfrak{M}_x(A)$, define $S_\mu(x)$ to be the partial type consisting of all closed $A$-conditions with $\mu$-measure 1.
We also define $S(\mu) \subseteq S_x(A)$ to be the set of all types satisfying $S_\mu(x)$, we call this \emph{the support} of $\mu$.
\end{defn}
Clearly the intersection of finitely many closed conditions in $S_\mu(x)$ has $\mu$-measure 1, so any finite subtype of $S_\mu(x)$ is satisfiable.

\begin{lem}\label{lem_inv_support}
Assume that $T$ is NIP.
Let $A\subset \mathcal{U}$ be such that $\mathcal{U}$ is $|A|^+$-saturated, let $\mathfrak{M}_x(\mathcal{U})$ be an $A$-invariant measure, and $p(x) \in S(\mu)$.
Then $p$ is $A$-invariant, meaning that for any formula $\phi(x;y)$, and any $b,b' \in \mathcal{U}$ with $b \equiv_A b'$, $(\phi(x;b) = \phi(x;b')) \in p(x)$.
\end{lem}
\begin{proof}
    Let $p(x) \in S(\mu)$ and let $\phi(x;y)$ be an $A$-formula.
    Then for any $b \equiv_A b'$, by $A$-invariance, $\int_{S_x(\mathcal{U})}\phi(x;b)\,d\mu = \int_{S_x(\mathcal{U})}\phi(x;b')\,d\mu$.
    Assume $p(x)$ is not $A$-invariant.
    Then there exist $b \equiv_A b'$ with $(\phi(x;b) = \phi(x;b'))\not\in p(x)$.
    Then without loss of generality, there is some $\varepsilon > 0$ with $\phi(x;b) = \phi(x;b') + \varepsilon \in p(x)$.
    Meanwhile, $\int_{S_x(\mathcal{U})}\phi(x;b)\,d\mu =\int_{S_x(\mathcal{U})} \phi(x;b')\,d\mu$.
    We will show that $\int_{S_x(\mathcal{U})}|\phi(x;b) - \phi(x;b')|\,d\mu = 0$.
    
    Assume for contradiction that $\int_{S_x(\mathcal{U})}|\phi(x;b) - \phi(x;b')|\,d\mu = \varepsilon > 0$.
    Then because $b \equiv_A b'$, we may find an $A$-indiscernible sequence $(b_i : i < \omega)$ with $b_0 = b$, $b_1 = b'$.
    For all $i$, we have $b_{2i}b_{2i + 1} \equiv_A bb'$, and by invariance of $\mu$,
    $\int_{S_x(\mathcal{U})}|\phi(x;b_{2i}) - \phi(x;b_{2i+1})|\,d\mu = \int_{S_x(\mathcal{U})}|\phi(x;b) - \phi(x;b')|\,d\mu = \varepsilon$.
    Thus by Lemma \ref{lem_measure_consistent},
    the partial type $\left\{|\phi(x;b_{2i}) - \phi(x;b_{2i+1})| \geq \frac{\varepsilon}{2} : i < \omega\right\}$ is consistent.
    This contradicts NIP.
\end{proof}

\subsection{(Borel) Definable Measures and the Morley Product}

\begin{lem}\label{lem_inv_borel}
    Let $\mu \in \mathfrak{M}_x(\mathcal{U})$ be $A$-(Borel) definable, and let $A \subseteq B \subseteq \mathcal{U}$.
    Then $\mu$ is $A$-(Borel) definable if and only if $\mu$ is $B$-(Borel) definable.
    In particular, if either holds, $\mu$ is $C$-(Borel) definable whenever $\mu$ is $C$-invariant.
\end{lem}
\begin{proof}
    The proof is essentially the same as the version for discrete logic (see \cite[Proposition 2.22]{gannon} \cite[Corollary 2.2]{cgh}).

    The map $\pi_{B,A} : S_x(B) \to S_x(A)$ given by $\pi_{B,A}(p) = p|_A$ is continuous, surjective, and closed\cite[Prop. 8.11]{mtfms}, and
    $F^{\phi}_{\mu, B} = F^\phi_{\mu,A} \circ \pi_{B,A}$.
    These properties of $\pi_{B,A}$ imply that $F^\phi_{\mu,A}$ is continuous/Borel if and only if $F^\phi_{\mu,A} \circ \pi_{B,A}$ is.
    Most of these implications are straightforward, but it is nontrivial that $F^\phi_{\mu,B}$ being Borel implies $F^{\phi}_{\mu, A}$ is as well.
\end{proof}

Borel definable measures are important largely because they are the measures for which we can define the Morley product of Keisler measures.

\begin{defn}\label{defn_product}
    Given an $A$-Borel definable measure $\mu$ and a global measure $\nu$, let $f_{\mu \otimes \nu}$ be the Keisler functional defined by
    $$f_{\mu \otimes \nu}(\phi(x;y)) = \int_{S_y(A')}F^\phi_{\mu,A'}(y)\,d\nu|_{A'}$$
    where $\phi(x;y)$ is a formula, and $A'$ contains $A$ and the parameters of $\phi$.
    Let $\mu \otimes \nu$ be the corresponding Keisler measure, so that
    $$\int_{S_{xy}(\mathcal{U})}\phi(x;y)\,d(\mu \otimes \nu) = \int_{S_y(A')}F^\phi_{\mu,A'}\,d\nu|_{A'}$$
    for all formulas $\phi(x;y)$ with parameters in $A' \supset A$.
\end{defn}

First we check that this definition does not depend on the choice of $A'$. It is enough to see that if $A'$ is enlarged to $B \supset A'$, that the value will not change.
In this case, if $\pi:S_y(B) \to S_y(A')$ is the projection map, then it is easy to see that $\nu|_{A'}$ is equal to the pushforward measure $\pi_* \nu|_{B}$. 
Also, $F^\phi_{\mu,A'} = F^\phi_{\mu,A'} \circ \pi$. Thus 
$$\int_{S_y(A')}F^\phi_{\mu,A'}(y)\,d\nu|_{A'} = \int_{S_y(B)}F^\phi_{\mu,A'}(y) \circ \pi \,d\nu|_{B} = \int_{S_y(B)}F^\phi_{\mu,B}(y),d\nu|_{B}.$$

This indeed defines a valid Keisler functional, as it is clearly linear and
$$f_{\mu \otimes \nu}(1) = \int_{S_y(A)}F^1_{\mu,A}\,d\nu|_A = 1.$$
It is also easy to see that if $\mu$ is $A$-Borel definable and $\nu$ is $A$-invariant, then $\mu \otimes \nu$ is also $A$-invariant.
Also, we see that for any $A$ such that $\mu$ is $A$-Borel definable and the parameters of $\phi(x;y)$ are contained in $A$, the value of
$\int_{S_y(\mathcal{U})}\phi(x;y)\,d(\mu\otimes\nu)$ depends only on $\nu|_A$.

\begin{lem}[{Generalizing \cite[Prop. 2.6]{cg}}]\label{lem_definable_product}
If $\mu \in \mathfrak{M}_x(\mathcal{U}), \nu \in \mathfrak{M}_y(\mathcal{U}), \lambda \in \mathfrak{M}_z(\mathcal{U})$
    are $M$-definable measures, then $\mu \otimes \nu$ is $M$-definable, and $(\mu \otimes \nu) \otimes \lambda = \mu \otimes (\nu \otimes \lambda)$.
\end{lem}
\begin{proof}
First we will show that $\mu \otimes \nu$ is definable by showing that for all formulas $\phi(x,y;z) \in \mathcal{L}(M)$, the function $F^{\phi(x,y;z)}_{\mu \otimes \nu} : S_z(M) \to [0,1]$ is continuous.
We can see that
$$F^{\phi(x,y;z)}_{\mu \otimes \nu} = \int_{S_{xy}(M)}\phi(x,y;z)\,d(\mu \otimes \nu) = \int_{S_y(M)}\left(\int_{S_x(M)}\phi(x,y;z)\,d\mu\right)\,d\nu.$$
As $\mu$ is definable, the function $F^{\phi(x;y,z)}_{\mu,M} = \int_{S_x(M)}\phi(x;y,z)\,d\mu$ is continuous, and is thus a definable predicate on $(y,z)$.
Thus as $\nu$ is definable, $\int_{S_y(M)}\left(\int_{S_x(M)}\phi(x,y;z)\,d\mu\right)\,d\nu$ is continuous as desired.

Now to verify associativity, it is enough to show that for all formulas $\phi(x,y,z) \in \mathcal{L}(M)$,
$\int\phi(x,y,z)\,d((\mu\otimes \nu)\otimes \lambda) = \int\phi(x,y,z)\,d(\mu\otimes (\nu\otimes \lambda))$.
We can see this in the simple-looking calculation
\begin{align*}
    \int_{S_{xyz}(M)}\phi(x,y,z)\,d((\mu\otimes \nu)\otimes \lambda)
    &= \int_{S_z(M)}\left(\int_{S_{xy}(M)}\phi(x,y,z)\,d(\mu \otimes \nu)\right)\,d\lambda\\
    &= \int_{S_z(M)}\left(\int_{S_{y}(M)}\left(\int_{S_{x}(M)}\phi(x,y,z)\,d\mu\right)\,d\nu\right)\,d\lambda\\
    &= \int_{S_{yz}(M)}\left(\int_{S_{x}(M)}\phi(x,y,z)\,d\mu\right)\,d(\nu \otimes \lambda)\\
    &= \int\phi(x,y,z)\,d(\mu\otimes (\nu\otimes \lambda)).
\end{align*}
These equations are justified by the definition of the Morley product, together with the fact that all the functions being integrated are continuous, and thus are definable predicates.
This continuity follows from the definability of $\mu, \nu, \mu \otimes \nu$.
\end{proof}

It will also be useful to generalize some of the behavior of continuous functions with respect to Morley products of definable measures to characteristic functions of open sets.
\begin{lem}[{Generalizing \cite[Prop. 2.17]{cgh}}]\label{lem_BD++}
If $\mu \in \mathfrak{M}_x(\mathcal{U})$ is $A$-definable, then for any open set $U \subseteq S_{xy}(A)$, the function $\int_{S_x(A)}\chi_U(x,y)\,d\mu$ is itself Borel,
and for any $\nu \in \mathfrak{M}_y(\mathcal{U})$, 
$$\int_{S_y(A)}\int_{S_x(A)}\chi_U(x,y)\,d\mu\,d\nu = \int_{S_{xy}(A)}\chi_U(x,y)\,d \mu \otimes \nu = (\mu \otimes \nu)(U).$$
\end{lem}
\begin{proof}
Fix $\mu$ and $U$. Let $\mathcal{F}$ be the set of all continuous functions $f : S_{xy}(A) \to [0,1]$ such that $f \leq \chi_U$ pointwise (in other words, the support of $f$ is contained in $U$).
The set $\mathcal{F}$ is a directed partial order (with pointwise $\leq$), and thus the function $f \mapsto \int_{S_x(A)}f(x,y)\,d\mu$ with domain $\mathcal{F}$ is an increasing net of continuous functions.
We can show that the pointwise limit of this net is $\int_{S_x(A)}\chi_U(x,y)\,d\mu$.
As for all functions $f \in \mathcal{F}$, $\int_{S_x(A)}f(x,y)\,d\mu \leq \int_{S_x(A)}\chi_U(x,y)\,d\mu$, and the net is increasing,
it suffices to show that for each $q \in S_{y}(A)$ with $b \vDash q$ and each $\varepsilon > 0$, there is some $f \in \mathcal{F}$ with
$f(q)\geq \int_{S_x(A)}\chi_U(x,b)\,d\mu - \varepsilon$.
Let $C \subseteq S_x(A)$ be a closed subset of the open fiber $U_b = \{p \in S_x(A) : (a,b) \in U\textrm{ for } a \vDash p\}$ with $\mu(C) \geq \mu(U_b) - \varepsilon$.
Let $C' \subseteq S_{xy}(A)$ be the closed set $\{\mathrm{tp}(a,b/A): \mathrm{tp}(a/A) \in C\}$.
By Urysohn's lemma, there is a continuous function $f$ with support contained in $U$ with value $1$ on all of $C'$.
Thus $f \leq \chi_U(x,y)$, and
\begin{align*}
    \int_{S_x(A)}f(x,b)\,d\mu 
    \geq \mu(C)
    \geq \mu(U_b) - \varepsilon
    = \int_{S_x(A)}\chi_U(x,y)\,d\mu - \varepsilon,
\end{align*}
so $f$ is the function we desired.

By the monotone convergence theorem for nets (\cite[Theorem IV.15]{reed_simon}), the pointwise limit of an increasing net of uniformly bounded continuous functions is Borel, and its integral relative to a regular Borel measure such as $\nu$ is the limit of the integrals of the functions in the net.
Thus $\int_{S_x(A)}\chi_U(x,y)\,d\mu$ is Borel, and $\int_{S_y(A)}\int_{S_x(A)}\chi_U(x,y)\,d\mu\,d\nu = \lim_{f \in \mathcal{F}} \int_{S_y(A)}\int_{S_x(A)}f(x,y)\,d\mu\,d\nu$.
By the definition of the Morley product, $\lim_{f \in \mathcal{F}} \int_{S_y(A)}\int_{S_x(A)}f(x,y)\,d\mu\,d\nu = \lim_{f \in \mathcal{F}} \int_{S_{xy}(A)}f(x,y)\,d\mu \otimes \nu$,
and once again using Urysohn's lemma, it is straightforward to find $f \in \mathcal{F}$ with $\int_{S_{xy}(A)}f(x,y)\,d\mu \otimes \nu \geq (\mu\otimes \nu)(U) - \varepsilon$ for each $\varepsilon$, so this limit is $(\mu\otimes \nu)(U)$.
\end{proof}

\subsection{Approximately Realizable Measures}

We provide another characterization of approximately realized measures, which justifies the name:
\begin{lem}\label{lem_approx_realized}
Let $\mu \in \mathfrak{M}_x(\mathcal{U})$ be a global measure.

Then $\mu$ is approximately realized in $A$ if and only if the following holds:

For every predicate $\phi(x)$ with parameters in $\mathcal{U}$, if $\phi(a) = 0$ for all $a \in A^x$, then $\int_{S_x(\mathcal{U})}\phi(x)\,d\mu = 0$.
\end{lem}
\begin{proof}
First, we show that this holds for all approximately realized measures.
    
Assume $\phi(x)$ is a predicate such that for all $a \in A^x$, $\phi(a) = 0$.

Let $\nu$ be a convex combination of Dirac measures at types of points in $A$ - specifically, let
$\nu = \sum_{i = 1}^n \lambda_i \delta_{a_i}$, where $\delta_{a_i}$ is the Dirac measure at the type realized by $a_i \in A$, and $\lambda_i \geq 0$, $\sum_{i = 1}^n \lambda n = 1$.
Then $\int_{S_x(\mathcal{U})}\phi(x)\,d\nu = \sum_{i =1}^n \lambda_i \phi(a_i) = 0$.

We then recall that $\nu \mapsto \int_{S_x(\mathcal{U})}\phi(x)\,d\nu$ is continuous, so
as this continuous function takes the value 0 everywhere in a set, it must take the value 0 everywhere in its closure - the set of approximately realized measures.

Now we will show that any measure with this property is approximately realized in $A$.
Assume $\mu$ is not approximately realized in $A$, and we will find some predicate $\phi(x)$ such that
$\phi(a) = 0$ for all $a \in A^x$, but $0 < \int_{S_x(\mathcal{U})}\phi(x)\,d\mu$.

Because $\mu$ is not approximately realized in $A$, $\mu$ is contained in an open set that does not contain any convex combinations of Dirac measures of types realized in $A$.
We may assume that the open set is basic - a finite intersection of sets of the form $\{\nu : r < \int_{S_x(\mathcal{U})}\phi(x)\,d\nu < s\}$ where $r < s$ and $\phi(x)$ is a formula with parameters.
By potentially replacing $\phi(x)$ with $1 - \phi(x)$, we may assume that this set is an intersection of sets of the form 
$\{\nu : r < \int_{S_x(\mathcal{U})}\phi(x)\,d\nu\}$, and by replacing $\phi(x)$ with $\phi(x)\dot{-} r$, we can replace these with sets of the form 
$\{\nu : 0 < \int_{S_x(\mathcal{U})}\phi(x)\,d\nu\}$.
Thus assume there are formulas $\phi_1,\dots, \phi_n$ such that for each $i$, $0 < \int_{S_x(\mathcal{U})}\phi_i(x)\,d\mu$, but
for each convex combination $\nu$ of Dirac measures at types realized in $A$, $\int_{S_x(\mathcal{U})}\phi_i(x)\,d\nu = 0$ for some $i$.
We wish to show that for some $i$, $\phi_i(a) = 0$ for all $a \in A^x$.
If not, then for each $i$, let $a_i \in A^x$ be such that $\phi_i(a_i)> 0$.
Then let $\nu = \frac{1}{n}\sum_{i = 1}^n \delta_{a_i}$, and note that 
$\int_{S_x(\mathcal{U})}\phi_i(x)\,d\nu  = \frac{1}{n}\sum_{i = 1}^n \phi_i(a_i) > 0$, a contradiction.

\end{proof}

We also note that all approximately realized measures are invariant:
\begin{lem}\label{lem_approx_realized_invariant}
    The set of $A$-invariant measures is closed in $\mathfrak{M}_x(\mathcal{U})$, and all measures approximately realized in $A$ are $A$-invariant.
\end{lem}
\begin{proof}
    The set of $A$-invariant measures is 
    $$\bigcap_{\phi(x;y),a\equiv_A b}\left\{\mu : \int_{S_x(\mathcal{U})}\phi(x;a)\,d\mu = \int_{S_x(\mathcal{U})}\phi(x;a)\,d\mu\right\}$$.
    As for each predicate $\phi(x;y)$ and each $a \in \mathcal{U}$, the function $\mu \mapsto \int_{S_x(\mathcal{U})}\phi(x;a)\,d\mu$ is continuous,
    each set in this intersection is closed, so the intersection itself is.

    The set of approximately realized measures is the topological closure of the convex hull of the Dirac measures at types of points in $A$.
    It is clear that the type of a point in $A$ is $A$-invariant, and that a convex combination of $A$-invariant measures is $A$-invariant.
    This is thus the closure of a set of $A$-invariant measures, which must then be contained in the closed set of $A$-invariant measures.
\end{proof}

The choice of model does not matter for defining approximately realized measures, as long as they are invariant.
\begin{lem}\label{lem_approx_realized_base}
    Let $\mu$ be a measure approximately realized/dfs in $A$, and invariant over a small model $M$.
    Then $\mu$ is approximately realized/dfs in $M$.
\end{lem}
\begin{proof}
    The result for dfs will follow from the result for approximate realization, as it holds for definability by Lemma \ref{lem_inv_borel}.
    Now assume $\mu$ is approximately realized in $A$ and $M$-invariant.
    Let $N$ be a measure extending $A \cup M$.
    Then clearly $\mu$ is approximately realized in $N$.
\end{proof}

Approximately realized measures are also closed under Morley products:
\begin{lem}\label{lem_approx_realized_product}
Let $\mu \in \mathfrak{M}_x(\mathcal{U})$ and $\nu \in \mathfrak{M}_y(\mathcal{U})$ be approximately realized in $A$.
Then $\mu \otimes \nu$ is as well.
\end{lem}
\begin{proof}
Let us use the characterization from Lemma \ref{lem_approx_realized}.
It suffices to show that for every predicate $\phi(x;y)$ with parameters from $\mathcal{U}$,
if $\int_{S_{xy}(\mathcal{U})}\phi(x;y)\,d\mu\otimes \nu > 0$, then $\phi(a;b) > 0$ for some $ab \in A^{xy}$.

If $\int_{S_{xy}(\mathcal{U})}\phi(x;y)\,d\mu\otimes \nu = \int_{S_{y}(\mathcal{U})}F^\phi_\mu(y)\,d\nu > 0$, then as $\nu$ is approximately realized, there is some $b \in A^y$ such that $F^\phi_\mu(b) > 0$.
Thus $\int_{S_{x}(\mathcal{U})}\phi(x;b)\,d\mu > 0$, so as $\mu$ is approximately realized, there is some $a \in A^x$ with $\phi(a;b) > 0$.    
\end{proof}

\subsection{Extensions and Orthogonality}
In order to understand extensions of Keisler measures to larger sets of parameters,
let us focus on the positive linear functional perspective, and apply a specialized version of Hahn-Banach.

First, we observe that for any model $M$, the space $\mathcal{C}(S_x(M), \R)$ is an ordered vector space,
with positive cone $C$ consisting of all $f \in \mathcal{C}(S_x(M), \R)$ such that $f(x)\geq 0$ always.
We can view it as an ordered topological vector space by giving it the $\ell_\infty$-norm.
Note that this means the interior points of the positive cone $C$ are exactly the functions $f$ such that $\inf_{p \in S_x(M)} f(p) > 0$.
As the functions in this space have a compact domain $S_x(M)$, these are all the strictly positive functions.

By the following fact, it is clear that every positive linear functional on $\mathcal{C}(S_x(M), \R)$, and thus every Keisler functional, is continuous with respect to the $\ell_\infty$-norm.
\begin{fact}[{\cite[Theorem 5.5]{schaefer_wolff}}]\label{fact_positive_is_continuous}
    Let $E$ be an ordered topological vector space with positive cone $C$, such that $C$ has nonempty interior.
    Then every positive linear form on $E$ is continuous.
\end{fact}

We can combine that conclusion with this fact, guaranteeing continuous positive extensions of continuous positive linear functionals defined on subspaces.
\begin{fact}[{\cite[Corollary 2 of Theorem 5.4]{schaefer_wolff}}]\label{fact_positive_extension}
    Let $E$ be an ordered topological vector space with positive cone $C$, and suppose that $V$ is a vector subspace of $E$
    such that $C \cap V$ contains an interior point of $C$. Then every continuous, positive linear form on $V$ can be extended to $E$, preserving continuity and positivity.
\end{fact}

\begin{cor}\label{cor_extension}
    Let $V$ be a vector subspace of $\mathcal{C}(S_x(M), \R)$, containing the constant functions, and $f : V \to \R$ a positive linear functional with $f(1) = 1$.
    Then $f$ can be extended to a Keisler functional in $\mathfrak{M}_x(M)$.
\end{cor}
\begin{proof}
    We see that $V$ contains an interior point of $C$, namely the constant function 1.
    Thus $V$, as a subspace of $\mathcal{C}(S_x(M), \R)$, is an ordered topological vector space whose positive cone has nonempty interior, so by Fact \ref{fact_positive_is_continuous}, $f$ is continuous.
    Thus also by Fact \ref{fact_positive_extension}, $f$ has an extension to all of $\mathcal{C}(S_x(M), \R)$, which is positive, and is thus a Keisler functional, as $f(1) = 1$.
\end{proof}

\begin{lem}\label{lem_extension}
    Let $V$ be a vector subspace of $\mathcal{C}(S_x(M), \R)$, containing the constant functions, and $f : V \to \R$ a positive linear functional with $f(1) = 1$.
    Let $\phi(x) \in \mathcal{C}(S_x(M), \R)$.
    Then the set of possible values $\hat{f}(\phi)$ where $\hat{f}$ is a Keisler functional extending $f$ is exactly the interval
    $$\left[\sup_{\psi \in V : \psi \leq \phi} f(\psi), \inf_{\psi \in V : \psi \geq \phi} f(\psi)\right].$$  
\end{lem}
\begin{proof}
    First, we note that if $\hat{f}$ is a Keisler functional extension of $f$, then $\hat{f}(\phi)$ must be in that interval,
    because for any $\psi \in V : \psi \leq \phi$, we have $f(\psi) = \hat{f}(\psi) \leq \hat{f}(\phi)$, and similarly for any
    $\psi \in V : \psi \geq \phi$, we have $f(\psi) = \hat{f}(\psi) \geq \hat{f}(\phi)$.

    Now fix $r \in \left[\sup_{\psi \in V : \psi \leq \phi} f(\psi), \inf_{\psi \in V : \psi \geq \phi} f(\psi)\right]$.
    Then we define $f_\phi$ on the vector subspace $V + \R \phi$ by $f_\phi(\theta + a \phi) = f(\theta) + a r$ for all $\theta \in V$ and $a \in \R$.
    This is clearly an extension of $f$ if $\phi \not \in V$, and if $\phi \in V$, then our conditions already guarantee $r = f(\phi)$, so $f_\phi = f$.
    It suffices to show that $f_\phi$ is positive, as Corollary \ref{cor_extension} will then guarantee that $\hat{f}$ extends to a Keisler functional,
    which has $\hat{f}(\phi) = f_\phi(\phi) = r$.

    To show that $f_\phi$ is positive, consider $\theta \in V$ and $a \in \R$ such that $\theta + a \phi$ is always nonnegative.
    Then we must show that $f_\phi(\theta + a \phi) = f(\theta) + a r$ is always nonnegative.
    If $a = 0$, this is guaranteed by the positivity of $f$.
    If $a$ is positive, we need to check that $r \geq  -a^{-1}f(\theta)$.
    This is true because for each $x$, $\theta(x) + a\phi(x) \geq 0$, so $-a^{-1}\theta(x) \leq \phi(x)$.
    Thus $f(-a^{-1}\theta) \leq \sup_{\psi \in V : \psi \leq \phi} f(\psi) \leq r$.
    Similarly, if $a$ is negative, then $f(-a^{-1}\theta) \geq \inf_{\psi \in V : \psi \geq \phi} f(\psi) \geq r$, so $f(\theta) + ar \geq 0$.
\end{proof}

Our first application of Lemma \ref{lem_extension} is extending Keisler measures to larger parameter sets.
\begin{cor}\label{cor_extension_params}
    Let $M \subseteq N$ be models, let $\mu \in \mathfrak{M}_x(M)$, and let $\phi(x)$ be an $N$-definable predicate.
    Then for $r \in [0,1]$, there is a Keisler measure $\nu \in \mathfrak{M}_x(N)$ extending $\mu$ such that $\phi(x) = r$ if and only if
    $$\sup_{\psi : \psi \leq \phi} f(\psi) \leq r \leq \inf_{\psi : \psi \geq \phi} f(\psi)$$
    where the $\sup$ and $\inf$ are over $M$-definable predicates $\psi(x)$.
\end{cor}
\begin{proof}
This follows from Lemma \ref{lem_extension}, applied to the image of $\mathcal{C}(S_x(M), \R)$ within $\mathcal{C}(S_x(N), \R)$.
\end{proof}

Our second will be an application to product measures, which will require generalizing a few more basic definitions to continuous logic.
\begin{defn}\label{defn_ortho}
    Let $x_1,\dots,x_n$ be variable tuples, with $x = (x_1,\dots,x_n)$.
    If $\mu_i \in \mathfrak{M}_{x_i}(M)$ is a family of Keisler measures, then we use the notation
    $\mu_1 \times \dots \times \mu_n$ to denote the partial Keisler ``measure'' (actually a functional) defined by
    $$\int_{S_x(M)}\prod_{i = 1}^n \psi_i(x_i;b)\,d\mu_1 \times \dots \times \mu_n = \prod_{i = 1}^n \int_{S_{x_i}(M)} \psi_i(x_i;b)\,d\mu_i$$
    whenever $\psi_i(x_i;y)$ are formulae (or definable predicates) and $b \in M^y$.

    A measure $\mu \in \mathfrak{M}_x(M)$ is a \emph{product measure} of $\mu_i \in \mathfrak{M}_{x_i}(M)$ when it extends $\mu_1 \times \dots \times \mu_n$.
    The measures $\mu_i \in \mathfrak{M}_{x_i}(M)$ are \emph{weakly orthogonal} when they have a unique product measure.
    If $M = \mathcal{U}$, we say that they are \emph{orthogonal}.
\end{defn}
Note that unlike with types, if $\mu \in \mathfrak{M}_{xy}(M)$ is a measure, and $\mu|_x, \mu|_y$ are the restrictions to the appropriate variables,
then $\mu$ need not be a product measure of $\mu|_x$ and $\mu|_y$.

\begin{cor}\label{cor_ortho}
    Let $x_1,\dots,x_n$ be variable tuples, and let $\mu_i \in \mathfrak{M}_{x_i}(M)$ for each $i$.
    The measures $\mu_i$ are weakly orthogonal if and only if for every $M$-definable predicate $\phi(x_1,\dots,x_n)$ and every $\varepsilon > 0$, there exist $M$-definable predicates $\psi^-(x_1,\dots,x_n),\psi^+(x_1,\dots,x_n)$,
where $\psi^\pm(x_1,\dots,x_n)$ are each of the form $\sum_{j = 1}^m \prod_{i = 1}^n\theta_{ij}^{\pm}(x_i)$, such that
\begin{itemize}
    \item For all $(x_1,\dots,x_n)$, $\psi^-(x_1,\dots,x_n) \leq \phi(x_1,\dots,x_n) \leq \psi^+(x_1,\dots,x_n)$.
    \item For any product measure $\omega$ of $\mu_1,\dots,\mu_n$, $\int_{S_{x_1\dots x_n}(M)}(\psi^+ - \psi^-)\,d\omega \leq \varepsilon$.
\end{itemize}
\end{cor}
\begin{proof}
    First, we assume the measures are weakly orthogonal.
    Let $x = x_1\dots x_n$. Then consider the vector subspace $V$ of $\mathcal{C}(S_x(M), \R)$ spanned by products 
    $\prod_{i = 1}^n \psi_i(x_i)$ where each $\psi_i(x_i) \in \mathcal{C}(S_{x_i}(M), \R)$.
    Define a positive linear functional $f$ on $V$ so that 
    $$f\left(\prod_{i = 1}^n \psi_i(x_i)\right) = \prod_{i = 1}^n \int_{S_{x_i}(M)}\psi_i(x_i)\,d\mu_i.$$
    Any positive linear extension of this to $\mathcal{C}(S_x(M), \R)$ gives rise to a product measure in $\mathfrak{M}_x(M)$, but we know that this is unique.

    Fix $\phi(x)$ and $\varepsilon > 0$. By Lemma \ref{lem_extension} and the uniqueness of the extension of $f$, 
    we know that $\left[\sup_{\psi \in V : \psi \leq \phi} f(\psi) = \inf_{\psi \in V : \psi \geq \phi} f(\psi)\right]$,
    so choose $\psi^-(x), \psi^+(x) \in V$ such that $\psi^-(x)\leq \phi(x) \leq \psi^+(x)$ and $f(\psi^+ - \psi^-)\leq \varepsilon$.
    Then $\int_{S_{x}(M)}(\psi^+ - \psi^-)\,d\omega = f(\psi^+ - \psi^-)\leq \varepsilon$.

    Now assume that the result holds, and we will show the measures are weakly orthogonal.
    If $\omega_1,\omega_2 \in \mathfrak{M}_x(M)$ are measures extending $\mu_1 \times \dots \times \mu_n$, and $\phi(x_1,\dots,x_n)$ is a definable predicate,
    then for every $\varepsilon > 0$, we may find $\psi^{\pm}$ as above.
    We see that $\int_{S_{x_1\dots x_n}(M)}\psi^\pm\,d\omega_1 = \int_{S_{x_1\dots x_n}(M)}\psi^\pm\,d\omega_2$, so the integrals $\int_{S_{x_1\dots x_n}(M)}\phi\,d\omega_j$ for $j = 1,2$ must lie in the interval
    $$\left[\int_{S_{x_1\dots x_n}(M)}\psi^-\,d\omega_1, \int_{S_{x_1\dots x_n}(M)}\psi^+\,d\omega_1\right]$$
    of width at most $\varepsilon$.
    Thus $\int_{S_{x_1\dots x_n}(M)}\phi\,d\omega_1 = \int_{S_{x_1\dots x_n}(M)}\phi\,d\omega_2$, and the measures are weakly orthogonal.
\end{proof}

We will show that weak orthogonality is preserved under localization to a positive-measure Borel set or positive-integral function.
Let $\mu \in \mathfrak{M}_x(M)$ be a Keisler measure, and let $\phi : S_x(M) \to [0,1]$ be Borel with $\int_{S_x(M)}\phi(x)\,d\mu > 0$.
Then the localization $\mu_\phi \in \mathfrak{M}_x(M)$ is the measure given by
$$\int_{S_x(M)}\psi(x)\,d\mu_\phi = \frac{\int_{S_x(M)}\phi(x)\psi(x)\,d\mu}{\int_{S_x(M)}\phi(x)\,d\mu}.$$
If $\phi$ is the characteristic function of a Borel set $X \subseteq S_x(M)$, we may also call the localization $\mu_X$.
\begin{lem}\label{lem_ortho_local}
Let $\mu_i \in \mathfrak{M}_{x_i}(M)$ for $1 \leq i \leq n$ be weakly orthogonal, and $\theta_i: S_{x_i}(M) \to [0,1]$ be Borel with $\int_{S_{x_i}(M)}\theta_i(x_i)\,d\mu_i > 0$.
Then the measures $(\mu_i)_{\theta_i}$ are weakly orthogonal also.
\end{lem}
\begin{proof}
    For readability, let $A_i = \int_{S_{x_i}(M)}\theta_i(x_i)\,d\mu_i$ for $1 \leq i \leq n$, and let $A = \prod_{i = 1}^n A_i$.
    Let $\omega$ be the unique extension of $\mu_1 \times \dots \times \mu_n$.
    If $\theta(x_1,\dots,x_n) = \prod_{i = 1}^n \theta_i(x_i)$, then $\omega_\theta$ extends $(\mu_1)_{\theta_1} \times \dots \times (\mu_n)_{\theta_n}$.
    Suppose that $\nu$ also extends $(\mu_1)_{\theta_1} \times \dots \times (\mu_n)_{\theta_n}$.
    Then we define a measure $\nu'$ on any $M$-definable predicate $\phi(x_1,\dots,x_n)$ by
    \begin{align*}
        &\int_{S_{x_1\dots x_n}(M)}\phi(x_1,\dots,x_n)\,d\nu'  \\
        =& A\int_{S_{x_1\dots x_n}(M)}\phi(x_1,\dots,x_n)\,d\nu
        + \int_{S_{x_1\dots x_n}(M)}\phi(x_1,\dots,x_n)(1 - \theta(x_1,\dots,x_n))\,d\omega.
    \end{align*}
    If for $1 \leq i \leq n$, $\psi_i(x_i)$ is an $M$-definable predicate, then
    \begin{align*}&\int_{S_{x_1\dots x_n}(M)}\prod_{i = 1}^n \psi_i(x_i)\,d\nu'\\
    =& A\int_{S_{x_1\dots x_n}(M)}\prod_{i = 1}^n \psi_i(x_i)\,d\nu'
    + \int_{S_{x_1\dots x_n}(M)}\left(\prod_{i = 1}^n \psi_i(x_i) - \prod_{i = 1}^n \psi_i(x_i)\theta_i(x_i)\right)\,d\omega \\
    =& A\prod_{i = 1}^n \frac{1}{A_i}\int_{S_{x_i}(M)}\psi_i(x_i)\theta_i(x_i)\,d\nu
    + \prod_{i = 1}^n\int_{S_{x_i}(M)} \psi_i(x_i)\,d\mu_i - \prod_{i = 1}^n\int_{S_{x_i}(M)} \psi_i(x_i)\theta_i(x_i)\,d\mu_i \\
    =& \prod_{i = 1}^n \psi_i(x_i)\,d\mu_i,
    \end{align*}
    so $\nu'$ extends $\mu_1 \times \dots \times \mu_n$, and thus equals $\omega$.
    
    Thus for any $\phi$,
    $$A\int_{S_{x_1\dots x_n}(M)}\phi(x_1,\dots,x_n)\,d\nu = \int_{S_{x_1\dots x_n}(M)}\phi(x_1,\dots,x_n)\theta(x_1,\dots,x_n)\,d\omega,$$
    so $\nu = \omega_\theta$, showing the uniqueness of extensions of $(\mu_1)_{\theta_1} \times \dots \times (\mu_n)_{\theta_n}$ and weak orthogonality of the localizations.
\end{proof}

\subsection{Smooth Measures}
In this subsection, we will update to the continuous setting several results about smooth measures that do not require NIP, and then the important result that in an NIP theory, all measures over models have smooth extensions.

The most important characterization of smooth measures is the following lemma, analogous to \cite[Lemma 7.8]{nip_guide}.
This result was already known by James Hanson, but we provide a proof here for completeness. 
\begin{lem}\label{lem_smooth_characterization}
    Let $\mu \in \mathfrak{M}_x(\mathcal{U})$ be a global measure.
    Then $\mu$ is smooth over $M$ if and only if for every definable predicate $\phi(x;y)$ with parameters in $M$, and $\varepsilon > 0$,
    there are open conditions $U_i(y)$ and definable predicates $\psi_i^+(x), \psi_i^-(x)$ with parameters in $M$ for $i = 1,\dots, n$
    such that 
    \begin{itemize}
        \item $U_1(y),\dots,U_n(y)$ cover $S_y(M)$
        \item For all $1 \leq i \leq n$, if $\vDash U_i(b)$, then $\forall x, \psi_i^-(x)\leq \phi(x;b)\leq \psi_i^+(x)$.
        \item For all $1 \leq i \leq n$, $\int_{S_x(\mathcal{U})} \psi_i^+(x) - \psi_i^-(x)\,d\mu<\varepsilon$.
    \end{itemize}
\end{lem}
\begin{proof}
    Assume that $\mu$ satisfies these requirements, and fix an $M$-definable predicate $\phi(x;y)$.
    We will show that for all $b \in \mathcal{U}^y$, the value of $\int_{S_x(\mathcal{U})}\phi(x;b)\,d\nu = \int_{S_x(\mathcal{U})}\phi(x;b)\,d\mu$ is determined for all global measures $\nu$ extending $\mu|_M$.
    Specifically, for every $\varepsilon > 0$, we show that $|\int_{S_x(\mathcal{U})}\phi(x;b)\,d\nu - \int_{S_x(\mathcal{U})}\phi(x;b)\,d\mu| \leq \varepsilon$.
    There must be some $i$ such that $\vDash U_i(b)$. Thus $\psi_i^-(x)\leq \phi(x;b)\leq \psi_i^+(x)$, so
    $$\int_{S_x(M)}\psi_i^-(x)\,d\mu|_M \leq \int_{S_x(\mathcal{U})}\phi(x;b)\,d\nu \leq \int_{S_x(M)}\psi_i^+(x)\,d\mu|_M,$$
    and $\int_{S_x(\mathcal{U})}\phi(x;b)\,d\mu$ lies in that same interval of width at most $\varepsilon$. Thus their difference is at most $\varepsilon$.

    Now assume that $\mu$ is smooth over $M$.
    For every $b \in \mathcal{U}^y$, we will find $\psi_b^-, \psi_b^+$ such that $\forall x, \psi_b^-(x)\leq \phi(x;b)\leq \psi_b^+(x)$,
    and $\int_{S_x(\mathcal{U})} \psi_b^+(x) - \psi_b^-\,d\mu<\varepsilon$. We will then apply compactness.

    By smoothness, for any global extension $\nu$ of $\mu|_M$, the integral $\int_{S_x(\mathcal{U})}\phi(x;b)\,d\nu$ has the same value, 
    namely $\int_{S_x(\mathcal{U})}\phi(x;b)\,d\nu = \int_{S_x(\mathcal{U})}\phi(x;b)\,d\mu$.
    Thus by Corollary \ref{cor_extension_params}, we must have 
    $$\sup_{\psi : \psi \leq \phi} \int_{S_x(M)}\psi(x)\,d\mu|_M =  \int_{S_x(\mathcal{U})}\phi(x;b)\,d\mu = \inf_{\psi : \psi \geq \phi} \int_{S_x(M)}\psi(x)\,d\mu|_M,$$
    where as above, $\psi(x)$ ranges over $M$-definable predicates.
    Thus there exist $\psi_b^-(x)$ and $\psi_b^+(x)$ with $\psi_b^-(x) \leq \phi(x;b) \leq \psi_b^+(x)$ and $\int_{S_x(\mathcal{U})}\phi(x;b)\,d\mu - \int_{S_x(M)}\psi_b^-(x)\,d\mu < \frac{\varepsilon}{2}$ and $\int_{S_x(M)}\psi_b^+(x)\,d\mu - \int_{S_x(\mathcal{U})}\phi(x;b)\,d\mu < \frac{\varepsilon}{2}$.
    Thus $\int_{S_x(M)}\psi_b^+(x)\,d\mu - \int_{S_x(M)}\psi_b^-(x)\,d\mu < \varepsilon$ as desired.

    For the compactness argument, it will be simpler to assume that $\phi(x;y)$ only takes values in the open interval $(0,1)$. We will prove the result for those $\phi$ first, and then make a correction for all $\phi$.
    For any $b \in \mathcal{U}^y$, by adding a very small amount to $\psi_b^+(x)$ and $\psi_b^-(x)$, we can guarantee that at every point $\psi_b^-(x) < \phi(x;b) < \psi_b^+(x)$,
    while still ensuring that $\int_{S_x(M)}\psi_b^+(x)\,d\mu - \int_{S_x(M)}\psi_b^-(x)\,d\mu < \varepsilon$.
    By the compactness of $S_x(M)$, there is some $\delta > 0$ such that $\inf_{x \in S_x(M)} \psi_b^+(x) - \phi(x;b) > \delta$ and $\inf_{x \in S_x(M)} \phi(x;b) - \psi_b^-(x) > \delta$.
    Thus there is an open subset $U_b(y)$ of $S_y(M)$, containing the type of $b$, defined by $\inf_{x \in S_x(M)} \psi_b^+(x) - \phi(x;y) > \delta$ and $\inf_{x \in S_x(M)} \phi(x;y) - \psi_b^-(x) > \delta$, such that $\psi_b^-(x) < \phi(x;y) < \psi_b^+(x)$ at every point satisfying $U_b(y)$.
    Let $U_1(y),\dots,U_n(y)$ be a finite subcover of these, with $U_i(y) = U_{b_i}(y)$.
    Then taking $\psi_i^\pm(x) = \psi_{b_i}^\pm(x)$, we have the desired result.

    Now for the correction.
    If $\phi(x;y)$ is any $M$-definable predicate, possibly taking the values $0$ or $1$, we apply the result to the predicate
    $\phi(x;y)' = \frac{1}{2}\phi(x;y) + \frac{1}{4}$, whose range is bounded to $\left[\frac{1}{4},\frac{3}{4}\right]$, finding open conditions $U_i(y)$ and definable predicates $\psi_i^{+'}(x), \psi_i^{-'}(x)$ with parameters in $M$ for $i = 1,\dots, n$
    such that 
    \begin{itemize}
        \item $U_1(y),\dots,U_n(y)$ cover $S_y(M)$
        \item For all $1 \leq i \leq n$, if $\vDash U_i(b)$, then $\forall x, \psi_i^{-'}(x)\leq \phi'(x;b)\leq \psi_i^{+'}(x)$.
        \item For all $1 \leq i \leq n$, $\int_{S_x(\mathcal{U})} \psi_i^{+'}(x) - \psi_i^{-'}\,d\mu<\frac{\varepsilon}{2}$.
    \end{itemize}
    By taking $\psi_i^\pm(x) = \min(\max(2\psi_i^{\pm'}(x) - \frac{1}{2},0),1)$,
    we find that for all $1 \leq i \leq n$, if $\vDash U_i(b)$, then $\forall x, \psi_i^{-}(x)\leq \phi(x;b)\leq \psi_i^{+}(x)$,
    and $\int_{S_x(\mathcal{U})} \psi_i^{+}(x) - \psi_i^{-}\,d\mu<\varepsilon$.    

\end{proof}

In the analogous characterization from discrete logic, the conditions $U_1(y),\dots, U_n(y)$ can be chosen to be disjoint.
By taking Boolean combinations, we can enforce disjointness and end up with a Borel partition.
In other applications, it will be more convenient to replace the open cover with a partition of unity using \cite[Fact 3.5]{anderson1}.
Using that fact, we can choose the open conditions $U_i(y)$ to each be the support of some
$u_i(y)$, with $\forall y, u_1(y) + \dots + u_n(y) = 1$.

This allows us to characterize smooth measures in terms of weak orthogonality.
\begin{lem}\label{lem_smooth_ortho}
    Let $\mu \in \mathfrak{M}_x(\mathcal{U})$ be a global measure, and let $M \subset \mathcal{U}$ be small. The following are equivalent:
    \begin{itemize}
        \item $\mu$ is smooth over $M$
        \item $\mu|_M$ is weakly orthogonal to all types $p(y) \in S_y(M)$
        \item $\mu|_M$ is weakly orthogonal to all measures $\nu(y) \in \mathfrak{M}_y(M)$.
    \end{itemize}
\end{lem}
\begin{proof}
    Clearly if $\mu|_M$ is weakly orthogonal to all measures over $M$, then it is weakly orthogonal to all types over $M$.
    
    Assume that $\mu|_M$ is weakly orthogonal to all types $p(y) \in S_y(M)$.
    Then fix a formula $\phi(x;y)$, and $b \in \mathcal{U}^y$.
    As $\mu|_M$ is weakly orthogonal to $\mathrm{tp}(b/M)$, there is a unique measure $\omega \in \mathfrak{M}_{xy}(M)$ extending $\mu|_M \times \mathrm{tp}(b/M)$,
    so for any $\mu'$ extending $\mu|_M$ and $c \vDash \mathrm{tp}(b/M)$,
    the measure $\lambda\in \mathfrak{M}_x(M)$ given by $\int_{S_{xy}(M)}\phi(x;y)\,d\lambda = \int_{S_x(\mathcal{U})}\phi(x;b)\,d\mu'$ equals $\omega$, so $\int_{S_x(\mathcal{U})}\phi(x;b)\,d\mu'$ is uniquely determined, from which we can conclude that $\mu$ is smooth over $M$.
    
    Now assume that $\mu$ is smooth over $M$, and let $\nu \in \mathfrak{M}_y(M)$ be another measure. We will show that $\mu|_M$ is weakly orthogonal to $\nu$.
    Let $\lambda \in \mathfrak{M}_{xy}(M)$ extend $\mu|_M \times \nu$ and fix an $M$-definable predicate $\phi(x;y)$ and $\varepsilon > 0$.
    Let $u_i(y),\psi_i^+(x),\psi_i^-(x)$ for $1 \leq i \leq n$ be definable predicates as given by Lemma \ref{lem_smooth_characterization},
    with the $u_i$s forming a partition of unity as in the remarks following that lemma.
    Then we see that $\forall x, \forall y, \sum_{i = 1}^n \psi_i^-(x)u_i(y) \leq \phi(x;y) \leq \sum_{i = 1}^n \psi_i^+(x)u_i(y)$.
    These bounds together with the separated amalgam property tell us that
    $$\int_{S_{xy}(\mathcal{U})}\phi(x;y)\,d\lambda \leq 
    \sum_{i = 1}^n\int_{S_{xy}(\mathcal{U})} \psi_i^+(x)u_i(y)\,d\lambda
    = \sum_{i = 1}^n\int_{S_{x}(\mathcal{U})} \psi_i^+(x)\,d\mu|_M \int_{S_{y}(\mathcal{U})}u_i(y)\,d\nu,$$
    and that
    $$\sum_{i = 1}^n\int_{S_{x}(\mathcal{U})} \psi_i^-(x)\,d\mu|_M \int_{S_{y}(\mathcal{U})}u_i(y)\,d\nu \leq \int_{S_{xy}(\mathcal{U})}\phi(x;y)\,d\lambda.$$
    This places $\int_{S_{xy}(\mathcal{U})}\phi(x;y)\,d\lambda$ in an interval not depending on $\lambda$, of width
    $$\sum_{i = 1}^n\int_{S_{x}(\mathcal{U})} \psi_i^+(x)\,d\mu|_M \int_{S_{y}(\mathcal{U})}u_i(y)\,d\nu - \sum_{i = 1}^n\int_{S_{x}(\mathcal{U})} \psi_i^-(x)\,d\mu|_M\int_{S_{y}(\mathcal{U})}u_i(y)\,d\nu \leq \varepsilon,$$
    which determines $\int_{S_{xy}(\mathcal{U})}\phi(x;y)\,d\lambda$ uniquely as $\varepsilon$ was arbitrary.
\end{proof}

\begin{cor}[{Generalizes \cite[Corollary 2.5]{HPS}}]\label{cor_smooth_commute}
    Let $\mu \in \mathfrak{M}_x(\mathcal{U})$ be smooth over a model $M$, and let $\nu \in \mathfrak{M}_y(\mathcal{U})$ be Borel-definable over $M$.
    Then $\mu \otimes \nu = \nu \otimes \mu$.
    \end{cor}
\begin{proof}
    Both $\mu \otimes \nu$ and $\nu \otimes \mu$ are separated amalgams of $\mu$ and $\nu$, thus by Lemma \ref{lem_smooth_ortho}, they are equal.
\end{proof}

Smooth measures are also preserved under localization to a positive-measure Borel set or positive-integral function.
\begin{cor}\label{cor_smooth_local}
    Let $\mu\in \mathfrak{M}_x(\mathcal{U})$ be a smooth measure and let $\theta: S_{x}(M) \to [0,1]$ be Borel with $\int_{S_{x}(M)}\theta(x)\,d\mu > 0$.
Then the measures $\mu_{\theta}$ is weakly orthogonal also.
\end{cor}
\begin{proof}
    By Lemma \ref{lem_smooth_ortho}, $\mu$ is smooth if and only if it is weakly orthogonal to all global types, and by Lemma \ref{lem_ortho_local}, $\mu_\theta$ is also weakly orthogonal to all global types.
\end{proof}

\begin{lem}[{Generalizes \cite[Lemma 7.17(i)]{nip_guide}}]\label{lem_smooth_dfs}
    Let $\mu \in \mathfrak{M}_x(\mathcal{U})$ be a global measure smooth over $M$.
    Then $\mu$ is dfs over $M$.
\end{lem}
\begin{proof}
First we show that $\mu$ is approximately realized in $M$.
By Lemma \ref{lem_approx_realized}, it suffices to show that for any $\phi(x;b)$, if $\int_{S_x(\mathcal{U})} \phi(x;b)\,d\mu > 0$, then for some $a \in M^x$, $\phi(a;b) > 0$.
Assume $\int_{S_x(\mathcal{U})} \phi(x;b)\,d\mu > \varepsilon > 0$.
Then by Lemma \ref{lem_smooth_characterization}, there are $\psi^-(x), \psi^+(x)$ such that $\forall x, \psi^-(x) \leq \phi(x;b)\leq \psi^+(x)$, and
$\int_{S_x(\mathcal{U})}\psi_i^+(x) - \psi_i^-\,d\mu<\varepsilon$.
Thus
$$\varepsilon < \int_{S_x(\mathcal{U})}\phi(x;b)\,d\mu \leq \int_{S_x(\mathcal{U})}\psi_i^+\,d\mu < \int_{S_x(\mathcal{U})}\psi_i^-\,d\mu + \varepsilon,$$
so $\psi^-(x)$ must take a positive value at some $a \in \mathcal{U}^x$.
By elementary equivalence, it must also take a positive value at some $a' \in \mathcal{U}^x$, where we have
$0 < \psi^-(a') \leq \phi(a';b)$.

Now we show that $\mu$ is definable over $M$.
Specifically, we fix $\phi(x;y)$, and wish to show that $F^\phi_{\mu,A}$ is continuous, by showing that if $r < s$, the set $\{p \in S_y(A): r < F^\phi_{\mu, A}(p) < s\}$ is open.
Let $p \in S_y(A)$ be such that $r < F^\phi(\mu,A)(p) < s$, and let $b \vDash p$.
Fix $0 < \varepsilon < \min(s - F^\phi(\mu,A)(p), F^\phi(\mu,A)(p) - r)$.
By Lemma \ref{lem_smooth_characterization}, there is an open condition $U(y)$ such that $\vDash U(b)$, and $\psi^-(x), \psi^+(x)$ such that $\int_{S_x(A)}\psi^+(x) - \psi^-(x)\,d\mu < \varepsilon$ and
for all $b'$ with $\vDash U(b')$, $\forall x, \psi^i(x)\leq \phi(x;b')\leq \psi^+(x)$.
We will show that for all $q \in S_y(A)$ in the open neighborhood defined by $U(y)$, $r < F^\phi(\mu,A)(q) < s$.

We have $F^\phi(\mu,A)(p) = \int_{S_x(\mathcal{U})}\phi(x;b)\,d\mu$, so 
$$r < F^\phi(\mu,A)(p) - \varepsilon \leq \int_{S_x(\mathcal{U})}\psi^-(x)\,d\mu \leq F^\phi(\mu,A)(p)\leq \int_{S_x(\mathcal{U})}\phi^+(x)\,d\mu \leq F^\phi(\mu,A)(p) + \varepsilon < s.$$
Now let $q \in S_y(A)$ in the open neighborhood defined by $U(y)$, and let $b' \vDash q$. Then $\vDash U(b')$, and thus
$$r < \int_{S_x(\mathcal{U})}\psi^-(x)\,d\mu \leq \int_{S_x(\mathcal{U})}\phi(x;b')\,d\mu \leq \int_{S_x(\mathcal{U})}\phi^+(x)\,d\mu < s,$$
so $F^\phi(\mu,A)(q) = \int_{S_x(\mathcal{U})}\phi(x;b')\,d\mu$ is in $(r,s)$.
\end{proof}

Now we note that the choice of small model is not critical when defining smoothness.
\begin{lem}\label{lem_inv_smooth}[{Generalizing \cite[Lemma 7.17]{nip_guide}}]
    Let $\mu \in \mathfrak{M}_x(\mathcal{U})$ be smooth over $M$, and invariant over another small model $N$.
    Then $\mu$ is smooth over $N$ also.
\end{lem}
\begin{proof}
    We may assume that $N \preceq M$, as otherwise we may replace $M$ with an elementary extension of $N$ containing $M \cup N$.

    Fix $\varepsilon > 0$, an $N$-definable predicate $\phi(x;y)$, and $p \in S_y(M)$.
    Then by smoothness, $\mu|_M$ and $p$ are weakly orthogonal, so there exist definable predicates $\psi^-(x,y,z),\psi^+(x,y,z)$,
    where $\psi^\pm(x,y,z)$ are each of the form $\sum_{j = 1}^m \theta_{1j}^{\pm}(x)\theta_{2j}^{\pm}(y)\theta_{3j}^{\pm}(z)$,
    and some $c \in M^z$, such that
    for all $a,b \in M^z$, $\psi^-(a,b,c) \leq \phi(a,b) \leq \psi^+(a,b,c)$
    and for all $b$ satisfying $p$,
    $\int_{S_{x}(M)}(\psi^+(x;b) - \psi^-(x;b))\,d\mu < \varepsilon$.

    Then $c$ satisfies $\inf_{x,y} (\psi^-(x,y,z) \dot{-} \phi(x,y)) < \varepsilon$ and 
    $\inf_{x,y} (\phi(x,y) \dot{-} \psi^+(x,y,z)) < \varepsilon$, both open conditions in $z$ which have parameters only in $N$.

    As $\mu$ is definable over $M$ and invariant over $N$, it is also definable over $N$,
    so it is also an open condition with parameters in $N$ that
    $F^{\psi^+}_{\mu,N}(p|_N,z) - F^{\psi^-}_{\mu,N}(p|_N,z) < \varepsilon$, and this open condition also applies to $c$.
    
    The conjunction of all of these open conditions is an open condition, so as such a $c$ realizes it in $M$, there is some $c'$ realizing it in $N$.
    Thus letting $\chi^-(x,y,c') = \psi^-(x,y,c') - \varepsilon$ and $\chi^+(x,y,c') = \psi^+(x,y,c') + \varepsilon$,
    we see that for all $a \in N^z$, $\chi^-(a,b,c') \leq \phi(a,b) \leq \chi^+(a,b,c')$,
    while $F^{\chi^+}_{\mu,N}(p|_N,c') - F^{\chi^-}_{\mu,N}(p|_N,c') < 3\varepsilon$,
    showing that $\mu|_N$ and $p|_N$ are weakly orthogonal,
    so by the generality of $p$, $\mu$ is smooth over $N$.
\end{proof}

\begin{lem}[{Generalizes \cite[Corollary 1.3]{associativity}}]\label{lem_smooth_product}
    Let $\mu \in \mathfrak{M}_x(\mathcal{U})$ and $\nu \in \mathfrak{M}_y(\mathcal{U})$ be smooth over a model $M$,
    Then $\mu \otimes \nu$ is smooth over $M$.
\end{lem}
\begin{proof}
    It suffices to show that if $\lambda |_M = (\mu \otimes \nu)|_M$, then $\lambda = \mu \otimes \nu$.
    As $(\mu \otimes \nu)|_M$ is a separated amalgam of $\mu, \nu$ and $\mu$ and $\nu$ are smooth, by Lemma \ref{lem_smooth_ortho} it suffices to show that $\lambda$ is as well.

    Let $\phi(x),\psi(y)$ be formulas (with parameters), and fix $\varepsilon > 0$.
    We will show that 
    $$\left|\int_{S_{xy}(\mathcal{U})}\phi(x)\psi(y)\,d\lambda - \int_{S_{x}(\mathcal{U})}\phi(x)\,d\mu\int_{S_{y}(\mathcal{U})}\psi(y)\,d\nu\right| < \left(\int_{S_{x}(\mathcal{U})}\phi(x)\,d\mu + \int_{S_{y}(\mathcal{U})}\psi(y)\,d\nu\right)\varepsilon + \varepsilon^2,$$
    which will show that $\lambda$ is a separated amalgam, as $\varepsilon$ is arbitrary.
    By Lemma \ref{lem_smooth_characterization}, there are formulas $\theta^-(x),\theta^+(x), \chi^-(y),\chi^+(y)$ with parameters from $M$ such that
    \begin{itemize}
        \item $\forall x, \theta^-(x) \leq \phi(x) \leq \theta^+(x)$
        \item $\forall y, \chi^-(y) \leq \psi(y) \leq \chi^+(y)$
        \item $\int_{S_x(\mathcal{U})}\theta^+(x) - \theta^-(x)\,d\mu < \varepsilon$
        \item $\int_{S_y(\mathcal{U})}\chi^+(y) - \chi^-(y)\,d\nu < \varepsilon$.
    \end{itemize}
    We will explicitly prove the upper bound on $\int_{S_{xy}(\mathcal{U})}\phi(x)\psi(y)\,d\lambda$, the lower bound will follow by the same logic.
    Using the fact that $\lambda |_M = (\mu \otimes \nu)|_M$, we see that
\begin{align*}
    \int_{S_{xy}(\mathcal{U})}\phi(x)\psi(y)\,d\lambda
    \leq& \int_{S_{xy}(\mathcal{U})}\theta^+(x)\chi^+(y)\,d\lambda\\
    =& \int_{S_{xy}(\mathcal{U})}\theta^+(x)\chi^+(y)\,d(\mu \otimes \nu)\\
    =& \int_{S_{x}(\mathcal{U})}\theta^+(x)\,d\mu\int_{S_{y}(\mathcal{U})}\chi^+(y)\,d\nu.
\end{align*}
We now note that $\int_{S_{x}(\mathcal{U})}\theta^+(x)\,d\mu < \int_{S_{x}(\mathcal{U})}\phi(x)\,d\mu + \varepsilon$
and $\int_{S_{y}(\mathcal{U})}\chi^+(y)\,d\nu < \int_{S_{y}(\mathcal{U})}\psi(y)\,d\nu + \varepsilon$,
and these inequalities give us the desired upper bound.
\end{proof}

We now assume NIP, and see that every Keisler measure over a model admits a smooth extension.

\begin{lem}[{Generalizes \cite[Prop 7.9]{nip_guide}}]\label{lem_smooth_extensions}
    Every Keisler measure $\mu \in \mathfrak{M}_x(M)$ over a small model $M$ admits a smooth extension over some $M \preceq N$.
\end{lem}
\begin{proof}
Assume for contradiction that $\mu$ has no smooth extensions.
Then we inductively build a chain of extensions of measures indexed by the ordinal $|\mathcal{L}|^+$.

That is, we will construct $((M_\alpha,\mu_\alpha): \alpha < |\mathcal{L}|^+)$, with $(M_0,\mu_0) = (M,\mu)$
and for each $\alpha < \beta$, $M_\alpha \subseteq M_\beta$ and $\mu_\alpha = \mu_\beta|_{M_\alpha}$.
At limit stages, we can take a union of the models and the measures, so we can just define the successor steps.
Let $(M_\alpha,\mu_\alpha)$ be defined.
As $\mu_\alpha$ extends $\mu$, it is not smooth, so let $\mu^+,\mu^-$ be two distinct global extensions of $\mu_\alpha$.
As they are distinct, there is some formula $\phi_\alpha(x;b_\alpha)$ with $\phi_\alpha\in \mathcal{F}_{xy}$ such that
$\int_{S_x(\mathcal{U})}\phi_\alpha(x;b_\alpha)\,d\mu^+ \geq \int_{S_x(\mathcal{U})}\phi_\alpha(x;b_\alpha)\,d\mu^- + \varepsilon_\alpha$ for some $\varepsilon_\alpha > 0$.
We let $M_{\alpha + 1}$ be a model containing $M_\alpha$ and $b_\alpha$, and let $\mu_{\alpha + 1} = \left(\frac{1}{2}\left(\mu^+ - \mu^-\right)\right)|_{M_{\alpha + 1}}$.
We then see that for any $\theta(x)$ with parameters in $M_\alpha$, as $\int_{S_x(\mathcal{U})}\theta(x)\,d\mu^+ = \int_{S_x(\mathcal{U})}\theta(x)\,d\mu^-$,
\begin{align*}
    &\int_{S_x(\mathcal{U})}\left|\theta(x) - \phi_\alpha(x;b_\alpha)\right|\,d\mu^+ + \int_{S_x(\mathcal{U})}\left|\theta(x) - \phi_\alpha(x;b_\alpha)\right|\,d\mu^-\\
    \geq &\left|\int_{S_x(\mathcal{U})}\theta(x) - \phi_\alpha(x;b_\alpha)\,d\mu^+\right| + \left|\int_{S_x(\mathcal{U})}\theta(x) - \phi_\alpha(x;b_\alpha)\,d\mu^-\right|\\
    \geq &\left|\int_{S_x(\mathcal{U})}\phi_\alpha(x;b_\alpha)\,d\mu^+ - \int_{S_x(\mathcal{U})}\phi_\alpha(x;b_\alpha)\,d\mu^-\right|\\
    \geq &\varepsilon_\alpha.
\end{align*}
Thus either $\int_{S_x(\mathcal{U})}\left|\theta(x) - \phi_\alpha(x;b_\alpha)\right|\,d\mu^+\geq \frac{\varepsilon_\alpha}{2}$ or $\int_{S_x(\mathcal{U})}\left|\theta(x) - \phi_\alpha(x;b_\alpha)\right|\,d\mu^-\geq \frac{\varepsilon_\alpha}{2}$,
so 
$$\int_{S_x(M_{\alpha + 1})}\left|\theta(x) - \phi_\alpha(x;b_\alpha)\right|\,d\mu_{\alpha + 1} \geq \frac{\varepsilon_\alpha}{4}.$$

We may assume that each $\varepsilon_\alpha$ is rational.
Then by an infinite pigeonhole principle, and the fact that there are at most $|\mathcal{L}|$ choices of $(\phi_\alpha, \varepsilon_\alpha)$,
we can restrict to a subsequence of the same length such that $\phi_\alpha$ and $\varepsilon_\alpha$ are constant.
We call these constant values simply $\phi$ and $\varepsilon$.
Thus if we let $M'$ be the union of all the models and $\mu'$ be the union of all the measures in our new sequence, have an infinite sequence $(b_\alpha : \alpha < |\mathcal{L}|^+)$ such that for all $\alpha < \beta$,
$$\int_{S_x(M')}|\phi(x;b_\alpha) - \phi(x;b_\beta)|\,d\mu' \geq \varepsilon.$$

Now using the same Ramsey and compactness argument as in the proof of Lemma \ref{lem_measure_consistent},
we extract an $\mathcal{L}_{\mathbb{E}}$-indiscernible sequence with the same EM-type as $(b_\alpha : \alpha < |\mathcal{L}|^+)$.
Call this indiscernible $(b_i' : i < \omega)$, and let $(M^*;\mu^*)$ be an elementary extension of $(M';\mu')$ containing it.

Then in particular, for all $i < \omega$,
$$\int_{S_x(M^*)}|\phi(x;b_{2i}') - \phi(x;b_{2i + 1}')|\,d\mu^* \geq \varepsilon.$$
so by Lemma \ref{lem_measure_consistent},
the partial type $\{|\phi(x;b_{2i}') - \phi(x;b_{2i + 1}')| \geq \frac{\varepsilon}{2} : i <\omega\}$ is consistent.
As $(b_i' : i < \omega)$ is indiscernible, this contradicts NIP.
\end{proof}

We can now use smooth extensions to prove that the Morley product is associative in an NIP context.

\begin{lem}\label{lem_associative}
    Assume $T$ is NIP. Let $\mu \in \mathfrak{M}_x(\mathcal{U}), \nu \in \mathfrak{M}_y(\mathcal{U}), \lambda \in \mathfrak{M}_z(\mathcal{U})$ be $M$-Borel definable, with $\mu \otimes \nu$ also $M$-Borel definable. 
    Then $(\mu \otimes \nu) \otimes \lambda = \mu \otimes (\nu \otimes \lambda)$. 
\end{lem}
\begin{proof}
The proof in \cite[Section 3.1]{associativity} suffices, as we have proven all of the ingredients of that proof still hold in the case of metric structures.
Specifically, it only requires the following tools,
\begin{itemize}
    \item associativity of the Morley product for smooth (or just definable) measures (Lemma \ref{lem_definable_product})
    \item the existence of smooth extensions (Lemma \ref{lem_smooth_extensions})
    \item the fact that the Morley product of smooth measures is smooth (Lemma \ref{lem_smooth_product})
    \item the fact that if $\mu$ is $A$-invariant, $(\mu \otimes \nu)|_A$ depends only on $\nu|_A$
\end{itemize}
all of which we have established in continuous logic.
\end{proof}

\section{Generically Stable Measures}\label{sec_gs}
In this section, we will obtain a continuous version of \cite[Theorem 3.2]{HPS}, which characterize generically stable measures in NIP theories.
We will prove the following properties are equivalent, and we will call any measure satisfying them \emph{generically stable}.
We then show how to find generically stable measures in metric structures using ultraproducts or averaging indiscernible segments.

\begin{thm}\label{thm_gen_stable}
    Assume $T$ is NIP.
    For any small model $M \subseteq \mathcal{U}$, if $\mu$ is a global $M$-invariant measure, the following are equivalent:
    \begin{enumerate}[(i)]
        \item $\mu$ is fim over $M$
        \item $\mu$ is fam over $M$
        \item $\mu$ is dfs over $M$
        \item $\mu(x) \otimes \mu(y) = \mu(y) \otimes \mu(x)$
        \item $\mu^{(\omega)}(x_0,x_1,\dots)|_{M}$ is totally indiscernible.
    \end{enumerate}
\end{thm}
Once we have this equivalence, we can see that by Lemma \ref{lem_smooth_dfs}, smooth measures are generically stable in NIP.

Several of these implications follow without the NIP assumption:
\begin{lem}
    For any small model $M \subseteq \mathcal{U}$, if $\mu$ is a global $M$-invariant measure, then each property implies the next:
    \begin{enumerate}[(i)]
        \item $\mu$ is fim over $M$
        \item $\mu$ is fam over $M$
        \item $\mu$ is dfs over $M$
    \end{enumerate}
\end{lem}
\begin{proof}
    (i) $\implies$ (ii) follows by definition.

    (ii) $\implies$ (iii): assume $\mu$ is fam over $M$. First we check that $\mu$ is approximately realized in $M$, which will imply that $\mu$ is $M$-invariant. Fix $\varepsilon > 0$ and $\phi(x;b)$ such that $\int \phi(x)\,d\mu < \varepsilon$.
    As $\mu$ is fam, there exists some $\frac{1}{2}(\varepsilon - \int \phi(x;b)\,d\mu)$-approximation $(a_1,\dots,a_n) \in (M^x)^n$ for $\phi(x;b)$ with respect to $\mu$.
    Thus $|\mathrm{Av}(a_1,\dots,a_n;\phi(x;b)) - \int \phi(x;b)\,d\mu)| \leq \frac{1}{2}(\varepsilon - \int \phi(x)\,d\mu)$, from which we conclude that $\mathrm{Av}(a_1,\dots,a_n;\phi(x;b)) < \varepsilon$.
    This means that for at least one $a_i$, $\vDash \phi(a_i;b)<\varepsilon$.
    
    Now we check definability. Fix $\varphi(x;y) \in \mathcal{L}$. Then for each $\varepsilon > 0$, there is a tuple $\bar a$ such that $|F^\phi_{\mu,M}(y) - F^\phi_{\mathrm{Av}(\bar a),M}(y)| < \varepsilon$ for all $y$. Thus if $(\bar a_n : n \in \omega)$ is a sequence of tuples with $|F^\phi_{\mu,M}(y) - F^\phi_{\mathrm{Av}(\bar a_n),M}(y)| < 2^{-n}$, then $\lim_{n \to \infty} F^\phi_{\mathrm{Av}(\bar a_n),M}(y) = F^\phi_{\mu,M}(y)$ is a uniform limit of continuous functions, which is thus continuous.    
\end{proof}

For the rest of this subsection, we will assume $T$ is NIP.

The following lemma shows that if $\mu$ is dfs over $M$, then $\mu$ commutes with itself.
\begin{lem}[{Generalizing \cite[Prop. 2.26]{nip_guide}}]\label{lem_dfs_commute}
Let $\mu \in \mathfrak{M}_x(\mathcal{U}), \nu \in \mathfrak{M}_y(\mathcal{U})$, with $\mu$ $M$-definable and $\nu$ approximately realized in $M$.
Then $\mu \otimes \nu = \nu \otimes \mu$.
\end{lem}
\begin{proof}
We take the general approach to this from \cite{cgh}.

First, we assume that $\nu$ is the Dirac measure of a type realized by some tuple $b$ in $M$.
Then we see that for any $\phi(x;y)$, $F^{\phi^*}_\nu(x) = \phi(x;b)$, so
\begin{align*}
    \int_{S_{xy}(\mathcal{U})}\phi(x;y)\,d\mu\otimes \nu
= &\int_{S_y(\mathcal{U})}F^\phi_\mu(y)\,d\nu\\
= &F^\phi_\mu(b)\\
= &\int_{S_x(\mathcal{U})}\phi(x;b)\,d\mu\\
= &\int_{S_{xy}(\mathcal{U})}\phi(x;y)\,d\nu\otimes \mu
\end{align*}
Then we see that if we take a convex combination of measures that commute with $\mu$, they will also commute with $\nu$, as convex combinations commute with integration.
Thus $\nu$ is the limit of a net of measures that commute with $\mu$, call these $(\nu_i : i \in I)$ for a directed set $I$.

Now let $\phi(x;y)$ be a formula with parameters. By enlarging $M$ if necessary, we may assume that $M$ contains all the parameters of $\phi$, as $\mu$ will still be $M$-definable and $\nu$ still approximately realized in $M$.
By Lemma \ref{lem_smooth_extensions}, let $\hat{\mu}$ be a global smooth extension of $\mu|_M$.
We now see that
\begin{align*}
    \int_{S_{xy}(\mathcal{U})}\phi(x;y)\,d\mu\otimes \nu
    =& \lim_{i \in I} \int_{S_{xy}(\mathcal{U})}\phi(x;y)\,d\mu\otimes \nu_i \\
    =& \lim_{i \in I} \int_{S_{xy}(\mathcal{U})}\phi(x;y)\,d\nu_i \otimes \mu\\
    =& \lim_{i \in I} \int_{S_{xy}(\mathcal{U})}\phi(x;y)\,d\nu_i \otimes \hat{\mu}\\
    =& \lim_{i \in I} \int_{S_{xy}(\mathcal{U})}\phi(x;y)\,d\hat{\mu} \otimes \nu_i \\
    =& \int_{S_{xy}(\mathcal{U})}\phi(x;y)\,d\hat{\mu} \otimes \nu \\
    =& \int_{S_{xy}(\mathcal{U})}\phi(x;y)\,d\nu\otimes \hat{\mu} \\
    =& \int_{S_{xy}(\mathcal{U})}\phi(x;y)\,d\nu\otimes \mu.
\end{align*}
The first and fifth equations are justified by the fact that for any definable $\mu'$,
the map $\nu' \mapsto \int_{S_{xy}(\mathcal{U})}\phi(x;y)\,d\mu'\otimes \nu'$ is continuous.
This is true as $\int_{S_{xy}(\mathcal{U})}\phi(x;y)\,d\mu'\otimes \nu' = \int_{S_y(\mathcal{U})}F^\phi_{\mu'}\,d\nu'$,
and by the definition of the topology of $\mathfrak{M}_x(\mathcal{U})$, the integral of a continuous predicate is continuous as a function on the space of measures.

The third and seventh equations are justified by the fact that for any $A$-invariant $\mu, \nu'$, $(\nu' \otimes \mu')|_A$ only depends on $\mu'|_A$.
The second follows from our observation that the $\nu_i$s commute with $\mu$, and the fourth and sixth follow from Corollary \ref{cor_smooth_commute}.
\end{proof}

The total indiscernibility of $\mu^{(\omega)}(x_0,x_1,x_2,\dots)$ follows from the associativity of the Morley product in NIP combined with Lemma \ref{lem_dfs_commute},

We now work towards showing that the indiscernibility of $\mu^{(\omega)}(x_0,x_1,x_2,\dots)$ implies fim, following \cite[Theorem 3.2]{HPS}.
We will assume that $\mu$ is a measure such that $\mu^{(\omega)}(x_0,x_1,x_2,\dots)$ is well-defined.
That is, we can recursively define $\mu^{(n)}(x_0,x_1,\dots,x_{n-1})$, and each time it will be $M$-Borel definable.
At the moment, we know that this is true if $\mu$ is $M$-definable.

For the following lemmas, we will need some more notation, following \cite[Section 7.5]{nip_guide}. If $\phi(x;y)$ is a formula and $n \in \N$,
define the formula
$$f_n^{\phi}(\bar x, \bar x') = \sup_{y}|\mathrm{Av}(x_1,\dots,x_n;\phi(x;y)) - \mathrm{Av}(x_1',\dots,x_n';\phi(x;y))|.$$

\begin{lem}[{Generalizes \cite[Lemma 7.24]{nip_guide}}]\label{lem_keisler_VC1}
Let $\phi(x;y)$ be a formula. Then for any $n$, any Keisler measure $\mu \in \mathfrak{M}_x(M)$ with $\mu^{(2n)}$ totally indiscernible, and any $\varepsilon > 0$,
    $$\mu^{(2n)}\left(f_n^\phi(\bar x,\bar x')> \varepsilon \right) \leq 4\mathcal{N}_{\phi(x;y),\varepsilon/4}(n)\exp\left(-\frac{n\varepsilon^2}{32}\right).$$
\end{lem}
\begin{proof}
Let $R = \{-1,1\}^n$. We claim that
$$\mu^{(2n)}(f_n^\phi(\bar x,\bar x')> \varepsilon) \leq \frac{1}{2^{n-1}}\sum_{\sigma \in R}\mu^{(n)}\left(\left\{\bar x : \sup_{y} \frac{1}{n}\abs{\sum_{i = 1}^n\sigma_i \phi(x_i;y)} > \frac{\varepsilon}{2}\right\}\right).$$

By definition, we have
$$\mu^{(2n)}(f_n^\phi(\bar x,\bar x')> \varepsilon) = \mu\left(\sup_y\frac{1}{n}\abs{\sum_{i = 1}^n(\phi(x_i;y) - \phi(x_i';y))} > \varepsilon\right),$$
and by symmetry, this equals
\begin{align*}
    &\frac{1}{2^n}\sum_{\sigma \in R}\mu_{2n}\left(\sup_y\frac{1}{n}\abs{\sum_{i = 1}^n\sigma_i(\phi(x_i;y) - \phi(x_i';y))} > \varepsilon\right) \\ 
    \leq &\frac{1}{2^n}\sum_{\sigma \in R}\mu_{2n}\left(\sup_y\frac{1}{n}\abs{\sum_{i = 1}^n \sigma_i \phi(x_i;y)} > \frac{\varepsilon}{2} \textrm{ or } \sup_y\frac{1}{n}\abs{\sum_{i = 1}^n \sigma_i \phi(x_i';y)} > \frac{\varepsilon}{2}\right) \\
    \leq & \frac{1}{2^{n-1}}\sum_{\sigma \in R}\mu_n\left(\sup_y\frac{1}{n}\abs{\sum_{i = 1}^n \sigma_i \phi(x_i;y)} > \frac{\varepsilon}{2}\right)
\end{align*}
where in the last inequality we use symmetry and a union bound, proving the claim.

For any vector $\bar c \in [0,1]^n$, and $\delta > 0$, let $R(\bar c,\delta)$ be the set of vectors $\sigma \in R$ such that $\frac{1}{n} \abs{\sigma \cdot \bar c}> \delta$. For each $\bar a \in (M^x)^n$, let $\phi(\bar a;y) = (\phi(a_1;y),\dots, \phi(a_n;y))$. We will bound $|\bigcup_y R(\phi(\bar a;y),\frac{\varepsilon}{2})|$. If we show that this is at most $2^{n-1} \cdot 4 \mathcal{N}_{\phi(x;y),\varepsilon/4}(n)\exp\left(-\frac{n\varepsilon^2}{32}\right)$, then the lemma follows.

We first observe that for $\delta, \delta'$ and vectors $\bar c, \bar c' \in [0,1]^n$, if $|\bar c - \bar c'|_\infty < \delta'$, then $R(\bar c, \delta) \subseteq R(\bar c', \delta - \delta')$. To see this, let $\sigma \in R(\bar c, \delta)$, that is, $\frac{1}{n} \abs{\sigma \cdot \bar c}> \delta$, so $$\frac{1}{n}\abs{\sigma \cdot \bar c'} > \frac{1}{n}(\abs{\sigma \cdot \bar c} - \abs{\sigma \cdot (\bar c - \bar c')}) \geq \delta - |\bar c - \bar c'|_\infty \geq \delta - \delta'.$$

For any given $\bar a \in X^n$, there exists a set $C \subset [0,1]^n$ of size $\mathcal{N}_{\phi(x;y),\varepsilon/4}(n)$ such that for every $b \in \mathcal{U}^y$, there is $\bar c \in C^n$ such that $|\phi(\bar a;b) - \bar c|_\infty \leq \frac{\varepsilon}{4}$, and thus $R(\phi(\bar a;b),\frac{\varepsilon}{2}) \subseteq R(\bar c,\frac{\varepsilon}{4})$. Thus $\bigcup_{b \in \mathcal{U}^y} R(\phi(\bar a;b),\frac{\varepsilon}{2}) \subseteq \bigcup_{\bar c \in C} R(\bar c, \frac{\varepsilon}{4})$, and $$|\bigcup_{b \in \mathcal{U}^y} R(\phi(\bar a;y),\frac{\varepsilon}{2})|
\leq |C|\max_{\bar c \in C}|R(\bar c, \frac{\varepsilon}{4})|
\leq \mathcal{N}_{\phi(x;y),\varepsilon/4}(n)\max_{\bar c \in C}|R(\bar c, \frac{\varepsilon}{4})|.$$

For each individual vector $\bar c$, we can think of $\frac{1}{2^n}|R(\bar c, \frac{\varepsilon}{4})|$ probabilistically. If $\sigma$ is chosen uniformly at random from $R$, then this is $\mathbb{P}_\sigma[\frac{1}{n}|\sigma \cdot \bar c| >\frac{\varepsilon}{4}]$. As $\sigma \cdot \bar c = \sum_i \sigma_i c_i$ is the sum of $n$ independent random variables of mean 0 supported on $[-\frac{1}{n},\frac{1}{n}]$, we can apply Hoeffding's inequality to find that $\mathbb{P}_\sigma[\frac{1}{n}|\sigma \cdot \bar c| >\frac{\varepsilon}{4}] \leq 2\exp\left(-\frac{n\varepsilon^2}{32}\right)$. Thus
$|\bigcup_{b \in \mathcal{U}^y}R(\phi(\bar a;b),\frac{\varepsilon}{2})| \leq 2^n\mathcal{N}_{\phi(x;y),\varepsilon/4}(n)\left(2\,\mathrm{exp}\left(-\frac{n\varepsilon^2}{32}\right)\right)$, as desired.
\end{proof}

\begin{lem}[{See \cite[Proposition 7.26]{nip_guide}}]\label{lem_keisler_VC2}
Let $\phi(x;y)$ be a formula. Let any $n$, $\varepsilon > 0$ be such that $n \geq \frac{9}{2\varepsilon^2}$ and 
$\mathcal{N}_{\phi(x;y),\varepsilon/12}(n)\mathrm{exp}\left(-\frac{n\varepsilon^2}{96}\right) < \frac{1}{8}$.

Then for any Keisler measure $\mu \in \mathfrak{M}_x(M)$ with $\mu^{(2n)}|_M$ totally indiscernible,
there is a formula $\theta_{n, \varepsilon}(x_1,\dots,x_n)$ with parameters in $M$ such that
    \begin{itemize}
        \item Any $\bar a \in (\mathcal{U}^x)^n$ satisfying $\theta_{n,\varepsilon}(\bar a) = 0$ is a $\varepsilon$-approximation to $\phi(x;y)$ with respect to $\mu$.
        \item $\mu^{(n)}(\theta_{n,\varepsilon}(x_1,\dots,x_n)) \geq 1 - 8\mathcal{N}_{\phi(x;y),\varepsilon/12}(n)\exp\left(-\frac{n\varepsilon^2}{96}\right)$
    \end{itemize}
\end{lem}
\begin{proof}
Let $\theta_{n,\varepsilon}'(x_1,\dots,x_n, x_1',\dots,x_n') = f_n^\phi(x_1,\dots,x_n, x_1',\dots,x_n') \dot{-}\frac{\varepsilon}{3}$.

Then by Lemma \ref{lem_keisler_VC1}, $\mu^{(2n)}(\theta_{n,\varepsilon}') \geq 1 - 4\mathcal{N}_{\phi(x;y),\varepsilon/12}(n)\mathrm{exp}\left(-\frac{n\varepsilon^2}{96}\right)$.
Thus there exists $\bar a' \in (\mathcal{U}^x)^n$ such that $\mu^{(n)}(\theta_{n,\varepsilon}'(x_1,\dots,x_n; \bar a')) \geq 1 - 4\mathcal{N}_{\phi(x;y),\varepsilon/12}(n)\mathrm{exp}\left(-\frac{n\varepsilon^2}{96}\right) > \frac{1}{2}$.
Then let $\theta_{n,\varepsilon} = \theta_{n,\varepsilon}'(\bar x; \bar a')$.

It now suffices to show that for all $\bar a$ satisfying $\theta_{n,\varepsilon}'(\bar a;\bar a') = 0$, and any $b \in \mathcal{U}^y$,
$\bar a$ is a $\varepsilon$-approximation to $\phi(x;b)$ with respect to $\mu$.
Fix $b$, and let $\zeta_n(x_1,\dots, x_n) = |\mathrm{Av}(x_1,\dots,x_n) - \mathbb{E}_\mu[\phi(x;b)]|\dot{-}\frac{\varepsilon}{3}$.
By the weak law of large numbers, as the functions $\phi(x_i;b)$ are $[0,1]$-valued i.i.d. random variables with respect to $\mu$,
$\mu^{(n)}(\zeta_n = 0) \geq 1 - \frac{1}{4n(\varepsilon/3)^2} = 1 - \frac{9}{4n\varepsilon^2}$.
Thus as $n \geq \frac{9}{2\varepsilon^2}$, we have $\mu^{(n)}(\zeta_n) \geq \frac{1}{2}$.
As $\mu^{(n)}(\theta_{n,\varepsilon} = 0) > \frac{1}{2}$ also, $\mu^{(n)}(\theta_{n,\varepsilon} \wedge \zeta_n = 0) > 0$,
so let $\bar a^*$ be such that $\vDash \theta_{n,\varepsilon}(\bar a^*) = 0$ and $\zeta_n(\bar a^*) = 0$.
Thus $|\mathrm{Av}(\bar a') - \mathbb{E}_\mu[\phi(x;b)]| \leq |\mathrm{Av}(\bar a') - \mathrm{Av}(\bar a^*)| + |\mathrm{Av}(\bar a^*) - \mathbb{E}_\mu[\phi(x;b)]|\leq \frac{2\varepsilon}{3}$.
Thus if $\bar a$ is such that $\vDash \theta_{n,\varepsilon}'(\bar a;\bar a') = 0$, we also have $|\mathrm{Av}(\bar a) - \mathbb{E}_\mu[\phi(x;b)]| \leq \varepsilon$.
\end{proof}

\begin{thm}[{Generalizes \cite[Proposition 7.26]{nip_guide}, \cite[Lemma 3.3]{alon97}}]\label{thm_ind_fim}
Any Keisler measure $\mu \in \mathfrak{M}_x(\mathcal{U})$ with $\mu^{(\omega)}|_M$ totally indiscernible is fim over $M$.
\end{thm}
\begin{proof}
Let $\phi(x;y)$ be a formula. By \cite[Lemma 4.12]{anderson1}, for all $\varepsilon > 0$, there exist $C, k$ such that
$\mathcal{N}_{\phi(x;y),\varepsilon/12}(n) \leq C n^k$ for all $n \geq 2$.
Thus we have
$$\lim_{n \to \infty}\mathcal{N}_{\phi(x;y),\varepsilon/12}(n)\mathrm{exp}\left(-\frac{n\varepsilon^2}{96}\right) \leq
    \lim_{n \to \infty}n^k\mathrm{exp}\left(-\frac{n\varepsilon^2}{96}\right) = 0.$$
Thus for each $m \in \N$, we can let $n_m \in \N$ be such that $n_m \geq \frac{9}{2}m^2$ and $n^k \mathrm{exp}\left(- \frac{n}{96m^2}\right) < \frac{1}{8m}$.

Then for each $n \in \N$, let $\theta_n(x_1,\dots,x_n) = \theta_{n,1/m_n}(x_1,\dots,x_n)$,
where $m_n \in \N$ is the greatest natural number such that $n_{m_n} \leq n$.
Then by Lemma \ref{lem_keisler_VC2}, any $\bar a \in (\mathcal{U}^x)^n$ satisfying $\theta_{n,1/m_n}(\bar a) = 0$ is a $\frac{1}{m_n}$-approximation to $\phi(x;y)$ with respect to $\mu$,
and 
$$\mu^{(n)}(\theta_n(x_1,\dots,x_n)) \geq 1 - 8\mathcal{N}_{\phi(x;y),1/(12 m_n)}(n)\exp\left(-\frac{n}{96m_n^2}\right) \geq 1 - \frac{1}{m_n}.$$

Now it suffices to show that $\lim_{n \to \infty}m_n = \infty$. This is true as for each $m$, if $n \geq n_m$, then $m_n \geq m$.
\end{proof}

\begin{cor}\label{cor_keisler_approx}
    If $\mathrm{vc}_{\varepsilon/12}(\phi(x;y)) \leq d$, and $\mu \in \mathfrak{M}_x(\mathcal{U})$ is such that $\mu^{(\omega)}|_M$ is totally indiscernible,
    then $\phi(x;y)$ admits a $\varepsilon$-approximation of size at most
    $O(\frac{d}{\varepsilon^2}\ln\frac{d}{\varepsilon})$ with respect to $\mu$.
    \end{cor}
    \begin{proof}
    It suffices to find $n$ such that $n \geq \frac{9}{2\varepsilon^2}$ and 
    $\mathcal{N}_{\phi(x;y),\varepsilon/12}(n)\mathrm{exp}\left(-\frac{n\varepsilon^2}{96}\right) < \frac{1}{8}$.
    Then by Lemma \ref{lem_keisler_VC2}, there is a positive-measure (with respect to $\mu^{(n)}$) set of $\varepsilon$-approximations of size $n$ to $\phi(x;y)$ with respect to $\mu$.
    
    Let $C$ be the constant depending on $d,\varepsilon$ such that $\mathcal{N}_{\phi(x;y),\varepsilon/12}(n) \leq n^{C\ln n} = e^{C\ln^2 n}$, according to \cite[Fact 2.18]{anderson1}, first stated in the proof of \cite[Lemma 3.5]{alon97}.
    \end{proof}

\begin{lem}
    If $\mu \in \mathfrak{M}_x(\mathcal{U})$ is dfs over $M$, then it is fim over $M$.
\end{lem}
\begin{proof}
Let $\mu$ be dfs.
We know by Lemma \ref{lem_dfs_commute} that $\mu(x) \otimes \mu(y) = \mu(y) \otimes \mu(x)$.
By definability of $\mu$ and Lemma \ref{lem_definable_product}, we know that $\mu^{(\omega)}(x_0,\dots,x_{n-1})|_M$ is well-defined, and by commutativity, it is totally indiscernible.
Thus by Theorem \ref{thm_ind_fim}, $\mu$ is fim over $M$.
\end{proof}

In particular, smooth measures are fim, so any measure admits a fim extension,
from which we can show that every measure is locally approximated by types in its support.

\begin{lem}\label{lem_loc_approx_types}
Let $\mu \in \mathfrak{M}_x(M)$ be a Keisler measure, $\phi(x;y)$ a definable predicate, $\varepsilon > 0$.

There are types $p_1,\dots,p_n \in S(\mu)$ such that for every $b \in M^y$,
$$\left|\int_{S_x(M)}\phi(x;b)\,d\mu - \mathrm{Av}(p_1,\dots,p_n;\phi(x;b))\right|\leq \varepsilon.$$
\end{lem}
\begin{proof}
Let $\nu \in \mathfrak{M}_x(N)$ be a fim extension of $\mu$, with $M \preceq N$.
Then for every $\phi(x;y)$, there is some closed $N$-condition $\theta(x_1,\dots,x_n)$ with
$\nu^{(n)}(\theta(x_1,\dots,x_n))> \frac{1}{2}$, such that any $(a_1,\dots,a_n)$ with $\vDash \theta(a_1,\dots,a_n)$
is a $\varepsilon$-approximation to $\phi(x;y)$ with respect to $\mu$.

We claim that there are some $a_1,\dots,a_n$ such that for each $i$, $\mathrm{tp}(a_i/M) \in S(\mu)$ and $\vDash \theta(a_1,\dots,a_n)$.
Then we can let $p_1,\dots,p_n$ be the types of $a_1,\dots,a_n$ over $M$.
To do that, we just have to show that any finite set of closed conditions in $\{\theta_n(x_1,\dots,x_n)\} \cup \bigcup_{i = 1}^n S_\mu(x_i)$ is satisfiable, where
$S_\mu(x_i)$ is the partial type indicating that $\mathrm{tp}(x_i/M) \in S(\mu)$, consisting of all closed $M$-conditions with positive $\mu$-measure.
As $\nu$ is an extension of $\mu$, the $\nu^{(n)}$-measure of the intersection of the finite set is at least $\nu(\theta_n(x_1,\dots,x_n))>\frac{1}{2}$.
Thus this finite partial type is satisfiable.
\end{proof}

Finally we are able to show that assuming NIP, every $M$-invariant measure is $M$-Borel definable, simplifying many of our earlier results.
This was originally shown for classical logic in \cite{HP} using a VC-Theorem argument.
\begin{lem}
    Let $\mu \in \mathfrak{M}_x(\mathcal{U})$ be $M$-invariant. Then $\mu$ is $M$-Borel definable.
\end{lem}
\begin{proof}
Let $\phi(x;y)$ be a definable predicate.
For every $\varepsilon > 0$, we will find a Borel function $f_\varepsilon : S_y(\mathcal{U}) \to [0,1]$ such that
$|F^\phi_{\mu, M} - f_\varepsilon(y)| \leq \varepsilon$.
Then $F^\phi_{\mu, M}$ is a uniform limit of Borel functions, and as Borel functions are closed under even pointwise limits,
$F^\phi_{\mu, M}$ is Borel.

By Lemma \ref{lem_loc_approx_types}, there are types $p_1,\dots,p_n \in S(\mu)$ such that for every $b \in M^y$,
$$\left|\int_{S_x(M)}\phi(x;b)\,d\mu - \mathrm{Av}(p_1,\dots,p_n;\phi(x;b))\right|\leq \varepsilon.$$
Let $f_\varepsilon(y) = \mathrm{Av}(p_1,\dots,p_n;\phi(x;b))$.
As each $p_i$ is $M$-invariant by Lemma \ref{lem_inv_support}, and an average of Borel functions is Borel, it is enough to show that the Dirac measure of an invariant type is Borel-definable.

Let $p$ be an $M$-invariant type.
It suffices to show that for each $r > 0$, the set
$\{q : \phi(x;b)<r \in p(x)\textrm{ for }b \vDash q\}$ is Borel.

Fix $b \in \mathcal{U}$, $\varepsilon > 0$.
By NIP, there is some maximal $N$ such that there is $(a_i : i \leq N)\vDash p^{(N)}(x)|_M$
with
$|\phi(a_i;b) - \phi(a_{i + 1};b)|\geq \varepsilon$
for $0 \leq i < N$.
By the maximality of $N$, we see that
$|\phi(p;b) - \phi(a_N;b)|< \varepsilon$.

Now let $A_{N, \varepsilon}(y)$ indicate the set in $S_y(M)$ such that
there exist $a_0,\dots, a_N$ satisfying $p^{(N)}|_M$ such that $|\phi(a_i;y) - \phi(a_{i + 1};y)|\geq \varepsilon$ for all $i < N$, and 
$\phi(a_N;y) \leq r - \varepsilon$.
This set is closed, as by saturation, this holds if and only if for each closed condition $\chi(x_0,\dots,x_n) = 0 \in p^{(N)}|_M$,
$$\inf_{x_0,\dots,x_n}\max\left(\chi(x_0,\dots,x_n), \phi(a_N;y)\dot{-}(r - \varepsilon),\max_{i < N}\left(\varepsilon \dot{-}\left|\phi(a_i;y) - \phi(a_{i + 1};y)\right|\right)\right) = 0$$
holds.

Let $B_{N,\varepsilon}(y)$ be the weaker condition that there exist $a_0,\dots, a_N$ satisfying $p^{(N)}|_M$ such that $|\phi(a_i;y) - \phi(a_{i + 1};y)|\geq \varepsilon$ for all $i < N$.
Then by NIP, for every $\varepsilon > 0$, every $b \in \mathcal{U}^y$, there is some $N$ such that $B_{N,\varepsilon}(b)$ holds, but $B_{N + 1,\varepsilon}(b)$ does not.
If in addition, $p \vDash \phi(x;b) < r - 2\varepsilon$, then we know that $A_{N,\varepsilon}(b)$ holds, as in any maximal sequence $a_0,\dots,a_N$ witnessing $B_{N,\varepsilon}(b)$,
we have $|\phi(p;b) - \phi(a_N;b)|< \varepsilon$, so $\phi(a_N;b)<r - \varepsilon$.
Also, if $b$ is such that $A_{N,\varepsilon}(y)$ holds but $B_{N + 1,\varepsilon}(b)$ does not, then $p \vDash \phi(x;b) < r$,
as in any witness sequence $a_0,\dots,a_N$, we have $|\phi(p;b) - \phi(a_N;b)|< \varepsilon$ and $\phi(a_N;b) < r - \varepsilon$.
Thus $\{q : \phi(x;b)<r \in p(x)\textrm{ for }b \vDash q\} = \bigcup_{N,m \in \N} \left(A_{N,1/m}(y) \setminus B_{N + 1,1/m}(y)\right)$,
which is a countable union of boolean combinations of closed sets, and is thus Borel.

\end{proof}

We can also use the indiscernibility of $\mu^{(\omega)}$ to prove a version of \cite[Theorem 2.25]{anderson1} with respect to generically stable measures:
\begin{thm}\label{thm_keisler_net}
For any $\varepsilon > 0$, $d \in \N$, $0 \leq r < s \leq 1$, there is $N = O_{r,s}(d\varepsilon^{-1}\log \varepsilon^{-1})$
such that if $\phi(x;y)$ is a definable predicate with $\mathrm{vc}_{r',s'}(\phi(x;y)) \leq d$ for some $r < r' < s' < s$,
and $\mu \in \mathfrak{M}_x(M)$ is generically stable, then there is an $\varepsilon$-net $A$ for
the fuzzy set system $\phi^{M^y}_{r,s}$ with respect to $\mu$ with $|A| \leq N$.
\end{thm}
\begin{proof}
This proof generalizes the argument by Haussler and Welzl used in \cite[Theorem 10.2.4]{matousek_GTM}.

Fix $\varepsilon > 0$, $d$, $\phi(x;y)$ a formula, and $\mu \in \mathfrak{M}_x(M)$ generically stable.
Let $r',s'$ be such that that $r < r' < s' < s$ and $\mathrm{vc}_{r',s'}(\phi(x;y))\leq d$.

Let $N = Cd\varepsilon \log(\varepsilon^{-1})$, with $C$ to be determined later.

We will define an open set $E_0 \subseteq S_N(M)$ such that if $(a_1,\dots,a_N)$ is not a $\varepsilon$-net, $\mathrm{tp}(a_1,\dots,a_N/M) \in E_0$.
Then we will find conditions on $N$ that guarantee $\mu^{(N)}(E_0) < 1$, implying that there exists some $(a_1,\dots,a_N)$ with $\mathrm{tp}(a_1,\dots,a_N/M)\not \in E_0$, which must therefore be a $\varepsilon$-net.

For $b \in M^y$, let $E_{0,b} \subseteq S_N(M)$ be $\bigcap_{i = 1}^N\{\phi(x_i;b)< r'\}$, and let $E_0 = \bigcup_{b \in M, \mu(\phi(x;b)\geq s) \geq \varepsilon} E_{0,b}$.
We see that each $E_{0,b}$ is open, and thus $E_0$ is open. (The purpose of using $< r'$ instead of $\leq r$ is to guarantee measurability of $E_0$.)
If $(a_1,\dots,a_N)$ is not a $\varepsilon$-net, then there exists some $b \in M^y$ such that $\mu(\phi(x;b)\geq s) \geq \varepsilon$ and for all $1 \leq i \leq N$, $\phi(a_i;b) \leq r < r'$, so $\mathrm{tp}(a_1,\dots,a_N/M) \in E_{0,b} \subseteq E_0$.

Now define $E_{1,b} \subseteq S_{2N}(M)$ be the (open) event that for all $1 \leq i \leq N$, $\phi(x_i;b)< r'$, and for at least $k = \lceil \frac{N\varepsilon}{2}\rceil$ values of $1 \leq i \leq N$,
$\phi(x_{N + i};b)> s'$, and let $E_{1}$ be $\bigcup_{b \in M, \mu(\phi(x;b)\geq s) \geq \varepsilon} E_{1,b}$.
We wish to show that $\mu^{(N)}(E_0) \leq 2\mu^{(2N)}(E_1)$ and that $\mu^{(2N)}(E_1) > \frac{1}{2}$.

In order to show that $\mu^{(N)}(E_0) \leq 2\mu^{(2N)}(E_1)$, we will split up the tuple of variables $(x_1,\dots, x_{2N})$ into $\bar x = (x_1,\dots,x_N)$ and $\bar x' = (x_{N + 1},\dots, x_{2N})$, and look at conditional probability.
By Lemma \ref{lem_BD++}, as $E_1$ is open, the function defined by $\mu^{(N)}(E_1|p) := \mu^{(N)}_{\bar x'}((\bar a,\bar x') \in E_1)$ where $\bar a \vDash p$ is Borel, and $\mu^{(2N)}(E_1) = \int \mu^{(N)}(E_1|p)\,d(\mu^{(N)})$.
Thus it suffices to show that for all $p$, $\chi_{E_0}(p) \leq 2\mu^N(E_1|p)$, where $\chi_{E_0}$ is the characteristic function of $E_0$.

Fix $p \in S_N(M)$ and $\bar a \vDash p$ with $\bar a = (a_1,\dots,a_N)$.
If $p \not \in E_0$, then for all $q \in S_{2N}(M)$ extending $p$, $q \not \in E_1$, so
$\chi_{E_0}(p) = 0 \leq 0 = 2\mu^N(E_1|p)$.

Now assume $p \in E_0$, and let $b$ be such that $p \in E_{0,b}$.
Let $I_i$ for $1 \leq i \leq N$ be the indicator random variables (on $S_{2N}(M)$) for $\phi(x_{N + i};b) > s'$, and $I = I_1 + \dots + I_N$.
We have that $\mu^{(N)}(E_1|p) = \mu^{(N)}(\{I \geq k\}|p)$.
The $I_i$s are i.i.d. random variables, equalling 1 with probability $\mu(\{\phi(x;b)> s'\}) \geq \varepsilon$.
By a standard Chernoff tail bound for binomial distributions, we have that $\mu^{(N)}(\{I \geq k\}|p) \geq \frac{1}{2} = \frac{1}{2}\chi_{E_0}(p)$.
Thus in general, $\mu^{(N)}(E_0) \leq 2\mu^{(2N)}(E_1)$.

To show that $\mu^{(2N)}(E_1)< \frac{1}{2}$, we will instead condition on the multiset $\{x_1,\dots,x_{2N}\}$.
Given a tuple $\bar a = (a_1,\dots,a_{2N})$ and a permutation $\sigma$ of $\{1,\dots,2N\}$, let $\sigma(\bar a)$ refer to $(a_{\sigma(1)},\dots, a_{\sigma(2N)})$.
To formally condition on the multiset $\{x_1,\dots,x_{2N}\}$, we will show that for every tuple $\bar a$,
$\mathbb{P}_{\sigma}[\sigma(\bar a) \in E_1] > \frac{1}{2}$, where $\sigma$ is a permutation on $\{1,\dots,2N\}$ selected uniformly at random.
This probability is calculated as
$$\mathbb{P}_{\sigma}[\sigma(\bar a) \in E_1] 
= \frac{1}{n!}\sum_\sigma\chi_{E_1}(\sigma(\bar a)),$$
and we also see that because $\mu^{(\omega)}|_M$ is totally indiscernible, for any $\sigma$ we have
$$\int_{S_{2N}(M)}\chi_{E_1}(\bar x)\,\mu^{(2N)} = \int_{S_{2N}(M)}\chi_{E_1}(\sigma(\bar x))\,\mu^{(2N)}.$$
Thus we see that
\begin{align*}
    \mu^{(2N)}(E_1) &= \int_{S_{2N}(M)}\chi_{E_1}(\bar x)\,\mu^{(2N)}\\
    &= \frac{1}{n!}\sum_{\sigma}\int_{S_{2N}(M)}\chi_{E_1}(\bar x)\,\mu^{(2N)}\\
    &= \frac{1}{n!}\sum_{\sigma}\int_{S_{2N}(M)}\chi_{E_1}(\sigma(\bar x))\,\mu^{(2N)}\\
    &= \int_{S_{2N}(M)}\frac{1}{n!}\sum_{\sigma}\chi_{E_1}(\sigma(\bar x))\,\mu^{(2N)}\\
    &= \int_{S_{2N}(M)}\mathbb{P}_{\sigma}[(\sigma(\bar x)) \in E_1]\,\mu^{(2N)}\\
    &<\frac{1}{2}
\end{align*}
finishing the reduction.

We now fix $\bar a$ and work with finite probability, selecting a random permutation $\sigma$ of the variables in $\bar a$.

Let $\mathcal{F}$ be the fuzzy set system on $\{1,\dots,2N\}$ consisting of the fuzzy sets $S_b$ for $b \in M^y$
where $S_{b+} = \{i : \phi(a_i;b)>s'\}$ and $S_{b-} = \{i : \phi(a_i;b)<r'\}$.
By \cite[Lemma 2.6]{anderson1}, there is a strong disambiguation $\mathcal{F}'$ for $\mathcal{F}$ of size $|\mathcal{F}'|=(2N)^{O(d\log(2N))}$,
or as we will prefer later, there is $C'$ such that $|\mathcal{F}'|\leq (2N)^{C'(d\log(2N))}$.
Recall that this means that for all $b \in M^y$, there is some $S \in \mathcal{F}'$ with
$\{i : \phi(a_i;b)>s'\} \subseteq S$ and $\{i : \phi(a_i;b)<r'\} \cap S = \emptyset$.
Given a set $S \in \mathcal{F}'$, let $E_S$ be the event that $\{\sigma(1),\dots,\sigma(N)\}\cap S = \emptyset$
and $|\{\sigma(N+1),\dots,\sigma(2N)\}\setminus S|\geq k$.
We see that if $\sigma(\bar a) \in E_1$, then there is some $b$ with $\sigma(\bar a) \in E_{1,b}$.
There is also some $S \in \mathcal{F}'$ with $\{i : \phi(a_i;b)>s'\} \subseteq S$ and $\{i : \phi(a_i;b)<r'\} \cap S = \emptyset$,
so $E_S$ occurs. Thus $\mathbb{P}_{\sigma}[\sigma(\bar a) \in E_1] \leq \sum_{S \in \mathcal{F}'} \mathbb{P}_{\sigma}[E_S]$.
For each $S$, if $|S|< k$, then $\mathbb{P}_{\sigma}[E_S] = 0$, but if $|S|\geq k$, then $\mathbb{P}_{\sigma}[E_S]$ is the probability that
when a permutation $\sigma$ is selected uniformly at random,
$\{\sigma(1),\dots,\sigma(N)\}\cap S = \emptyset$. This is at most
$$\frac{{2N - |D \cap S'|\choose N}}{{2N \choose N}} \leq \frac{{2N - k\choose N}}{{2N \choose N}} \leq \left(1 - \frac{k}{2N}\right)^N \leq e^{- (k/2N)N} = \varepsilon^{Cd/4}.$$

Now we bound $\mathbb{P}_{\sigma}[\sigma(\bar a) \in E_1]$:
\begin{align*}
    \mathbb{P}_{\sigma}[\sigma(\bar a) \in E_1] &\leq \sum_{S \in \mathcal{F}'} \mathbb{P}_{\sigma}[E_{S}] \\
    &\leq |\mathbb{F}'|\varepsilon^{-Cd/4}\\
    &\leq (2N)^{C'(d \log (2N))}\varepsilon^{Cd/4} \\
    &= \left((2Cd\varepsilon^{-1} \log \varepsilon^{-1})^{C'(\log (2Cd\varepsilon^{-1} \log \varepsilon^{-1}))}\varepsilon^{C/4}\right)^d
\end{align*}
While this expression is somewhat complicated, it is still clear that an increasing quasipolynomial function of $C$ times a decreasing exponential of $C$ will limit to $0$, so for large enough $C$,
we find that $\mathbb{P}_\sigma[\sigma(\bar a) \in E_1] < \frac{1}{2}$.
\end{proof}

\subsection{Ultraproducts}
In this subsection, we recall the definition of the ultraproduct of a family of Keisler measures from discrete logic. See for instance \cite[Page 98]{nip_guide}.
Let $I$ be an index set, $(M_i : i \in I)$ a family of models, $(\mu_i : i \in I)$ a family of Keisler measures in $\mathfrak{M}_x(M_i)$,
    and $U$ an ultrafilter on $I$. By $[a_i : i \in I]$, if $a_i \in M_i^x$ for each $i$, we denote the equivalence class of $(a_i : i \in I)$,
    as an element of the sort over the ultraproduct $(\prod_U M_i)^x$.

\begin{defn}\label{defn_ultra}
     Then we define the ultraproduct $\prod_U \mu_i$ to be the Keisler measure in $\mathfrak{M}_x(\prod_U M_i)$ such that for all $\phi(x;y)$,
    and all $b = [b_i : i \in I] \in (\prod_U M_i)^y$, we have
        $$\int \phi(x; b) \,d \prod_U \mu_i = \lim_U \int \phi(x;b_i)\,d\mu_i,$$
    where $\lim_U$ is the ultralimit, defined as the values lie in a compact subset of $\R$.
\end{defn}

\begin{lem}[See {\cite[Corollary 1.3]{simon_distal_reg}}]\label{lem_ultra_gs}
    Assume NIP.
    If $(\mu_i : i \in I)$ is a sequence of generically stable measures and $U$ an ultrafilter, then $\prod_U \mu_i$ is generically stable.
\end{lem}
\begin{proof}
    Let $\phi(x,y,z)$ be a formula, and fix $\varepsilon > 0$ and $c = [c_i : i \in I] \in \prod_U M_i$.
    We will show that $\prod_U \mu_i$ is fam by finding an $\varepsilon$-approximation 
    to the family $\{\phi(x,b;c) : b \in \left(\prod_U M_i\right)^y\}$.

    First we observe that by Lemma \ref{lem_keisler_VC2}, as each $\mu_i$ is generically stable,
    there is some $n$ depending only on $\phi(x,y,z)$ and $\varepsilon$ such that 
    for each $i$, there exists an $\varepsilon$-approximation $(a^1_i,\dots,a^n_i)$ to
    $\{\phi(x,b_i;c_i) : b_i \in M_i^y\}$.

    Now for $1 \leq j \leq n$, let $a^j  = [a^j_i : i \in I]$.
    We claim that $(a^1,\dots,a^n)$ is a $\varepsilon$-approximation to $\{\phi(x,b;c) : b \in \left(\prod_U M_i\right)^y\}$.
    Fix $b  = [b_i : i \in I]\left(\prod_U M_i\right)^y$.
    Then $|\mathrm{Av}(a^1_i,\dots,a^n_i;\phi(x;b_i,c_i)) - \int_{S_x(M_i)}\phi(x;b_i,c_i) \,d\mu_i| \leq \varepsilon$.
    By the definitions of ultraproducts, we know that 
    $\mathrm{Av}(a^1,\dots,a^n;\phi(x;b,c)) = \lim_U \mathrm{Av}(a^1_i,\dots,a^n_i;\phi(x;b_i,c_i))$
    and $\int_{S_x\left(\prod_U M_i\right)}\phi(x;b,c) \,d\prod_U \mu_i = \lim_U \int_{S_x(M_i)}\phi(x;b_i,c_i) \,d\mu_i$,
    and these ultralimits differ by at most $\varepsilon$, as the sequences do pointwise.
\end{proof}

\subsection{Indiscernible Segments}
An \emph{indiscernible segment} is an indiscernible sequence indexed by the order $[0,1] \subseteq \R$.
We will prove one characterization of NIP using indiscernible segments, and then assume $T$ is NIP for the rest of this subsection.

\begin{lem}\label{lem_ind_regulated}
A theory $T$ is NIP if and only if the following holds:
For every indiscernible segment $I = (a_i : t \in [0,1])$ with $|a_t| = |x|$ and any definable predicate $\phi(x;b)$ with $b \in \mathcal{U}^y$, the function $t \mapsto \phi(a_t;b)$ is regulated (a uniform limit of step functions), and thus measurable.
\end{lem}
\begin{proof}
NIP is equivalent to every $\phi(x;b)$ having a limit on every indiscernible sequence of order type $\omega$.

First, assume that $T$ is NIP.
Thus on every indiscernible segment, for any increasing or decreasing sequence $(t_n : n \in \N) \in [0,1]^\N$, $\lim_{n}\phi(a_{t_n};b)$ exists,
so at every point $t_0 \in (0,1)$, the limits $\lim_{t \to t_0^+}\phi(a_t;b)$ and $\lim_{t \to t_0^-}\phi(a_t;b)$ both exist.
If they did not, we could find an increasing or decreasing sequence limiting to $t_0$ on which $\phi(x;b)$ has no limit.
By \cite[Th\'eor\`eme 3, FVR II.5]{bourbaki_FVR}, a function on $[0,1]$ is regulated if and only if its left and right limits all exist.

Now assume that $T$ is not NIP - there must be some indiscernible sequence $(a_n : n \in \N)$ and some $\phi(x;b)$ with $\lim_n \phi(a_n;b)$ undefined.
Then by restricting to a non-Cauchy subsequence, we find that there is some $\varepsilon > 0$ such that $|\phi(a_n;b) - \phi(a_{n+1};b)|\geq \varepsilon$.
We claim that there also exists an indiscernible segment $(a_t' : t \in [0,1])$ and some $b'$ where $\lim_n \phi(a_{1/n}';b')$ also does not exist, making $\phi(x_t;b')$ not regulated.
We can find this counterexample by realizing the type given by $\psi(x_{t_1},\dots,x_{t_n}) = \psi(x_{t_1'},\dots,x_{t_n'})$ for all increasing tuples $t_1 < \dots < t_n$ and $t_1' < \dots < t_n'$
and $|\phi(x_{1/n};y) - \phi(x_{1/(n+1)};y)|\geq \varepsilon$, every finite subtype of which is realized by any finite subsequence of $(a_n : n \in \N)$.
\end{proof}



We now assume $T$ is NIP. This allows us to define the average measure of an indiscernible segment.
\begin{defn}\label{defn_average}
If $I = (a_t : t \in [0,1])$ is an indiscernible segment, define the \emph{average measure} of $I$, $\mu_I \in \mathfrak{M}_x(\mathcal{U})$, to be the unique global Keisler measure with
$$\int_{S_x(\mathcal{U})} \phi(x;b)\,d\mu_I = \int_0^1\phi(a_t;b)\,dt.$$
\end{defn}

We will show that these measures are generically stable using the following lemma:
\begin{lem}\label{lem_reg_approx}
If $F$ is a family of functions $[0,1] \to [0,1]$ such that for each $f \in F$, there is no sequence $0 \leq t_1 < t_1' <\dots <t_N < t_N'\leq 1$
with $|f(t_i) - f(t_i')|> \frac{\varepsilon}{2}$ for all $1 \leq i \leq N$, then for any $M > \frac{2N}{\varepsilon}$, 
the set $A = \left\{\frac{k}{M} : 0 \leq k < M\right\}$ is a $\varepsilon$-approximation to $F$ with respect to the Lebesgue measure on $[0,1]$.
\end{lem}
\begin{proof}
    For any $f \in F$,
\begin{align*}
    \left|\mathrm{Av}_{a \in A} f(a) - \int_0^1 f(t)\,d t\right|
    & = \frac{1}{M}\left|\sum_{k = 0}^{M - 1}\left(f\left(\frac{k}{M}\right) - M\int_{k/M}^{(k+1)/M}f(t)\,dt\right)\right|\\
    & \leq \frac{1}{M}\sum_{k = 0}^{M - 1}\left|f\left(\frac{k}{M}\right) - M\int_{k/M}^{(k+1)/M}f(t)\,dt\right|,
\end{align*}
so it suffices to show that there are few values of $k$ such that $\left|f\left(\frac{k}{M}\right) - M\int_{k/M}^{(k+1)/M}f(t)\,dt\right|$ is large.

For any integer $0 \leq k < M$, either $|f\left(\frac{k}{M}\right) - M\int_{k/M}^{(k+1)/M}f(t)\,dt| \leq \frac{\varepsilon}{2}$, or
there is some $c \in [\frac{k}{M}, \frac{k+1}{M}]$ such that $|f\left(\frac{k}{M}\right) - \phi(\sigma_c;b)| > \frac{\varepsilon}{2}$.
By the choice of $N$, there are at most $N$ values of $k$ such that the latter case holds, in which case we can bound $\left|\phi(a_{k/M};b) - M\int_{k/M}^{(k+1)/M}f(t)\,dt\right| \leq 1$.
Thus $\sum_{k = 0}^{M - 1}\left|f\left(\frac{k}{M}\right) - M\int_{k/M}^{(k+1)/M}f(t)\,dt\right|$ is bounded by $M\frac{\varepsilon}{2}$ for most intervals plus $N$ for exceptional intervals, yielding
\begin{align*}
    \frac{1}{M}\sum_{k = 0}^{M - 1}\left|f\left(\frac{k}{M}\right) - \int_{k/M}^{(k+1)/M}f(t)\,dt\right|
\leq \frac{1}{M}\left(\frac{M\varepsilon}{2} + N\right)
< \frac{\varepsilon}{2} + \frac{\varepsilon}{2} = \varepsilon
\end{align*}
as desired.
\end{proof}

\begin{lem}\label{lem_ave_gs}
If $\sigma \in \mathfrak{M}_{[0,1]}(\mathcal{U})$ is an indiscernible segment, then $\mu_\sigma$ is generically stable.
\end{lem}
\begin{proof}
Specifically, we show that $\mu_I$ is fam in $I$.
Fix a definable predicate $\phi(x;y)$ and $\varepsilon > 0$.
By Lemma \ref{lem_ind_regulated}, for all $b \in \mathcal{U}^y$, there is some $N$ such that there is no sequence $0 \leq t_1 < t_1' <\dots <t_N < t_N'\leq 1$ with
$|\phi(a_{t_i};b) - \phi(a_{t_i'};b)|> \frac{\varepsilon}{2}$ for all $1 \leq i \leq N$.
By compactness, we may choose some $N$ that will work simultaneously for all $b \in \mathcal{U}^y$.
Now fix an integer $M >\frac{2N}{\varepsilon}$, and let $A = \{a_{k/M} : 0 \leq k < M\}$.
We can apply Lemma \ref{lem_reg_approx} to the family $F = \{f_b(t): b \in \mathcal{U}^y\}$ defined by $f_b(t) = \phi(a_t;b)$,
and find that $\left\{\frac{k}{M} : 0 \leq k < M\right\}$ is a $\varepsilon$-approximation to $F$ with respect to Lebesgue measure.
Then as $\int_{S_x(\mathcal{U})}\phi(a_t;b)\,d\mu_I = \int_0^1 f_b(t)\,dt$,
we see that $A$ is a $\varepsilon$-approximation to $\{\phi(a_t;b) : b \in \mathcal{U}^y\}$ with respect to $\mu_I$, as desired.
\end{proof}

\section{Weak Orthogonality and Regularity}\label{sec_wo}
In this section, we will characterize weak orthogonality of measures with several Szemer\'edi-style regularity properties.
In the next section, we will use uniform versions of these properties to characterize distality.
In order to best explain our techniques and choices, we will first prove some versions of NIP regularity.
NIP regularity theorems already exist for classical logic, for instance in \cite{nip_reg}, and indeed for continuous logic, in the form of \cite[Theorem 6.6]{ct},
which proves a regularity lemma for real-valued definable predicates in a generalization of NIP structures.
However, we find that NIP regularity is the correct setting to first develop the definitions we will need for distal regularity,
namely definable and constructible regularity partitions.

\subsection{NIP Regularity}
NIP regularity is a consequence of the ability to approximate definable predicates relative to generically stable measures.

In order to understand it, we must first understand how the classic partitioning into $\phi$-types over finite sets works in continuous logic.
As even for finite parameter sets $B$, the set $S_\phi(B)$ of consistent $\phi$-types $\mathrm{tp}_\phi(x;y)(a/B)$ is usually infinite in continuous logic,
we have to look at partitions where the type $\mathrm{tp}_\phi(x;y)(a/B)$ varies by at most $\varepsilon$ on the support of each piece of the partition.
Even considering partitions may require us to sacrifice definability of the pieces in continuous logic, leaving us with two options.
We can either look at partitions of unity on $S_x(M)$ into definable predicates, get an actual partition, but settle for Borel sets of uniform low complexity.

\begin{defn}\label{defn_constructible}
    Inspired by terminology from algebraic geometry, call a subset of $S_x(A)$ \emph{constructible} when it is a finite boolean combination of closed sets.
    We will also refer to the indicator functions of constructible sets in type spaces, or their restrictions to models, \emph{constructible predicates}.
    
    If $P$ is a finite partition of unity on $M^x$, we say that a function $\psi(x;z) : S_{xz}(M) \to [0,1]$ \emph{defines} $P$ when for each piece $\pi \in P$, there is $d \in M^z$
    such that $\pi$ is the support of $\psi(x;d)$.
    If $\psi(x;z)$ is a definable predicate (that is, continuous), then we call $P$ \emph{definable},
    and if $\psi(x;z)$ is a constructible predicate, we call $P$ \emph{constructible}, and as its pieces are all $\{0,1\}$-valued, we may identify $P$ with the partition into the supports of its pieces.
    When different partitions of unity are definable by the same $\psi$, we say that they are \emph{uniformly} definable/constructible.

    If $P$ is a partition of unity on $M^{x_1\dots x_n}$ defined by a definable or constructible predicate of the form $\prod_{i = 1}^n \psi_i(x_i;z_i)$, then we call $P$ \emph{rectangular}.
    If $P_1,\dots,P_n$ are partitions of unity on $M^{x_i}$, then let $\otimes_{i = 1}^n P_i$ denote the partition on $M^{x_1 \dots x_n}$ given by
    $\left(\prod_{i = 1}^n \pi_i(x_i):\pi_1 \in P_1,\dots,\pi_n \in P_n\right)$. We call such a partition a \emph{grid}
    
    If $P$ is a partition of unity on $S_x(M)$ and $B \subseteq M^y$, we call $P$ a $(\phi,\varepsilon)$-\emph{partition} over $B$ when
    for each $\pi(x) \in P$, if $a_1,a_2 \in M^x$ both satisfy $\pi(a_i) > 0$, then $|\phi(a_1;b) - \phi(a_2;b)|\leq \varepsilon$ for all $b \in B$.
\end{defn}

In the classical case, there is a unique minimal such partition whose size is bounded by the VC-dimension of $\phi$.
In this lemma, we will show that a suitable partition exists with size given by a covering number bound.
\begin{lem}\label{lem_cover_pou}
    Let $\phi(x;y;w)$ be a definable predicate.
    Then for any $\varepsilon > 0$ and $n \in \N$, there is a definable predicate $\psi(x;z)$ such that for any $B \subset M^y$ with $|B| \leq n$, and any $c \in M^w$,
    there is a set $D \subset M^z$ with $|D| \leq \mathcal{N}_{\phi(x;y),0.49\varepsilon}(n)$ such that $(\psi(x;d) : d \in D)$ 
    forms a $(\phi,\varepsilon)$-partition over $B$,
    and $\psi$ is a continuous combination of instances of $\phi$ over $D$.

    We may alternately choose $\psi$ to be constructible with the slightly better bound $|D| \leq \mathcal{N}_{\phi(x;y),\varepsilon/2}(n)$, although $\psi$ will no longer be a continuous combination of instances of $\phi$.
\end{lem}
\begin{proof}
    Fix $c$ and write $\phi(x;y) = \phi(x;y;w)$, and we will see that the resulting formula $\psi(x;z)$ is constructed from $\phi(x;y;w)$ in a uniform way, with the parameter $c$ reoccuring in the parameters $d \in D$.

    For ease of notation, let $m = \mathcal{N}_{\phi(x;y),0.49\varepsilon}(n)$.
    Define the predicate
    $$\theta(x;x',y_1,\dots,y_n) = \max_{1 \leq i \leq n} \frac{\varepsilon}{2} \dot{-} |\phi(x;y_i) - \phi(x';y_i)|,$$
    so that for any $a,a' \in M^x,b_1,\dots,b_m \in M^y$, $\theta(a;a',b_1,\dots,b_m) > 0$ if and only if $|\phi(a;b_i) - \phi(a';b_i)| < \frac{\varepsilon}{2}$ for all $1 \leq i \leq n$.
    Similarly, define
    $$\psi(x;x_0,\dots,x_m,y_1,\dots,y_n) = \frac{\theta(x;x_0,y_1,\dots,y_n)}{\sum_{i = 1}^m\theta(x;x_i,y_1,\dots,y_n)}.$$
    
    Now if $B \subseteq M^y$ has $|B| \leq n$, express $B = \{b_1,\dots,b_n\}$, and let $A = \{a_1,\dots,a_m\}$ be a set such that for each $a \in M^x$, there exists $1 \leq i \leq m$
    such that for all $b \in B$, $|\phi(a;b) - \phi(a_i;b)| \leq 0.49\varepsilon$.
    Then $\{\psi(x;a,a_1,\dots,a_m,b_1,\dots,b_n) : a \in A\}$ constitutes a partition of unity,
    and on each piece $\psi(x;a,a_1,\dots,a_m,b_1,\dots,b_n) > 0$, for each $b \in B$, $\phi(x;b)$ varies from $\phi(a;b)$ by at most $\frac{\varepsilon}{2}$,
    so overall, $\phi(x;b)$ varies by at most $\varepsilon$.

    If instead we wish $\psi$ to be constructible, then we instead let $m = \mathcal{N}_{\phi(x;y),\varepsilon/2}(n)$, and we can let $A = \{a_1,\dots,a_m\}$ be a set such that for each $a \in M^x$, there exists $1 \leq i \leq m$
    such that for all $b \in B$, $|\phi(a;b) - \phi(a_i;b)| \leq \frac{\varepsilon}{2}$.
    Then the sets $X_i = \{|\phi(x;b) - \phi(a_i;b)| \leq \frac{\varepsilon}{2}\}$ cover $S_x(M)$, so we can use standard coding tricks to let $\psi(x;z)$ define any of the sets
    $X_i \setminus \bigcup_{j < i}X_i$. 
\end{proof}

\begin{lem}[{See \cite[Proposition 2.18]{nip_reg}}]\label{lem_nip_reg_1}
    Let $M$ be an NIP structure, let $\phi(x;y;w)$ be a definable predicate, let $\varepsilon > 0$,
    and let $\mathrm{vc}_{\varepsilon/12}(\phi(x;y))\leq d$.
    Then for any generically stable measure $\mu \in \mathfrak{M}_x(M)$ and any $c \in M^w$,
    there is a set $A \subseteq M^x$ of size $O(\frac{d}{\varepsilon^2}\ln\frac{d}{\varepsilon})$ such that 
    if $P$ is a $(\phi^*(y;x),\varepsilon)$-partition over $A$, then for $b,b' \in M^y$ in the support of the same piece of $P$,
    $\int_{S_x(M)}|\phi(x;b;c) - \phi(x;b';c)|\,d\mu \leq \varepsilon$.
\end{lem}
\begin{proof}
    Fix $c$ and write $\phi(x;y) = \phi(x;y;w)$, and the resulting predicate $\psi(y;z)$ will be constructed from $\phi(x;y;w)$ in a uniform way, based on the predicate defined in Lemma \ref{lem_cover_pou}.

    Let $\chi(x;y,y') = |\phi(x;y) - \phi(x';y)|$, and using Corollary \ref{cor_keisler_approx}, there is $n = O(\frac{d}{\varepsilon^2}\ln\frac{d}{\varepsilon})$
    such that for each $\mu$, there is a $\frac{\varepsilon}{2}$-approximation with respect to $\mu$ of size at most $n$.
    Now for every $\mu$, let $A = \{a_1,\dots,a_n\}$ be a $\frac{\varepsilon}{2}$-approximation to $\chi(x;y,y')$ with respect to $\mu$,
    and let $P$ be a $(\phi^*(y;x),\varepsilon)$-partition over $A$.
    If $\pi(x) \in P$ and $b_1,b_2 \in M^y$ both satisfy $\pi(x) > 0$, then for all $a \in A$, $|\phi(a;b_1) - \phi(a;b_2)| \leq \frac{\varepsilon}{2}$,
    so $\chi(a;b_1,b_2) \leq \frac{\varepsilon}{2}$, and thus by the $\frac{\varepsilon}{2}$-approximation definition of $A$,
    $$\int_{S_x(M)}|\phi(x;b_1) - \phi(x;b_2)|\,d\mu = \int_{S_x(M)}\chi(x;b_1,b_2)\,d\nu \leq \frac{1}{n}\sum_{i = 1}^n\chi(a_i;b_1,b_2) + \frac{\varepsilon}{2} \leq \varepsilon.$$
\end{proof}

\begin{lem}[{See \cite[Theorem 2.19]{nip_reg}}]\label{lem_nip_reg_2}
    Assume $T$ is NIP and $M \vDash T$.
    Let $\phi(x_1,\dots,x_n;w)$ be a definable predicate and let $\varepsilon > 0$.
    Then there is a definable predicate $\theta(x_1,\dots,x_n;z)$ of the form $\sum_{j = 1}^m \prod_{i = 1}^n\theta_{ij}(x_i;z_i)$
    with $m$ depending only on $\phi$ and $\varepsilon$
    such that if $\mu_i \in \mathfrak{M}_{x_i}(M)$ are measures with $\mu_i$ generically stable for $i < n$ and $c \in M^w$,
    then there is $d \in M^z$ such that
    $$\int_{S_{x_1 \dots x_n}(M)} \left|\phi(x_1,\dots,x_n;c) - \theta(x_1,\dots,x_n;d)\right|\,d\mu_1\otimes \dots \otimes \mu_n \leq \varepsilon.$$
\end{lem}
\begin{proof}
    As before, fix $c$ and write $\phi(x_1,\dots,x_n) = \phi(x_1,\dots,x_n;c)$, and the resulting predicates will be constructed from $\phi(x;y;w)$ in a uniform way, based on the predicate defined in Lemma \ref{lem_nip_reg_1}.

    We start by proving the two-dimensional case.
    Let $\phi(x;y)$ be an $M$-definable predicate and let $\varepsilon > 0$.
    Let $\psi(y;z), k$ be as given in Lemma \ref{lem_nip_reg_1}, and let $m$ be the upper bound on the size of the resulting partition of unity.
    Then we will show that for any measures $\mu \in \mathfrak{M}_x(M), \nu \in \mathfrak{M}_y(M)$ with $\mu$ generically stable,
    there are parameters $b_1,\dots,b_m \in M^y,d_1,\dots,d_m \in M^z$ such that
    $$\int_{S_{xy}(M)}\left|\phi(x;y) - \sum_{i = 1}^m\phi(x;b_i)\psi(y;d_i)\right|\,d\mu \otimes \nu \leq \varepsilon.$$

    Given $\mu, \nu$, let $D = \{d_1,\dots,d_m\}$ be such that $(\psi(y;d_i) : 1 \leq i \leq m)$ is the partition of unity given by Lemma \ref{lem_nip_reg_1},
    and for each $1 \leq i \leq m$, let $b_i \in M^y$ be such that $\psi(b_i;d_i) > 0$.
    Then for all $b \in M^y$, we calculate that 
    \begin{align*}
        &\int_{S_x(M)}\left|\phi(x;b) - \sum_{i = 1}^m\phi(x;b_i)\psi(b;d_i)\right|\,d\mu \\
        =& \int_{S_x(M)}\sum_{i = 1}^m|\phi(x;b)\psi(b;d_i) - \phi(x;b_i)\psi(b;d_i)|\,d\mu \\
        \leq & \sum_{i = 1}^m \psi(a;d_i)\int_{S_x(M)}|\phi(x;b) - \phi(x;b_i)|\,d\mu \\
        \leq & \sum_{i : \psi(b;d_i) > 0} \psi(b;d_i)\int_{S_x(M)}|\phi(x;b) - \phi(x;b_i)|\,d\mu \\
        \leq & \varepsilon, \\
    \end{align*}
    as for each $i$ with $\psi(b;d_i) > 0$, $\int_{S_x(M)}|\phi(x;b) - \phi(x;b_i)|\,d\mu \leq \varepsilon$ by assumption.
    Thus also integrating over $y$, we find that $\int_{S_{xy}(M)}\left|\phi(x;y) - \sum_{i = 1}^m\phi(x;b_i)\psi(x;d_i)\right|\,d\mu \otimes \nu \leq \varepsilon$,
    finishing the base case.

    Now assume this works for all $n$-ary predicates, and consider $\phi(x_1,\dots,x_{n+1})$.
    Applying our proof to the repartitioned binary predicate $\phi(x_1,x_2,\dots;x_{n+1})$, we see that there is $\psi(x_{n+1};z_{n+1})$, and some $m_n$ such that for all generically stable $\mu_1,\dots,\mu_{n+1}$,
    there are $(b_1,\dots,b_{m_n})$ and $(d_1,\dots,d_{m_n})$ such that
    $$\int_{S_{x_1\dots x_{n+1}}(M)}\left|\phi(x_1,\dots,x_{n+1}) - \sum_{i = 1}^{m_n}\phi(x_1,\dots,x_{n};b_i)\psi(x_{n+1};d_i)\right|\,d\mu_1\otimes \dots \otimes \mu_{n+1} \leq \frac{\varepsilon}{2}.$$
    We now apply the induction hypothesis to each $\phi(x_1,\dots,x_{n};b_i)$, and we see that there is some $\theta(x_1,\dots,x_n;z)$ that belongs to the tensor product of the algebras of definable predicates on the separate variables $x_1,\dots,x_n$
    such that for all generically stable $\mu_1,\dots,\mu_n$, and every $b \in M^{x_{n+1}}$, there is some $c \in M^z$ such that
    $$\int_{S_{x_1\dots x_{n}}(M)}\left|\phi(x_1,\dots,x_{n},b) - \theta(x_1,\dots,x_{n};c)\right|\,d\mu_1\otimes \dots \otimes \mu_{n} \leq \frac{\varepsilon}{2}.$$
    Then for any $\mu_1,\dots,\mu_{n+1}$, there are $(b_1,\dots,b_n),(d_1,\dots,d_n)$ as above, and for each $b_i$ we choose $c_i \in M^z$ as above.
    Letting $\omega$ abbreviate $\mu_1\otimes \dots \otimes \mu_{n+1}$, we calculate
    \begin{align*}
        &\int_{S_{x_1\dots x_{n+1}}(M)}\left|\phi(x_1,\dots,x_{n+1}) - \sum_{i = 1}^{m_n}\theta(x_1,\dots,x_{n};c_i)\psi(x_{n+1};d_i)\right|\,d\omega \\
        \leq& \frac{\varepsilon}{2} + \int_{S_{x_1\dots x_{n+1}}(M)}\left|\sum_{i = 1}^{m_n}\phi(x_1,\dots,x_{n};b_i)\psi(x_{n+1};d_i) - \sum_{i = 1}^{m_n}\theta(x_1,\dots,x_{n};c_i)\psi(x_{n+1};d_i)\right|\,d\omega \\
        \leq& \frac{\varepsilon}{2} + \sum_{i = 1}^{m_n}\int_{S_{x_1\dots x_{n+1}}(M)}\psi(x_{n+1};d_i)|\phi(x_1,\dots,x_{n};b_i) - \theta(x_1,\dots,x_{n};c_i)|\,d\omega \\
        \leq& \frac{\varepsilon}{2} + \frac{\varepsilon}{2}\int_{S_{x_1}(M)}\sum_{i = 1}^{m_n}\psi(x_1;d_i)\,d\mu_1 \\
        =& \varepsilon.
    \end{align*}
\end{proof}

We now define the precise kind of regularity partition that NIP allows us to find, as well as homogeneous tuples, which will be essential for distal regularity.
\begin{defn}\label{defn_NIP_regularity}
    A $(\varepsilon,\delta)$-NIP regularity partition for $\phi(x_1,\dots,x_n)$ with respect to measures $\mu_i \in \mathfrak{M}_{x_i}(M)$ is a grid partition defined by $\psi(x;z) = \prod_{i = 1}^n \psi_i(x_i;z_i)$ over a set $D$
    such that there is a subset $D_0 \subset D$ with
    $$\sum_{d \in D_0}\int_{S_{x}(M)}\psi(x;d)\,d\mu_1 \times \dots \times \mu_n \leq \delta$$
    and for each $d \in D \setminus D_0$, there are is a value $r_d$ such that
    $$\int_{S_{x}(M)}\psi(x;d)|\phi(x) - r_d|\,d\mu_1 \times \dots \times \mu_n \leq \varepsilon\int_{S_{x}(M)}\psi(x;d)\,d\mu_1 \times \dots \times \mu_n.$$

    If $\phi(x_1,\dots,x_n)$ is a definable predicate, $\varepsilon > 0$, and $A_i \subseteq M^{x_i}$ for each $i$, we say that $(A_1,\dots,A_n)$ is $(\phi,\varepsilon)-$\emph{homogeneous}
    when for all $a, a' \in A_1\times \dots \times A_n$, $|\phi(a) - \phi(a')| \leq \varepsilon$.
    
    We also say that definable/constructible predicates $\psi_i(x_i)$ are $(\phi,\varepsilon)-$homogeneous when their supports are,
    or we may say this about their product $\prod_{i=1}^n\psi_i(x_1)$.
\end{defn}

\begin{lem}\label{lem_hom_pou}
    For any definable predicate $\theta(x_1,\dots,x_n;y)$ which of the form
    $\sum_{j = 1}^m \prod_{i = 1}^n \theta_{ij}(x_i;y_{ij})$, there is a predicate $\psi(x_1,\dots,x_n;z) = \prod_{i = 1}^n \psi_i(x_i;z_i)$
    such that for any $b \in M^y$, $\psi$ defines a grid partition of unity such that each piece is $(\theta(x_1,\dots,x_n;b),\varepsilon)$-homogeneous,
    and the size of the partition is bounded by a function of $m,n$.
    We can choose $\psi$ to be either definable or constructible.
\end{lem}
\begin{proof}
    Let $N$ be large enough that $m\left(\left(1 + \frac{2}{N}\right)^n-1\right) \leq \frac{\varepsilon}{2}$.
    Let $(f_k : 0 \leq k \leq N)$ be a continuous partition of unity on $[0,1]$ such that the support of each $f_k$ lies in the interval $\left(\frac{k - 1}{N},\frac{k + 1}{N}\right)$.
    Then for any $\phi(x)$, $(f_k\circ \phi : 0 \leq k \leq N)$ is a partition of unity on $M^x$, such that on each piece, $\phi$ varies by at most $\frac{2}{N}$.
    By taking products of the partitioning functions, we can find a definable partition of unity $P_i$ on each $M^{x_i}$ such that each $\theta_{ij}(x_i)$ varies by at most $\frac{2}{N}$ on each piece, and set $P = \otimes_{i = 1}^n P_i$.
    Then for each $\pi \in P$, the function $\theta(x_1,\dots,x_n)$ varies from some value by at most $m\left(\left(1 + \frac{2}{N}\right)^n-1\right)\leq \frac{\varepsilon}{2}$,
    and thus vary in total by at most $\varepsilon$, on the support of $\pi$.
    Each partition $P_i$ is a refinement of $m$ partitions of unity into $N + 1$ pieces, so $|P_i| \leq (N + 1)^m$, and thus $|P| \leq (N + 1)^{mn}$.
    
    If instead we desire a constructible partition, then we can let $f_k$ be indicator functions of the intervals $\left[\frac{k}{N},\frac{k + 1}{N}\right)$.
\end{proof}

\begin{thm}[{See \cite[Theorem 3.3]{nip_reg}}]\label{thm_nip_reg}
    Assume $T$ is NIP and $M \vDash T$.
    Let $\phi(x_1,\dots,x_n)$ be a definable predicate and let $\varepsilon, \delta > 0$.
    Then there is a predicate $\psi(x_1,\dots,x_n;z) = \prod_{i = 1}^n \psi_i(x_i;z_i)$ such that for all measures $\mu_i \in \mathfrak{M}_{x_i}(M)$ with $\mu_i$ generically stable for $i < n$,
    $\psi$ defines a $(\varepsilon,\delta)$-NIP regularity partition of unity for $\phi(x_1,\dots,x_n)$ with respect to the $\mu_i$s, with the size of the partition depending only on $\phi,\varepsilon,\delta$.
    We can choose $\psi$ to be either definable or constructible.
\end{thm}
\begin{proof}
    By Lemma \ref{lem_nip_reg_2}, we know that there is a definable predicate $\theta(x_1,\dots,x_n;z)$ of the form $\sum_{j= 1}^m \prod_{i = 1}^n\theta_{ij}(x_i;z_i)$ such that
    if $\mu_i \in \mathfrak{M}_{x_i}(M)$ are measures with $\mu_i$ generically stable for $i < n$ and $c \in M^w$,
    then there is $d \in M^z$ such that
    $$\int_{S_{x_1 \dots x_n}(M)} \left|\phi(x_1,\dots,x_n;c) - \theta(x_1,\dots,x_n;d)\right|\,d\mu_1\otimes \dots \otimes \mu_n \leq \delta^2.$$

    We now apply Lemma \ref{lem_hom_pou} to $\theta$, finding that there is some predicate $\psi$ that for any $d \in M^z$, defines a grid partition of unity where each piece is $\left(\theta(x_1,\dots,x_n;d),\varepsilon\right)$-homogeneous.
    Then for any appropriate measures $\mu_i$ and $c \in M^w$, denoting $\omega = \mu_1 \otimes \dots \otimes \mu_n$, we find $d \in M^z$ as before, and let $P$ be the grid partition of unity defined by $\psi$ with homogeneous pieces.
    Now define $e : P \to [0,1]$ by
    $$e(\pi) = \frac{\int_{S_{x}(M)} \left|\phi(x;c) - \theta(x;d)\right|\pi(x)\,d\omega}
    {\int_{S_{x}(M)} \pi(x)\,d\omega}.$$
    Giving $P$ the measure $\mu(\{\pi\}) = \int_{S_{x}(M)} \pi(x)\,d\omega$, we see that 
    $$\int_P e\,d\mu = \sum_{\pi \in P}\int_{S_{x}(M)} \left|\phi(x;c) - \theta(x;d)\right|\pi(x)\,d\omega \leq \delta^2,$$
    so by Markov's inequality, the measure of all $\pi \in P$ such that $\int_{S_{x}(M)} \left|\phi(x;c) - \theta(x;d)\right|\pi(x)\,d\omega > \delta$
    is at most $\delta$.
    If that set is $P_0 \subseteq P$, then
    $\sum_{\pi \in P_0}\int_{S_{x}(M)} \pi(x)\,d\omega \leq \delta$,
    and for each $\pi \in P \setminus P_0$, we find that 
    $\int_{S_{x}(M)} \left|\phi(x;c) - \theta(x;d)\right|\pi(x)\,d\omega \leq \delta$.

    Now suppose $\pi \in P \setminus P_0$.
    By $\left(\theta,\varepsilon\right)$-homogeneity, there is some $r_\pi$ be such that on the support of $\pi$, $|\theta(x;d) - r_\pi| \leq \frac{\varepsilon}{2}$,
    and thus $\left|\phi(x;c) - r_\pi\right| \dot{-} \frac{\varepsilon}{2} \leq \left|\phi(x;c) - \theta(x;d)\right|,$
    so $\int_{S_{x}(M)} \left|\phi(x;c) - r_\pi\right|\dot{-}\frac{\varepsilon}{2}\pi(x)\,d\omega \leq \delta$, and thus this partition is a $(\varepsilon,\delta)$-NIP regularity partition.
\end{proof}

\subsection{Weak Orthogonality and Strong Erd\H{o}s-Hajnal}
In this subsection, we introduce continuous versions of the notions of regularity that characterize weakly orthogonal measures, and thus distality.

\begin{defn}\label{defn_SEH}
    We say that a predicate $\phi(x_1,\dots,x_n)$ has the $\varepsilon$-\emph{strong Erd\H{o}s-Hajnal property} (or $\varepsilon$-SEH) when there exists $\delta$ such that for any finite sets $A_i \subseteq M^{x_i}$,
    there are subsets $B_i \subseteq A_i$ such that $|B_i| \geq \delta |A_i|$ and $(B_1,\dots,B_n)$ is $(\phi,\varepsilon)$-homogeneous.

    A predicate $\phi$ has the \emph{definable} $\varepsilon$-SEH with respect to measures $\mu_1,\dots,\mu_n$ when there are predicates $\psi_i(x_i;z_i)$ and $\delta > 0$ such that
    there are parameters $d_i \in M^{z_i}$ such that for each $i$, $\int_{S_{x_i}(M)}\psi_i(x_i;d_i)\,d\mu_i \geq \delta$ and the supports of $\psi_i(x_i;d_i)$ are $(\phi,\varepsilon)$-homogeneous.
    If for some $\phi(x_1,\dots,x_n;y)$, some $\varepsilon > 0$, and some class of tuples of Keisler measures, $\phi(x_1,\dots,x_n;b)$ has the definable $\varepsilon$-SEH with respect to all tuples of measures in that class,
    and the same $\delta > 0$ and predicates $\psi_i(x_i;z_i)$ can be used in each case, then we say that $\phi$ has the \emph{uniformly definable} $\varepsilon$-SEH with respect to that class of tuples of measures.
\end{defn}

In classical logic, the definable strong Erd\H{o}s-Hajnal property implies a Szemer\'edi-style regularity lemma, as in \cite[Section 5]{distal_reg},
We will now define a real-valued version of this regularity property.

\begin{defn}\label{defn_distal_reg_part}
    If $P$ is a rectangular partition of unity on $M^{x_1\dots x_n}$, then we call $P$ a $(\varepsilon,\delta)$-\emph{distal regularity partition} for $\phi$ with respect to measures $\mu_i \in \mathfrak{M}_{x_i}(M)$
    when 
    $$\sum_{\pi_1,\dots,\pi_n} \prod_{i = 1}\int_{S_{x_i}(M)}\pi_i(x_i)\,d\mu_i \leq \delta,$$
     where the sum ranges over all $(\pi_1,\dots,\pi_n)$ with $\prod_{i = 1}\pi_i \in P$ that are not $(\phi,\varepsilon)$-homogeneous.
\end{defn}

In a series of lemmas, we will show that a fixed tuple of measures $\mu_1,\dots,\mu_n$ is weakly orthogonal if any of several equivalent regularity properties hold.
We will prove the implications of this equivalence with enough detail to later show that if any of these properties holds in a uniformly definable way across all generically stable Keisler measures, then the theory is distal.

Recall that by \emph{a product measure of continuous localizations of} $\mu_1,\dots,\mu_n$, we mean a measure $\omega$ such that there exist $M$-definable predicates $\theta_i(x_i)$
such that for all $M$-definable predicates $\phi_i(x_i)$,
$$\int_{S_{x_1\dots x_n}(M)}\prod_{i = 1}^n \phi_i(x_i)\,d\omega = \prod_{i = 1}^n\int_{S_{x_i}(M)}\phi_i(x_i)\theta_i(x_i)\,d\mu_i.$$

\begin{thm}\label{thm_ortho_reg}
    Let $\mu_1 \in \mathfrak{M}_{x_1}(M),\dots,\mu_n \in \mathfrak{M}_{x_n}(M)$.
    The following are equivalent:
    \begin{itemize}
        \item The measures $\mu_1,\dots,\mu_n$ are weakly orthogonal.
        \item For each $M$-definable predicate $\phi(x_1,\dots,x_n)$ and each $\varepsilon, \delta > 0$, there is some $C$ such that $\phi$ admits a definable $(\varepsilon,\delta)$-distal regularity partition
        \item For each $M$-definable predicate $\phi(x_1,\dots,x_n)$ and each $\varepsilon, \delta > 0$, there is some $C$ such that $\phi$ admits a constructible $(\varepsilon,\delta)$-distal regularity partition
        \item For each $M$-definable predicate $\phi(x_1,\dots,x_n)$ and each $\varepsilon > \gamma \geq 0$, there is some $\delta > 0$ such that for any product measure $\omega$ of continuous localizations of $\mu_1,\dots,\mu_n$,
            if $\int_{S_{x_1\dots x_n}(M)}\phi\,d\omega \geq \varepsilon$, then there are $M$-definable predicates $\psi_i(x_i)$ 
            such that $\phi(a_1,\dots,a_n) \geq \gamma$ whenever $\psi_i(a_i) > 0$ for each $i$, and $\int_{S_{x_i}(M)}\psi_i(x_i)\,d\mu_i \geq \delta$ for each $i$.
        \item For each $M$-definable predicate $\phi(x_1,\dots,x_n)$ and each $\varepsilon > 0$, $\phi$ has the definable $\varepsilon$-SEH with respect to any continuous localizations of $\mu_1,\dots,\mu_n$.
    \end{itemize}
    Furthermore, if these hold, then the $(\varepsilon,\delta)$-distal regularity partitions can be chosen to be grid partitions of size $O(\delta^{-C})$ for some constant $C$ depending on $\phi,\varepsilon,\mu_1,\dots,\mu_n$. 
\end{thm}
\begin{proof}
    We will show that weak orthogonality is equivalent to the existence of distal regularity partitions, and then we will relate definable distal regularity partitions to both strong Erd\H{o}s-Hajnal statements.
    We start by restating the result of Corollary \ref{cor_ortho} about weakly orthogonal Keisler measures in terms of partitions of unity.
\begin{lem}\label{lem_ortho_pou}
    For $1 \leq i \leq n$, let $\mu_i \in \mathfrak{M}_{x_i}(M)$ be Keisler measures.
    Let $\phi(x_1,\dots,x_n)$ and $\psi^\pm(x_1,\dots,x_n)$ be $M$-definable predicates such that
\begin{itemize}
    \item The predicates $\psi^\pm(x_1,\dots,x_n)$ are each of the form $\sum_{j = 1}^m \prod_{i = 1}^n\theta_{ij}^{\pm}(x_i)$
    \item For all $(x_1,\dots,x_n)$, $\psi^-(x_1,\dots,x_n) \leq \phi(x_1,\dots,x_n) \leq \psi^+(x_1,\dots,x_n)$.
    \item For any product measure $\omega$ of $\mu_1,\dots,\mu_n$, $\int_{S_{x_1\dots x_n}(M)}(\psi^+ - \psi^-)\,d\omega \leq \varepsilon$.
\end{itemize}
    Then there is a grid partition of unity $P = \otimes_{i = 1}^n P_i$ on $M^{x_1\dots x_n}$,
    which can be chosen to either be definable or constructible,
    such that if for each tuple $\pi \in P$,
    we set $r_\pi^- = \inf_{x : \pi(x) > 0} \phi(x)$ and $r_\pi^+ = \sup_{x : \pi(x) > 0} \phi(x)$,
    and then define
    $$\chi^\pm(x_1,\dots,x_n) = \sum_{\pi \in P} r_{\pi}^\pm\pi(x_1,\dots,x_n),$$
    then 
    \begin{itemize}
        \item For all $(x_1,\dots,x_n)$, $\chi^-(x_1,\dots,x_n) \leq \phi(x_1,\dots,x_n) \leq \chi^+(x_1,\dots,x_n)$.
        \item For any product measure $\omega$ of $\mu_1,\dots,\mu_n$, $\int_{S_{x_1\dots x_n}(M)}(\chi^+ - \chi^-)\,d\omega \leq 2\varepsilon$.
    \end{itemize}
    Furthermore, the definition of $P$ depends only on the predicates $\phi, \psi^+,\psi^-$, and not the parameters used in their definitions.
\end{lem}
\begin{proof}
    By Lemma \ref{lem_hom_pou}, there are predicates $\pi^\pm(x_1,\dots,x_n;z) = \prod_{i = 1}^n \pi_i^\pm(x_i;z_i)$
    such that $\pi^\pm$ defines a rectangular partition of unity such that each piece is $(\psi^\pm(x_1,\dots,x_n),\frac{\varepsilon}{2})$-homogeneous.
    Thus if we let $\pi_i(x_i;z_i^+,z_i^-) = \pi_i^+(x_i;z_i^+)\pi_i^-(x_i;z_i^-)$, then $\pi(x_1,\dots,x_n;z') = \prod_{i = 1}^n \pi(x_i;z_i^+,z_i^-)$ defines a refinement of the two partitions of unity,
    so that each piece is $(\psi^+(x_1,\dots,x_n),\frac{\varepsilon}{2})$-homogeneous and $(\psi^+(x_1,\dots,x_n),\frac{\varepsilon}{2})$-homogeneous.
    
    Then for each $\pi$, we let $r_{\pi}^+ = \sup \psi^+(x_1,\dots,x_n)$ where the sup ranges over the support of $\pi$, and let 
    $r_{\pi}^- = \inf \psi^-(x_1,\dots,x_n)$.
    We then let $\chi^\pm(x_1,\dots,x_n) = \sum_{\pi \in P} r_{\pi}^\pm\pi$.
    By symmetry, it suffices to show that $\psi^+ \leq \chi^+$ and for any product measure $\omega$, $\int_{S_{x_1\dots x_n}(M)}(\chi^+ - \psi^+)\,d\omega \leq \frac{\varepsilon}{2}$.
    The integral fact will follow by showing that $\psi^+ \leq \chi^+ \leq \psi^+ + \frac{\varepsilon}{2}$.
    Let $a \in M^{x_1\dots x_n}$.
    Then for every $\pi \in P$ with $\pi(a) > 0$,
    $\psi^+$ varies by at most $\frac{\varepsilon}{2}$ on the set containing $a$ on which $r_{\pi}^+$ is a supremum,
    so we have $\psi^+(a) \leq r_{\pi}^+ \leq \psi^+(a) + \frac{\varepsilon}{2}$.
    As $\chi^+(a)$ is a convex combination of such numbers, $\psi^+(a) \leq \chi^+(a) \leq \psi^+(a) + \frac{\varepsilon}{2}$.
\end{proof}

\begin{lem}\label{lem_ortho_reg}
    Let $x_1,\dots,x_n$ be variable tuples, and let $\mu_i \in \mathfrak{M}_{x_i}(M)$ for each $i$.
    Then the measures $\mu_i$ are weakly orthogonal if and only if for every $M$-definable predicate $\phi(x_1,\dots,x_n)$ and every $\delta > 0,\varepsilon > 0$, 
    there exists a $(\varepsilon,\delta)$-distal regularity partition $P$ for $\phi$ with respect to the measures $\mu_i$.
\end{lem}
\begin{proof}
    First assume that $\mu_i$ are weakly orthogonal.
By Corollary \ref{cor_ortho} put together with Lemma \ref{lem_ortho_pou}, we can find a rectangular partition of unity $P$
such that if $r_\pi^+ =\sup_{x : \pi(x) > 0}\phi(x)$, $r_\pi^- =\inf_{x : \pi(x) > 0}\phi(x)$,
and $\chi^\pm= \sum_{\pi \in P} r_{\pi}^\pm\pi,$
then $\chi^-(x_1,\dots,x_n) \leq \phi(x_1,\dots,x_n) \leq \chi^+(x_1,\dots,x_n)$ and for any product measure $\omega$ of $\mu_1,\dots,\mu_n$,
$\int_{S_{x_1\dots x_n}(M)}(\chi^+ - \chi^-)\,d\omega \leq \delta\varepsilon$.

We clearly see for some $\pi = \prod_{i = 1}^n \pi_i \in P$, the tuple $(\pi_1,\dots,\pi_n)$ is $(\phi,\varepsilon)$-homogeneous if and only if $r_\pi^+ - r_\pi^- \leq \varepsilon$.

We calculate that $\int_{S_{x_1\dots x_n}(M)}(\chi^+ - \chi^-)\,d\omega = \sum_{\pi \in P}(r_\pi^+ - r_\pi^-)\int_{S_{x_1\dots x_n}(M)}\pi \,d\omega \leq \delta\varepsilon,$
so placing a measure on the finite set $P$ by giving $\pi$ the measure $\int_{S_{x_1\dots x_n}(M)}\pi \,d\omega$ and applying Markov's inequality,
we see that the measure of the non-homogeneous predicates $\pi$ is at most $\delta$, so this is a $(\varepsilon,\delta)-$distal regularity partition.

On the other hand, if $P$ is a $(\varepsilon,\delta)-$distal regularity partition, then as before, set the notations $r_\pi^+ =\sup_{x : \pi(x) > 0}\phi(x)$, $r_\pi^- =\inf_{x : \pi(x) > 0}\phi(x)$
and $\chi^\pm= \sum_{\pi \in P} r_{\pi}^\pm\pi.$
As before, on the homogeneous pieces, $r_\pi^+ - r_\pi^- \leq \varepsilon$ and for any product measure $\omega$,
$$\int_{S_{x_1\dots x_n}(M)}(\chi^+ - \chi^-)\,d\omega = \sum_{\pi \in P}(r_\pi^+ - r_\pi^-)\int_{S_{x_1\dots x_n}(M)}\pi \,d\omega,$$
where the sum over the homogeneous pieces is at most $\varepsilon$, and the sum over the non-homogeneous pieces is at most $\delta$, so the $\int_{S_{x_1\dots x_n}(M)}(\chi^+ - \chi^-)\,d\omega \leq \delta + \varepsilon$.
As the integrals of $\chi^\pm$ do not depend on the choice of $\omega$, this means that for two different product measures $\omega_1,\omega_2$, we have
$|\int_{S_{x_1\dots x_n}(M)}\phi\,d\omega_1 - \int_{S_{x_1\dots x_n}(M)}\phi\,d\omega_2| \leq \delta + \varepsilon$.
Thus if such partitions exist for $\delta,\varepsilon > 0$ arbitrarily small, $\omega_1 = \omega_2$ and the measures $\mu_i$ are weakly orthogonal.
\end{proof}

We now show that strong Erd\H{o}s-Hajnal behavior on all continous localizations implies a distal regularity partition in a uniform way.

\begin{lem}\label{lem_seh_reg}
    Let $\phi(x_1,\dots,x_n)$ be an $M$-definable predicate, let $\varepsilon, \delta > 0$, and let $\mu_i \in \mathfrak{M}_{x_i}(M)$ be a Keisler measure for each $i$.
    Suppose that for any Borel localizations $\nu_1,\dots,\nu_n$ of $\mu_1,\dots,\mu_n$,
    there exist $M$-definable predicates $\psi_i(x_i;d_i)$ such that for each $i$, $\int_{S_{x_i}(M)}\psi_i(x_i;d_i)\,d\nu_i \geq \delta$ and the supports of $\psi_i(x_i)$ are $(\phi,\varepsilon)$-homogeneous.
    Then for any $\gamma > 0$, there exists a $(\varepsilon,\gamma)$-distal regularity partition for $\phi$ with respect to $\mu_1,\dots,\mu_n$ definable over $D$, with $|D| = O(\gamma^{-C})$ for some $C$ depending only on $\delta$.
    Furthermore, if the predicates $\psi_i(x_i;z_i)$ can be chosen uniformly for all continuous localizations $\nu_1,\dots,\nu_n$ of $\mu_1,\dots,\mu_n$,
    then the distal regularity partition can be defined by a predicate which is a continuous combination of the $\psi_i$s depending only on $\delta, \gamma$.
\end{lem}
\begin{proof}
    First, we note that if $\int_{S_{x_i}(M)}\psi_i(x_i;d_i)\,d\nu_i \geq \delta$, then $\nu_i\left(\psi_i(x_i;d_i) \geq \frac{\delta}{2}\right) \geq \frac{\delta}{2}$.

    We will actually find a set $D \subset M^{z_1\dots z_n}$ such that
    for each $(d_1,\dots,d_n) \in D$, the supports $\psi_i(x_i;d_i)$ are $(\phi,\varepsilon)$-homogeneous, and
    $$\mu_1\times\mu_n\left(\bigcup_{d_1,\dots,d_n}\psi_i(x_i;d_i) \geq \frac{\delta}{2}\right) \geq 1 - \gamma.$$
    Once we have found this, we use Lemma \ref{lem_cover_pou} to partition each $M^{x_i}$ into a partition of unity $P_i$ such that on the support of each piece, for each $d_i \in D_i$, either $\psi_i(x_i;d_i) > 0$
    or $\psi_i(x_i;d_i) < \frac{\delta}{2}$.
    Then if $\pi_i \in P_i$ for each $i$, either there is some $(d_1,\dots,d_n) \in D$ such that the support of each $\pi_i(x_i)$ is contained in the support of $\psi_i(x_i;d_i)$,
    and this tuple of supports is $(\phi,\varepsilon)$-homogeneous, or the product of the supports of $\pi_i$ are disjoint from the set of measure $1 - \gamma$ mentioned before.
    Thus the integrals of the non-$(\phi,\varepsilon)$-homogeneous pieces add up to at most $\gamma$.

    We will construct $D_m$ with $\mu_1\times\dots\times\mu_n\left(\bigcup_{d_1,\dots,d_n}\left[\psi_i(x_i;d_i) \geq \frac{\delta}{2}\right]\right) \geq 1 - \left(1-\left(\frac{\delta}{2}\right)^n\right)^m$ for all $m$ recursively,
    and iterate until $\left(1-\left(\frac{\delta}{2}\right)^n\right)^m \leq \gamma$.
    We simultaneously construct a rectangular constructible partition $P_m$ on $M^{x_1\dots x_n}$, such that $\bigcup_{d_1,\dots,d_n}\left[\psi_i(x_i;d_i) \geq \frac{\delta}{2}\right]$
    is a union of pieces of $P_m$.
    If $X_m$ is the union of all pieces of $P_m$ contained in $\bigcup_{d_1,\dots,d_n}\left[\psi_i(x_i;d_i) \geq \frac{\delta}{2}\right]$, we will make sure that at each stage, $\mu_1 \times \dots \times \mu_n (X_m) \geq 1 - \left(1-\left(\frac{\delta}{2}\right)^n\right)^m$.
    At each stage, we will ensure that $|D_m|,|P_m| \leq (n + 1)^m$.
    
    For $m = 0$, we may use $D = \emptyset$, and let $P_0$ be the trivial 1-piece partition. Here it is possible that $X_0 = \emptyset$.
    Assume for induction that we have $D_m$ and $P_m$ such that $\mu_1 \times \dots \times \mu_n (X_m) \geq 1 - \left(1-\left(\frac{\delta}{2}\right)^n\right)^m$.
    Then to form $D_{m+1}$ and $P_{m+1}$, we will replace each piece $A_1 \times \dots \times A_n$ of $P_m$ into at most $n+1$ pieces, and add at most one element to $D$ for each such piece.
    Let $A_1 \times \dots \times A_n \in P_m$.
    If $A_1 \times \dots \times A_n \subseteq \bigcup_{d_1,\dots,d_n}\psi_i(x_i;d_i) \geq \frac{\delta}{2}$, then it is already in $X_m$, and we can leave this piece in $P_{m+1}$.
    This way we ensure that $X_m \subseteq X_{m+1}$.
    If $\prod_{i = 1}^n \mu_i(A_i) = 0$, then we will still leave this piece, as it does not affect the measure of $X_{m+1}$.
    Otherwise, we can find $d_1,\dots,d_n$ such that the supports of $\psi_i(x_i;d_i)$ are $(\phi,\varepsilon)$-homogeneous and
    if $\nu_i$ is the localization of $\mu_i$ to $A_i$, then $\nu_i\left(\psi_i(x_i;d_i) \geq \frac{\delta}{2}\right) \geq \frac{\delta}{2}$.
    Thus if we add $(d_1,\dots,d_n)$ to $D$ and replace the piece $A_1 \times A_n$ with the $n+1$ pieces 
    $\prod_{i \leq j}\left(A_i \cap \left[\psi_i(x_i;d_i) \geq \frac{\delta}{2}\right]\right) \times \prod_{i > j}\left(A_i \setminus \left[\psi_i(x_i;d_i) \geq \frac{\delta}{2}\right]\right)$,
    we find that the one piece $\prod_{i = 1}^n\left(A_i \cap \left[\psi_i(x_i;d_i) \geq \frac{\delta}{2}\right]\right)$ which definitely contributes to $X_{m+1}$ has total measure at least $\left(\frac{\delta}{2}\right)^n \prod_{i =1}^n \mu_i(A_i)$.
    Thus if we do this for all pieces of $P_m$, the total measure of $M^{x_1\dots x_n} \setminus X_{m+1}$ decreases by a factor of at least $\left(\frac{\delta}{2}\right)^n$,
    meaning that $\mu_1 \times \dots \times \mu_n (X_{m + 1}) \geq 1 - \left(1-\left(\frac{\delta}{2}\right)^n\right)^{m+1}$.
    We also find that $|P_{m+1}| \leq (n+1)|P_m| \leq (n+1)^{m+1}$, and $|D_{m+1}| \leq |D_m| + |P_m| \leq 2(n+1)^m \leq (n+1)^{m+1}$.
    
    Now if we choose $M = \left\lceil \frac{\log \gamma}{\log(1 - \delta^n)}\right\rceil$,
    we find that $\mu_1\times \dots\times\mu_n(X_M) \geq 1 - \gamma$ as desired.
    If $C = -\frac{\log (n + 1)}{\log(1 - \delta^n)}$, then $(n + 1)^M \leq (n + 1)\gamma^{-C}$, so the number of pieces is $O(\gamma^{-C})$.
\end{proof}

From a distal regularity partition, we can derive a statement about integrals of predicates which is analogous to the density version of the strong Erd\H{o}s-Hajnal property in \cite[Section 4]{distal_reg}.
\begin{lem}\label{lem_reg_dseh}
    Let $\mu_i \in \mathfrak{M}_{x_i}(M)$ for $1 \leq i \leq n$, let $\phi(x_1,\dots,x_n)$ be an $M$-definable predicate, and let $(\varepsilon,\delta), \psi_1(x_1;z_1),\dots,\psi_n(x_n;z_n)$
    be such that $\prod_{i = 1}^n\psi_i(x_i;z_i)$ defines a $(\varepsilon,\delta)$-distal regularity partition for $\phi$ with respect to $\mu_1,\dots,\mu_n$ of size $K$.

    Suppose that $\alpha,\beta\geq 0$ are such that $\alpha > \beta + \delta + \varepsilon$.
    Then if $\omega$ is a measure extending $\mu_1 \times \dots \times \mu_n$ such that $\int_{S_{x_1 \dots x_n}(M)}\phi \,d\omega \geq \alpha$,
    then there are some $d_i \in M^{z_i}$ such that if $a = (a_1,\dots,a_n)$ satisfies $\prod_{i = 1}^n \psi_i(a_i;d_i) > 0$, then
    $\phi(a) \geq \beta$
    and for each $i$, 
    $$\int_{S_{x_i}(M)} \psi_i(x_i;d_i)\,d\mu_i \geq \frac{\alpha - \beta - \delta - \varepsilon}{K} > 0.$$
\end{lem}
\begin{proof}
    Let $\psi(x;z) = \prod_{i = 1}^n\psi_i(x_i;z_i)$, and let $D$ be the set of parameters such that $(\psi(x;d) : d \in D)$ is the distal regularity partition.
    We will break $D$ into three disjoint sets $A, B, C$.
    Let $A$ be the set of all $d$ such that the support of $\psi(x;d)$ is $(\phi,\varepsilon)$-homogeneous and on that support, $\phi(x) \geq \gamma$.
    Let $B$ be the set of all other $d$ such that the support of $\psi(x;d)$ is $(\phi,\varepsilon)$-homogeneous, and let $C$ be the set of all $d$ such that the support is non-homogeneous.

    Then by the partition of unity assumption,
    \begin{align*}
        &\sum_{d \in A} \int_{S_x(M)}\phi(x)\psi(x;d)\,d\omega + \sum_{d \in B}\int_{S_x(M)}\phi(x)\psi(x;d)\,d\omega + \sum_{d \in C}\int_{S_x(M)}\phi(x)\psi(x;d)\,d\omega\\ 
        =& \int_{S_x(M)}\phi(x)\,d\omega \geq \alpha.
    \end{align*}
    However, by the assumption of homogeneity, on the support of $\psi(x;d)$ for any $d \in B$, $\phi(x) \leq \beta + \varepsilon$,
    so 
    $$\sum_{d \in B}\int_{S_x(M)}\phi(x)\psi(x;d)\,d\omega \leq \beta + \varepsilon,$$
    and by the distal regularity assumption, 
    $$\sum_{d \in C}\int_{S_x(M)}\phi(x)\psi(x;d)\,d\omega \leq \delta,$$
    so 
    $$\sum_{d \in A} \int_{S_x(M)}\phi(x)\psi(x;d)\,d\omega \geq \alpha - \beta - \delta - \varepsilon.$$
    As this is the sum over at most $K$ cells, we find that at least one of the terms of this sum is at least 
    $\frac{\alpha - \beta - \delta - \varepsilon}{K}$, so for the larger integral,
    $$\int_{S_x(M)}\psi(x;d)\,d\omega = \prod_{i = 1}^N \int_{S_{x_i}(M)} \psi_i(x_i;d_i)\,d\mu_i \geq \frac{\alpha - \beta - \delta - \varepsilon}{K}.$$
    As each term in the product is at most 1, they are all at least $\frac{\alpha - \beta - \delta - \varepsilon}{K}$.
\end{proof}

Going full circle, we can use this density property to imply the Strong Erd\H{o}s-Hajnal property.
\begin{lem}\label{lem_dseh_seh}
    Let $\mu_i \in \mathfrak{M}_{x_i}(M)$ for $1 \leq i \leq n$, let $\phi(x_1,\dots,x_n)$ be an $M$-definable predicate, let $s \in \N,\delta > 0,\psi_1(x_1;z_1),\dots,\psi_n(x_n;z_n)$.
    Assume that for all $0 \leq j \leq s$, if $\omega$ is a measure extending $\mu_1 \times \dots \times \mu_n$, and
    $\int_{S_{x_1 \dots x_n}(M)}\frac{1}{s} \dot{-}\left|\phi - \frac{j}{s}\right| \,d\omega \geq \frac{1}{s(s + 1)}$,
    then there are some $d_i \in M^{z_i}$ such that if $a = (a_1,\dots,a_n)$ satisfies $\prod_{i = 1}^n \psi_i(a_i;d_i) > 0$, then
    $\left|\phi - \frac{j}{s}\right| \leq \frac{1}{s}$
    and for each $i$,
    $$\int_{S_{x_i}(M)} \psi_i(x_i;d_i)\,d\mu_i \geq \delta > 0.$$

    Then there are some $d_i \in M^{z_i}$ such that the supports of $\psi_i(x_i;d_i)$ are $\left(\phi,\frac{2}{s}\right)$-homogeneous and
    $$\int_{S_{x_i}(M)} \psi_i(x_i;d_i)\,d\mu_i \geq \delta.$$
\end{lem}
\begin{proof}
    For any $r \in [0,1]$, $\sum_{j = 0}^s \frac{1}{s} \dot{-}\left|r - \frac{j}{s}\right| = \frac{1}{s}.$
    Thus for some $0 \leq j \leq s$, $\int_{S_{x_1 \dots x_n}(M)}\frac{1}{s} \dot{-}\left|\phi - \frac{j}{s}\right| \,d\omega \geq \frac{1}{s(s + 1)}$,
    and we find $d_i \in M^{z_i}$ such that
    and for each $i$,
    $\int_{S_{x_i}(M)} \psi_i(x_i;d_i)\,d\mu_i \geq \delta > 0,$
    and on the support of $\prod_{i = 1}^n \psi_i(a_i;d_i)$, $\left|\phi - \frac{j}{s}\right| \leq \frac{1}{s}$, so the supports of $\psi_i(x_i;d_i)$ are $\left(\phi,\frac{2}{s}\right)$-homogeneous.
\end{proof}

We now finish the proof of Theorem \ref{thm_ortho_reg} by observing that if $\mu_1,\dots,\mu_n$ are weakly orthogonal, then by Lemma \ref{lem_ortho_local}, any continuous localizations are weakly orthogonal, so for any $\phi(x_1,\dots,x_n)$, distal regularity partitions exist, and by Lemma \ref{lem_reg_dseh}, the density version of strong Erd\H{o}s-Hajnal holds for these measures, and thus by Lemma \ref{lem_dseh_seh}, strong Erd\H{o}s-Hajnal holds for these measures.
If we assume that strong Erd\H{o}s-Hajnal holds for any continuous localizations of $\mu_1,\dots,\mu_n$, then by Lemma \ref{lem_seh_reg}, a distal regularity partition exists, so the measures are weakly orthogonal.
This shows that all of the properties are equivalent to weak orthogonality.
Also, Lemma \ref{lem_seh_reg} produces grid partitions of polynomial size.
\end{proof}

\section{Keisler Measures in Distal Theories}\label{sec_distal_measure}
In classical logic, a theory is distal if and only if all generically stable measures are smooth.
We prove that this still holds in continuous logic, and show that it is enough to check that all generically stable measures are weakly orthogonal.
\begin{thm}\label{thm_distal_measure}
    The following are equivalent:
    \begin{itemize}
        \item The theory $T$ is distal
        \item Every generically stable measure is smooth
        \item All pairs of generically stable measures are weakly orthogonal.
    \end{itemize}
\end{thm}
\begin{proof}
First we show that in distal theories, generically stable measures are smooth.
\begin{lem}[{Generalizing \cite[Prop. 9.26]{nip_guide}}]\label{lem_gs_smooth}
Assume $T$ is distal. Then all generically stable measures are smooth.
\end{lem}
\begin{proof}
    Assume $T$ is distal, and let $\mu \in \mathfrak{M}_x(\mathcal{U})$ be a generically stable measure, invariant over a small model $M$.
    To show that $\mu|_M$ is smooth, fix $M \preceq N$, a predicate $\phi(x)$ with parameters in $N$, and $\varepsilon > 0$.
    By \cite[Theorem 5.9]{anderson1}, both $\phi$ and $1-\phi$ admit strong honest definitions, and thus strong$^*$ honest definitions by \cite[Lemma 5.11]{anderson1}.
    Thus there is an extension $(N, P_M) \preceq (N',P_{M'})$ and a predicate $\psi^+(x)$ (a strong$^*$ honest definition for $\phi$ over $M$) with parameters in $M'$
    such that for $a \in M$, $\psi^+(a) = \phi(a)$ and for all $a'$ in $N'$, $\phi(a')\leq \psi^+(a')$.
    By applying the same result to $1 - \phi$ and then subtracting from 1, we can find $\phi^-$, also with parameters in $M'$, such that
    for $a \in M$, $\psi^-(a) = \phi(a)$, and for all $a'$ in $N'$, $\phi(a')\geq \psi^-(a')$.
    To refer to the parameters more easily, we now write $\psi^+(x;d_+)$ and $\psi^-(x;d_-)$, where $d_+, d_- \in M'^z$.
    
    We see that $\psi^+(x;d_+)-\phi(x)$ and $\phi(x)-\psi^-(x;d_-)$, nonnegative everywhere, are both 0 at all tuples in $M$.
    Thus by Lemma \ref{lem_approx_realized} and the fact that $\mu$ is approximately realized in $M$,
    $\int_{S_x(\mathcal{U})}\psi^+(x;d_+)-\phi(x)\,d\mu$ and $\int_{S_x(\mathcal{U})}\phi(x) -\psi^-(x;d_-)\,d\mu$ are both 0.
    As $\mu$ is definable, the function $\int_{S_x(\mathcal{U})}\psi^+(x;z)-\phi(x)\,d\mu$ is continuous from $S_z(N)$ to $\mathbb{R}$.
    Thus there is some some basic open neighborhood of $\mathrm{tp}(d_+/N)$ such that $\int_{S_x(\mathcal{U})}\psi^+(x;z)-\phi(x)\,d\mu < \varepsilon$
    and $\sup_x\left(\phi(x) - \psi^+(x;z)\right)<\varepsilon$ on that neighborhood.
    We may assume that this neighborhood is defined by $\theta(z)<\delta$, where $\theta(z)$ is some formula with parameters in $N$ such that $N'\vDash \theta(d_+)< \delta$.
    By elementarity of the extension $(N,P_M) \preceq (N', P_{M'})$, we know that there is some (possibly infinite) tuple $d_+'$ in $M$ such that $N \vDash \theta(d_+') < \delta$, and thus
    $\int_{S_x(\mathcal{U})}\psi^+(x;d'_+)-\phi(x)\,d\mu < \varepsilon$
    and $\sup_x\left(\phi(x) - \psi^+(x;d'_+)\right)<\varepsilon$.
    Similarly there exists a tuple $d'^-$ in $M$ such that
    $\int_{S_x(\mathcal{U})}\psi^-(x;d'^-)-\phi(x)\,d\mu < \varepsilon$
    and $\sup_x(\psi^-(x;d'^-) - \phi(x))<\varepsilon$.
    Combining these, we see that
    $\int_{S_x(\mathcal{U})}\psi^+(x;d'_+)\,d\mu - \int_{S_x(\mathcal{U})}\psi^-(x;d'^-)\,d\mu < 2\varepsilon$.
    As these formulas have parameters in $M$, and we can bound $\phi(x)$ above and below with
    $\phi^-(x;d'^-)(x) - \varepsilon < \phi(x) < \phi^+(x;d'_+)(x) - \varepsilon$, 
    we see that for any measure $\nu$ extending $\mu|_M$, we have
    $$
        \int_{S_x(M)}\psi^-(x;d'^-)\,d\mu|_M -\varepsilon < \int_{S_x(\mathcal{U})}\phi(x)\,d\nu
    <\int_{S_x(M)}\psi^+(x;d'_+)\,d\mu|_M + \varepsilon
    $$
    limits the value of $\int_{S_x(\mathcal{U})}\phi(x)\,d\nu$ to an interval of width at most $4\varepsilon$ depending only on $\mu|_M$.

    As $\varepsilon$ was arbitrary, we see that $\nu$ is determined by $\mu|_M$, so $\mu$ is smooth.
\end{proof}

Now as smooth measures are weakly orthogonal to all measures, it suffices to show that if all generically stable measures are smooth, then the theory is distal.

\begin{lem}\label{lem_ave_distal}
    Let $I = (a_t : t \in [0,1])$ be an indiscernible segment.
    If for some model $M$ containing $I$, $\mu_I \in \mathfrak{M}_x(M)$ is weakly orthogonal to itself, then $I$ is distal.
\end{lem}
\begin{proof}
    Assume that $I$ is not distal.
    Then there exist points $0 < t_1 < t_2 < 1$ and $b_1, b_2$ such that the sequences
    $I_{[0,t_i)} + b_i + I_{(t_i,1]}$ defined by replacing $a_{t_i}$ with $b_i$ are indiscernible for both $i = 1,2$, but
    the sequence $I_{[0,t_1)} + b_1  + I_{(t_1,t_2)} + b_2 + I_{(t_2,1]}$ defined by making both replacements is not indiscernible.
    By \cite[Lemma 5.4]{anderson1}, we may assume that $\mathrm{tp}(a_{t_i}/M) = \mathrm{tp}(b_i/M) = \lim (I_{[0,t_i)}/M)$ for $i = 1,2$,
    where $M$ is some small model containing $I$.

    By the non-indiscernibility assumption, there is some formula $\phi(y_1,x_1,y_2,x_2,y_3)$ and $c_1,c_2,c_3$ finite subtuples of
    $I_{[0,t_1)}, I_{(t_1,t_2)}, I_{(t_2,1]}$ respectively such that $\phi(c_1,a_{t_1},c_2,a_{t_2},c_3) \neq \phi(c_1,b_1,c_2,b_2,c_3)$.
    Assume that $\phi(c_1,a_{t_1},c_2,a_{t_2},c_3) = 0$ while $\phi(c_1,b_1,c_2,b_2,c_3) = \varepsilon > 0$.
    Let $u_1$ be the maximum index such that $a_{u_1} \in c_1$, let $v_1$ be the minimum index such that $a_{v_2} \in c_2$, let $u_2$ be the maximum index such that $a_{u_2} \in c_2$,
    and let $v_2$ be the minimum index such that $a_{v_2} \in c_3$. Then $0 \leq u_1 < t_1 < v_1 \leq u_2 < t_2 < v_2 \leq 1$.

    If $t_i' \in (u_i,v_i)$ for each $i$, then the partial type $\lim (I_{[0,t_1')}/M) \times \lim (I_{[0,t_2')}/M)$ is consistent with $\phi(c_1,x,c_2,y,c_3) = 0$,
    because there are realizations of these limit types that could replace $a_{t_1'}$ and $a_{t_2'}$ while preserving the indiscernibility of the sequence.
    We will show that $\lim (I_{[0,t_1')}/M) \times \lim (I_{[0,t_2')}/M)$ is also consistent with $\phi(c_1,x,c_2,x,c_3) = \varepsilon$.
    Let $\tau : [0,1] \to [0,1]$ is an order-preserving map that fixes all points in $[0,u_1] \cup [v_1,u_2] \cup [v_2,1]$,
    but $\tau(t_1) = t_1'$ and $\tau(t_2) = t_2'$.
    Then by the homogeneity of $\mathcal{U}$, there is an automorphism $\sigma \in \mathrm{Aut}(\mathcal{U})$ such that for all $t \in [0,1]$, $\sigma(a_t) = a_{\tau(t)}$.
    We see then that $\phi(c_1,\sigma(b_1),c_2,\sigma(b_2),c_3) = \varepsilon$, and that replacing $a_{t_1'}$ with $\sigma(b_1)$ or $a_{t_2'}$ with $\sigma(b_2)$ leaves $I$ indiscernible.
    Thus by \cite[Lemma 5.4]{anderson1}, there are $b_1',b_2'$ with $\mathrm{tp}(b_i'/M) = \lim (I_{[0,t_i')}/M)$ for $i = 1,2$ but $\phi(c_1,b_1',c_2,b_2',c_3) = \varepsilon$.

    This tells us that if $J$ and $K$ are the indiscernible segments obtained by linearly reindexing $I_{[u_1,v_1]}$ and $I_{[u_2,v_2]}$ respectively,
    we find that $\mu_J, \mu_K \in \mathfrak{M}_x(M)$ cannot be weakly orthogonal.
    Otherwise, by Corollary \ref{cor_ortho}, there is an $M$-definable predicate $\psi(x,y) = \sum_{i = 1}^m \theta_i(x)\theta_i'(y)$ such that
    $M \vDash \phi(c_1,x,c_2,y,c_3) \leq \psi(x,y)$, but also for any $\omega$ extending $\mu_J \times \mu_K$,
    $\int_{S_{xy}(M)}\psi(x,y)\,d\omega < \varepsilon.$
    This cannot be true, as for any $t_1' \in [u_1,v_1], t_2' \in [u_2,v_2]$, we have $\psi(a_{t_1'},a_{t_2'}) \geq \varepsilon$,
    and for any $\omega$ extending $\mu_J \times \mu_K$, we have
    $$\int_{S_{xy}(M)}\psi(x,y)\,d\omega = \frac{1}{(v_1 - u_1)(v_2 - u_2)} \int_{u_1}^{v_1}\int_{u_2}^{v_2}\psi(a_t,a_t')\,dt'\,dt \geq \varepsilon.$$

    However, if $\mu_J$ and $\mu_K$ are not weakly orthogonal, then $\mu_I$ is not weakly orthogonal with itself.
    We see this by a proof analogous to that of Lemma \ref{lem_ortho_local}, as
    $\mu_I = (v_1 - u_1)\mu_J + (1 - (v_1 - u_1))\mu_{[0,1]\setminus J} = (v_2 - u_2)\mu_K + (1 - (v_2 - u_2))\mu_{[0,1]\setminus K}$,
    and we see that for any $\omega$ extending $\mu_J \times \mu_K$,
    $\mu_I \otimes \mu_I + (v_1 - u_1)(v_2 - u_2)(\omega - \mu_J \otimes \mu_K)$ will also be a Keisler measure extending $\mu_I \times \mu_I$,
    which will differ from $\mu_I \otimes \mu_I$ if we choose $\omega \neq \mu_J \otimes \mu_K$.
\end{proof}
As average measures for indiscernible segments are generically stable by Lemma \ref{lem_ave_gs}, this completes the proof of Theorem \ref{thm_distal_measure}.
\end{proof}

\subsection{Regularity by way of Weak Orthogonality}

We now generalize the results from \cite{distal_reg} about the Strong Erd\H{o}s-Hajnal property and regularity in distal structures.
First we will use the approach from \cite{simon_distal_reg} to prove a regularity lemma nonconstructively using weakly orthogonal measures and ultraproducts, which we will prove equivalent to distality.
Then in the next subsection we will show the same results using the explicit combinatorial approach from \cite{distal_reg}.

By Theorem \ref{thm_distal_measure}, we know that a theory is distal if and only if all sequences of measures $\mu_1,\dots,\mu_n$ with $\mu_i$ generically stable for $i < n$ are weakly orthogonal,
and thus by Theorem \ref{thm_ortho_reg}, a theory is distal if and only for each such tuples of measures and each predicate $\phi(x_1,\dots,x_n)$,
one of the regularity properties from that theorem applies to $\phi$ over $\mu_1,\dots,\mu_n$.
Now we will show that in fact, if the theory is distal, all of those properties hold in uniformly definable ways.

\begin{lem}\label{lem_ortho_uniform}
    Assume $T$ is distal.
    Then for each definable predicate $\phi(x_1,\dots,x_n;y)$ and each $\varepsilon > 0$, there is a finite set $\Psi$ of definable predicates such that 
    each $\psi(x_1,\dots,x_n;z) \in \Delta$ can be expressed as a sum of predicates of the form $\prod_{i = 1}^n \psi_i(x_i;z_i)$, and
    if $M \vDash T$, $\mu_i \in \mathfrak{M}_{x_i}(M)$ are Keisler measures, with $\mu_i$ generically stable for $i < n$, and $b \in M^y$,
    then there are $\psi^-,\psi^+ \in \Delta, d_-,d_+ \in M^z$ such that if we write $x = (x_1,\dots,x_n)$,
    $\psi^-(x;d_-) \leq \phi(x;y) \leq \psi^+(x;d_+)$ and $\int_{S_x(M)}\psi^+(x;d_+) - \psi^-(x;d_-)\,d\mu_1 \times \dots \times \mu_n \leq \varepsilon$.
\end{lem}
\begin{proof}
    Fix $\phi(x_1,\dots,x_n;y)$ and $\varepsilon > 0$.
    It suffices to show that for some finite set $\Psi$,
    and any model $M$, appropriate measures $\mu_i$, and $b \in M^y$,
    there are $\psi^-, \psi^+ \in \Psi$ and $d_-,d_+ \in M^z$ such that
    $\sup_x \psi^-(x;d_-)\dot{-}\phi(x;y)\leq \frac{\varepsilon}{3}, \sup_x \phi(x;y)\dot{-}\psi^+(x;d_+) \leq \frac{\varepsilon}{3}$ and $\int_{S_x(M)}\psi^+(x;d_+) - \psi^-(x;d_-)\,d\mu_1 \times \dots \times \mu_n \leq \frac{\varepsilon}{3}$.
    If so, then we may simply subtract $\frac{\varepsilon}{3}$ from $\psi^-$ and add $\frac{\varepsilon}{3}$ to $\psi^+$.

    Suppose that no such finite set $\Psi$ works.
    Let $\Sigma$ be the set of all possible definable predicates that can be expressed as finite sums of the form $\prod_{i = 1}^n \psi_i(x_i;z_i)$.
    Let $I$ be the set of finite subsets of $\Sigma$, for every finite subset $\Delta \in I$,
    let $S_\Delta = \{\Delta' \in I  : \Delta \subseteq \Delta'\}$, and let
    $F = \{S \subseteq I : \exists \Delta \in I, S_\Delta \subseteq S \}$.
    This is the standard filter used in the ultrafilter proof of the compactness theorem,
    so there exists an ultrafilter $U$ extending it.    

    For each finite $\Delta \subset \Sigma$, by our contradiction assumption, there are $M, \mu_1,\dots,\mu_n, b$ such that
    for all $\psi^-,\psi^+ \in \Delta, d_-,d_+ \in M^z$,
    \begin{align*}
        \max&\left(\sup_x \left(\psi^-(x;d_-)\dot{-}\phi(x;b)\right),\sup_x \left(\phi(x;b)\dot{-}\psi^+(x;d_+)\right),\right.\\
        &\left.\int_{S_x(M)}\psi^+(x;d_+) - \psi^-(x;d_-)\,d\mu_1 \times \dots \times \mu_n\right) > \frac{\varepsilon}{3}.
    \end{align*}
    Then we let $\tilde M$ be the ultraproduct of all these $M$ with the ultrafilter $U$, and let $\tilde \mu_i$ be the ultralimits of the measures, with $\tilde b$ the ultraproduct of the parameters.
    By Lemma \ref{lem_ultra_gs}, for $i < n$, the measure $\tilde \mu_i$ is generically stable and thus smooth, so by Corollary \ref{cor_smooth_commute}, all of these measures are weakly orthogonal.
    Thus by Lemma \ref{lem_ortho_local}, there are actually some $\psi^-,\psi^+$ and $\tilde d_-,\tilde d_+ \in \tilde M^z$ such that
    $\psi^-(x;d_-) \leq \phi(x;y) \leq \psi^+(x;d_+)$ and $\int_{S_x(M)}\psi^+(x;d_+) - \psi^-(x;d_-)\,d\mu_1 \times \dots \times \mu_n < \frac{\varepsilon}{3}$.
    Thus on a $U$-large set of models $M$, we have
    \begin{align*}
        \max&\left(\sup_x \psi^-(x;d_-)\dot{-}\phi(x;y), \sup_x \phi(x;y)\dot{-}\psi^+(x;d_+),\right.\\
        &\left.\int_{S_x(M)}\psi^+(x;d_+) - \psi^-(x;d_-)\,d\mu_1 \times \dots \times \mu_n\right) < \frac{\varepsilon}{3}.
    \end{align*}
    This contradicts our assumption, which made sure that on the $U$-large set of $\Delta$ containing $\psi^-,\psi^+$, this quantity was greater than $\frac{\varepsilon}{3}$. 
\end{proof}

We can use Lemma \ref{lem_ortho_uniform} to make the definability and constructibility in Theorem \ref{thm_ortho_reg} uniform.
We state these consequences separately as a distal regularity lemma and strong Erd\H{o}s-Hajnal properties.
We note that also by Theorem \ref{thm_ortho_reg}, any of these properties implies weak orthogonality of all generically stable measures, and thus by Theorem \ref{thm_distal_measure}, distality.

\begin{thm}\label{thm_distal_reg}
    Assume $T$ is distal.
    For each definable predicate $\phi(x_1,\dots,x_n;y)$ and $\varepsilon > 0$, there exist predicates $\psi_i(x_i;z_i)$, which can be chosen to be either definable or constructible, and a constant $C$ such that
    if $\mu_1 \in \mathfrak{M}_{x_1}(M),\dots,\mu_n \in \mathfrak{M}_{x_n}(M)$ are such that for $i < n$, $\mu_i$ is generically stable, $b \in M^y$, and $\delta > 0$, the following all hold:
    The predicate $\prod_{i = 1}^n \psi_i(x_i;z_i)$ defines a $(\varepsilon,\delta)$-distal regularity grid partition for $\phi(x_1,\dots,x_n;b)$ of size $O(\delta^{-C})$.
\end{thm}
\begin{proof}
The lemmas in the proof of Theorem \ref{thm_ortho_reg} all preserve the uniformity of predicates.
Thus by starting with Lemma \ref{lem_ortho_uniform}, we see that one of a finite set of predicates can be used to define distal regularity partitions,
which we may assume is a single predicate by standard coding tricks.
\end{proof}

In the case where $|x_1| = \dots = |x_n|$ and all of the measures are equal, we can find a common refinement of the partitions of unity on each piece, and deal with a single partition.
\begin{cor}\label{cor_distal_reg}
    Assume $T$ is distal.
    For each definable predicate $\phi(x_1,\dots,x_n;y)$ with $|x_1| = \dots = |x_n| = |x|$, and $\varepsilon> 0$, there exists a predicate $\psi(x;z)$, which can be chosen to either be definable or constructible, and $\delta > 0$ such that
    if $\mu \in \mathfrak{M}_{x}(M)$ is generically stable, and $b \in M^y$, then
    $\psi$ defines a partition $P$ such that $\otimes_{i = 1}^n P$ is a $(\varepsilon,\delta)$-distal regularity partition for $\phi(x_1,\dots,x_n;b)$,
    such that $|P| = O(\delta^{-C})$.
\end{cor}

Finally, we state the characterization of distality in terms of the definable strong Erd\H{o}s-Hajnal property, generalizing \cite[Theorem 3.1]{distal_reg} and \cite[Corollary 4.6]{distal_reg}.
This follows by applying the equivalences in the proof of Theorem \ref{thm_ortho_reg} to Theorem \ref{thm_distal_reg}. 
\begin{cor}\label{cor_distal_seh}
    A theory $T$ is distal if and only if each definable predicate $\phi(x_1,\dots,x_n;y)$ has the unformly definable 
    $\varepsilon$-strong Erd\H{o}s-Hajnal property with respect to all Keisler measures $\mu_1 \in \mathfrak{M}_{x_1}(M),\dots,\mu_n \in \mathfrak{M}_{x_n}(M)$ for every $\varepsilon > 0$.

    Specifically, there exist definable predicates $\psi_i(x_i;z_i)$ and $\delta > 0$ such that
    if $\mu_1 \in \mathfrak{M}_{x_1}(M),\dots,\mu_n \in \mathfrak{M}_{x_n}(M)$ are such that for $i < n$, $\mu_i$ is generically stable, and $b \in M^y$, then for any
    product measure $\omega$ of $\mu_1,\dots,\mu_n$, there are $d_i \in M^{z_i}$ such that $\psi_i(x_i;d_i)$ are $(\phi(x;b),\varepsilon)$-homogeneous 
            and $\int_{S_{x_i}(M)}\psi_i(x_i;d_i)\,d\mu_i \geq \delta$ for each $i$.

    Furthermore, for any $\varepsilon > \gamma \geq 0$, there are $\psi_i(x_i;z_i)$ and $\delta > 0$ such that
    if $\mu_1 \in \mathfrak{M}_{x_1}(M),\dots,\mu_n \in \mathfrak{M}_{x_n}(M)$ are such that for $i < n$, $\mu_i$ is generically stable, $b \in M^y$,
    and $\omega$ is a product measure of $\mu_1,\dots,\mu_n$ such that $\int_{S_{x_1\dots x_n}(M)}\phi\,d\omega \geq \varepsilon$,
    then there are $d_i \in M^{z_i}$ such that $\phi(a_1,\dots,a_n;b) \geq \gamma$ whenever $\psi_i(a_i;d_i) > 0$ for each $i$, and $\int_{S_{x_i}(M)}\psi_i(x_i;d_i)\,d\mu_i \geq \delta$ for each $i$.
\end{cor}

\subsection{Distal Cutting Lemma}
We now show how to find the predicates defining strong Erd\H{o}s-Hajnal properties and distal regularity partitions more explicitly, using techniques that are also useful for distal combinatorics.

\begin{defn}\label{defn_cutting}
Let $\phi(x;y)$ be a definable predicate, and let $\nu \in \mathfrak{M}_x(M)$ be a generically stable measure.
Then we say that a predicate $\psi(x;z)$ defines a $(\varepsilon,\delta)$-\emph{cutting of weight} $\gamma$ for $\phi(x;y)$ with respect to $\nu$ when
there is a finite set $D \subseteq M^z$ such that
$\inf_x\sum_{d \in D}\psi(x;d) \geq \gamma$, and for each $d \in D$,
$$\nu\left(y : \mathrm{osc}(\phi(x;y),\{a : \psi(a;d) > 0\}) > \varepsilon\right) \leq \delta.$$

A $(\varepsilon,\delta)$-cutting \emph{of size at most} $N$ consists of $\psi(x;z)$ and a particular valid choice of $D$ with $|D| \leq N$.
\end{defn}

\begin{lem}[{Generalizes \cite[Claim 3.5]{distal_reg} and \cite[Theorem 3.2]{cgs}}]\label{lem_cutting}
    If $M$ is distal, and $\phi(x;y)$ is a definable predicate, then for every $\varepsilon, \delta > 0$, there exists a $\gamma > 0$ and a predicate $\psi(x;z)$ that defines a $(\varepsilon,\delta)$-cutting of size at most $O_{\phi,\varepsilon}(\delta^{-1}\ln \delta^{-1})$ and weight at least $\gamma$ with respect to any generically stable measure $\nu \in \mathfrak{M}_y(M)$.
\end{lem}
    \begin{proof}
    Let $M, \phi(x;y), \varepsilon, \delta$ be as above. Let $\theta(x;z)$ be a strong honest definition for $\phi(x;y)$.
    Then we define $\chi(y;z) = \sup_x \left(\phi(x;y) - \theta(x;z)\right) \dot{-} \inf_x \left(\phi(x;y) + \theta(x;z)\right)$.
    
    Let $C = \mathrm{vc}_{\varepsilon/4}(\chi(y;z))$. By Theorem \ref{thm_keisler_net}, there is a $\delta$-net $B$ for the fuzzy set system $(1-\chi)_{1-\varepsilon/2,1}^{M^z}$ with respect to $\nu$,
    with $B = O\left(C\varepsilon^{-1}\ln \varepsilon^{-1}\right).$
    This means that if $d$ is such that $\chi(b;d) = 0$ for all $b \in B$, then $\nu\left(\chi(y;d) > \frac{\varepsilon}{2}\right) < \delta$.
    
    Because $\theta$ is a strong honest definition, for every $a \in M^x$, there is some $d \in B^z$ such that $\theta(a;d) = 0$ and for all $a' \in M^x, b \in B$, $\theta(a';b) \geq |\phi(a;b) - \phi(a';b)|$.
    Thus also for all $b \in B, a' \in M^x$, $\phi(a';b) - \theta(a';d) \leq \phi(a;d) \leq \phi(a';b) + \theta(a';d)$, so $\chi(b;d) = 0$.
    Let $D$ be the set of all $d \in B^z$ with $\chi(b;d) = 0$ for all $b \in B$, and recall then that for each $d \in D$, $\nu\left(\chi(y;d) > \frac{\varepsilon}{2}\right) < \delta$.
    We also find that for any $b \in M^y, d \in D$, and any $a,a' \in M^x$, we have $|\phi(a;b) - \phi(a';b)| \leq \chi(b;d) + \theta(a;d) + \theta(a';d)$.
    Now let $k$ be such that there exists a definable predicate $\theta'(x;y_1,\dots,y_k)$ such that for all $d$, if $(d_1,\dots,d_k)$ is an initial segment of the tuple $d$, then $|\theta'(x;d_1,\dots,d_k) - \theta(x;d)|\leq \frac{\varepsilon}{8}$.
    We find that then there is a set $D_k \subseteq D$ of size at most $|B|^k$ such that for each $d \in D$, there is $(d_1,\dots,d_k) \in D_0$ an initial segment of $d$,
    so for all $a \in M^x$, there is $d \in D_k$ with $\theta(a;d) \leq \frac{\varepsilon}{8}$.
    Thus also $\inf_x\sum_{d \in D_k}\left(\frac{\varepsilon}{4}\dot{-}\theta(x;d)\right) \geq \frac{\varepsilon}{4}$, so we let $\psi(x;z) = \frac{\varepsilon}{4}\dot{-}\theta(x;d)$ and let $\gamma = \frac{\varepsilon}{4}$.
    If $d \in D_k, a,a' \in M^x$ are such that $\psi(a;d),\psi(a';d) > 0$, then $\theta(a;d), \theta(a';d) < \frac{\varepsilon}{4}$,
    and for all $b$ outside a set of $\nu$-measure at most $\delta$, $\chi(b;d) \leq \frac{\varepsilon}{2}$, so $|\phi(a;b) - \phi(a';b)| \leq \chi(b;d) + \theta(a;d) + \theta(a';d) \leq \varepsilon$.
\end{proof}

We can now use a cutting to prove a version of uniformly definable strong Erd\H{o}s-Hajnal, and from it distal regularity, in two variables.
\begin{lem}
Let $\phi(x;y;w)$ be a definable predicate, and let $\varepsilon > 0$.
Then for any $0 < \beta < \frac{1}{50\varepsilon^{-1} + 5}$, there are $0 < \alpha < 1$ and definable predicates $\psi_1(x;z_1)$, $\psi_2(x;z_2)$ such that
for any Keisler measure $\mu \in \mathfrak{M}_x(M)$, any generically stable measure $\nu \in \mathfrak{M}_y(M)$, and any $c \in M^w$, there are $d_1 \in M^{z_1}, d_2 \in M^{z_2}$ such that
$\int_{S_x(M)}\psi_1(x;d_1)\,d\mu \geq \alpha$, $\int_{S_y(M)}\psi_2(y;d_2)\,d\nu \geq \beta$, and the pair $\psi_1(x;d_1), \psi_2(y;d_2)$ is $(\phi(x;y;c),\varepsilon)$-homogeneous.
\end{lem}
\begin{proof}
    We will prove this for some $M$-definable $\phi(x;y) = \phi(x;y;c)$.
    As the formulas $\psi_1(x;z_1)$ and $\phi_2(y;z_2)$ will be constructed from a particular choice of strong honest definition for $\phi(x;y)$,
    it will suffice to show that there is some formula $\theta(x;z;w)$ such that for any $c \in M^z$, $\theta(x;z;c)$ is a strong honest definition for $\phi(x;y;c)$.
    To do this, we find a strong honest definition for $\phi(x;y,w)$, calling this $\theta(x;z,w)$, where $z$ is a tuple of copies of $y$, and we have set all copies of $w$ equal.

    Let $s = \left\lceil\frac{10}{\varepsilon}\right\rceil$, and let $\delta = 1 - 5(s+1)\beta$, so that $\delta > 0$ but also $\beta = \frac{1 - \delta}{5(s + 1)}$.

    As in the proof of Lemma \ref{lem_cutting}, let $\theta(x;z)$ be a strong honest definition for $\phi(x;y)$,
    define $\psi^+(y;z) = \sup_x(\phi(x;y) - \theta(x;z)), \psi^-(y;z) = \inf_x(\phi(x;y) + \theta(x;z))$, and $\chi(y;z) = \psi^+(y;z) - \psi^-(y;z)$.
    Recall that there is some $\gamma > 0$ such that for any choice of $\nu$, there is a finite set $D \in M^z$ of size at most $O_{\phi,\varepsilon}(\delta^{-1}\ln \delta^{-1})$ such that
    for each $d \in D$, $\nu(\chi(y;d) > \frac{1}{s}) < \delta$ and $\inf_x \sum_{d \in D}\left(\frac{\varepsilon}{4} \dot{-}\theta(x;d)\right) \geq \frac{\varepsilon}{4}.$
    
    If we let $\alpha > 0$ be such that $\frac{\gamma}{|D|}\geq \alpha$, then there is always some $d \in D$ such that $\int_{S_x(M)}\psi(x;d)\,d\mu \geq \alpha$.
    We can thus let $\psi_1(x;d_1) = \psi(x;d)$.

    We now let $f_i$ be defined by $f_i(t) = 1 - |st - i|$, so that $(f_0,\dots,f_s)$ is a partition of unity on $[0,1]$ and the support of $f_i$ is $(\frac{i - 1}{s},\frac{i + 1}{s})$.
    Thus $(f_i(\psi^+(y;d))f_j(\psi^-(y;d)) : 0 \leq i,j\leq s)$ forms a partition of unity on $S_y(M)$.
    If $b$ is such that $f_i(\psi^+(b;d))f_j(\psi^-(b;d)) > 0$, then $i \geq j - 1$, and also, $\frac{i - j - 2}{s} \leq \chi(b;d) = \psi^+(b;d) - \psi^-(b;d) \leq \frac{i - j + 2}{s}$.
    Thus if $\chi(b;d) \leq \frac{1}{s}$, we find that $i - j \leq 3$.
    Thus on all such $b$, $\sum_{i,j : -1 \leq i - j \leq 3}f_i(\psi^+(b;d))f_j(\psi^-(b;d)) = 1$.
    The measure of such $b$ is at least $1 - \delta$, so $\sum_{i,j : -1 \leq i - j \leq 3}\int_{S_{y}(M)}f_i(\psi^+(y;d))f_j(\psi^-(y;d))\,d\nu \geq 1 - \delta$,
    and thus for some $i,j$ with $-1 \leq i - j \leq 3$, 
    $\int_{S_{y}(M)}f_i(\psi^+(y;d))f_j(\psi^-(y;d))\,d\nu \geq \frac{1 - \delta}{5(s + 1)} = \beta$,
    so we can let $\psi_2(y;d_2) = f_i(\psi^+(y;d))f_j(\psi^-(y;d))$, using standard coding tricks to account for the finitely many choices of $i,j$.

    We now check homogeneity. For all $b$ in the support of that $f_i(\psi^+(y;d))f_j(\psi^-(y;d))$, and all $a$ in the support of $\frac{\varepsilon}{4} \dot{-}\theta(x;d)$,
    we have that $\frac{j - 1}{s} - \frac{\varepsilon}{4} \leq \phi(a;b) \leq \frac{i + 1}{s} + \frac{\varepsilon}{4}$, and this interval is of width at most $\frac{5}{s} + \frac{\varepsilon}{2} \leq \varepsilon$.
    Thus the pair $\left(\frac{\varepsilon}{4} \dot{-}\theta(x;d),f_i(\psi^+(y;d))f_j(\psi^-(y;d))\right)$ is $(\phi,\varepsilon)$-homogeneous.
\end{proof}

Fixing some $\beta$ and setting $\delta = \min(\alpha,\beta)$, we get an actual definable strong Erd\H{o}s-Hajnal statement.
\begin{cor}\label{cor_seh}
    Let $\phi(x;y;w)$ be a definable predicate, and let $\varepsilon > 0$.
    There are $\delta > 0$ and definable predicates $\psi_1(x;z_1)$, $\psi_2(x;z_2)$ such that
    for any Keisler measure $\mu \in \mathfrak{M}_x(M)$, any generically stable measure $\nu \in \mathfrak{M}_y(M)$, and any $c \in M^w$, there are $d_1 \in M^{z_1}, d_2 \in M^{z_2}$ such that
    $\int_{S_x(M)}\psi_1(x;d_1)\,d\mu \geq \delta$, $\int_{S_y(M)}\psi_2(y;d_2)\,d\nu \geq \delta$, and the pair $\psi_1(x;d_1), \psi_2(y;d_2)$ is $(\phi,\varepsilon)$-homogeneous.
\end{cor}

We now use Lemmas \ref{lem_seh_reg} and \ref{lem_reg_dseh} to show that if the integral of $\phi$ is large enough, the value of the predicate is positive on the whole pair.
This allows us to induct in dimension, and find an alternate proof of Theorem \ref{thm_distal_reg}.

\begin{thm}\label{thm_density_hyper}
    For each definable predicate $\phi(x_1,\dots,x_n;y)$ and $\varepsilon > \gamma \geq 0$, there exist definable predicates $\psi_i(x_i;z_i)$ and $\delta > 0$ such that
    if $\mu_1 \in \mathfrak{M}_{x_1}(M),\dots,\mu_n \in \mathfrak{M}_{x_n}(M)$ are such that for $i < n$, $\mu_i$ is generically stable, and $b \in M^y$,
    for any product measure $\omega$ of $\mu_1,\dots,\mu_n$,
    if $\int_{S_{x_1\dots x_n}(M)}\phi\,d\omega \geq \varepsilon$, then there are $d_i \in M^{z_i}$
    such that $\phi(a_1,\dots,a_n;b) \geq \gamma$ whenever $\psi_i(a_i;d_i) > 0$ for each $i$, and $\int_{S_{x_i}(M)}\psi_i(x_i;d_i)\,d\mu_i \geq \delta$ for each $i$.
\end{thm}
\begin{proof}
    For a base case, we start with Corollary \ref{cor_seh}, and then applying Lemmas \ref{lem_seh_reg} and \ref{lem_reg_dseh}, recalling that all localizations of generically stable measures are generically stable by Corollary \ref{cor_smooth_local}.

    Now assume that this holds all predicates with variables partitioned in $n$ pieces, and consider $\phi(x_1,\dots,x_n,x_{n+1};y)$.
    We repartition it as $\phi(x_1,\dots,x_n;x_{n+1};y)$, and apply the base case to this binary predicate, getting some $\psi(x_1,\dots,x_n;z), \psi_{n+1}(x_{n+1};z_{n+1})$
    such that for any measures $\mu \in \mathfrak{M}_{x_1\dots x_n}(M),\mu_{n+1} \in \mathfrak{M}_{x_{n+1}}(M)$ with $\mu$ generically stable,
    any product measure $\omega$ of $\mu,\mu_{n+1}$, and any $c \in M^y$,
    if $\int_{S_{x_1\dots x_{n+1}}(M)}\phi(x_1,\dots,x_{n+1};c)\,d\omega \geq \varepsilon$, then there are $d \in M^z, d_{n+1} \in M^{z_{n+1}}$
    such that $\phi(a_1,\dots,a_{n+1};c) \geq \gamma$ whenever $\psi(a_1,\dots,a_n;d) > 0$ and $\psi_{n+1}(a_{n+1};d_{n+1}) > 0$,
    and as far as integrals, $\int_{S_{x_1\dots x_n}(M)}\psi(x_1,\dots,x_n;d)\,d\mu \geq \delta_0$ and $\int_{S_{x_{n+1}}(M)}\psi_{n+1}(x_{n+1};d_{n+1})\,d\mu_{n+1} \geq \delta_0$.

    Now we can apply the induction hypothesis to $\psi(x_1,\dots,x_n;z)$, getting predicates $\psi_i(x_i;z_i)$ for $1 \leq i \leq n$ and some $\delta_1 > 0$ such that
    such that for any measures $\mu_i \in \mathfrak{M}_{x_i}(M)$ with each $\mu_i$ generically stable, any product measure $\omega$ of the $\mu_i$s, and any $d \in M^z$,
    if $\int_{S_{x_1\dots x_{n}}(M)}\psi(x_1,\dots,x_n;d)\,d\omega \geq \delta_0$
    there are $d_i \in M^{z_i}$ such that $\psi(a_1,\dots,a_n;d) > 0$ whenever $\psi_i(a_i;d_i) > 0$ for each $i$,
    and for each $i$, $\int_{S_{x_i}(M)}\psi_i(x_i;d_i)\,d\mu_i \geq \delta_1$.

    We now let $\delta = \min(\delta_0,\delta_1)$.
    For any $c \in M^y$, generically stable measures $\mu_i \in \mathfrak{M}_{x_i}(M)$ for $1 \leq i \leq n$, and measure $\mu_{n+1} \in \mathfrak{M}_{x_{n+1}}(M)$,
    and any product measure $\omega$ of the $\mu_i$s, we let $\mu$ be the restriction of $\omega$ to the variables $x_1 \dots x_n$.
    As $\mu$ is a product measure of the $\mu_1,\dots,\mu_n$, and these measures are smooth, it is $\mu_1 \otimes \dots \otimes \mu_n$, which is itself smooth.
    Thus there are $d,d_{n+1}$ such that on the support of $\psi(x_1,\dots,x_n;d)\psi_{n+1}(x_{n+1};d_{n+1})$, $\phi(x_1,\dots,x_n;c) \geq \gamma$,
    $\int_{S_{x_{n+1}}(M)}\psi_{n+1}(x_{n+1};d_{n+1})\,d\mu_{n+1} \geq \delta_0 \geq \delta$,
    and $\int_{S_{x_1\dots x_n}(M)}\psi(x_1,\dots,x_n;d)\,d\mu \geq \delta_0$.
    From this last integral, we see that there are $d_i$ for $1 \leq i \leq n$ such that on the support of $\prod_{i = 1}^n\psi_i(x_i;d_i)$,
    $\psi(x_1,\dots,x_n;d) > 0$, so on the support of $\prod_{i = 1}^{n+1}\psi_i(x_i;d_i)$,
    $\phi(x_1,\dots,x_n;c) \geq \gamma$.
    Also, for each $i$, $\int_{S_{x_i}(M)}\psi_i(x_i;d_i)\,d\mu_i \geq \delta_1 \geq \delta$.
\end{proof}

Now by the equivalences in the proof of Theorem \ref{thm_ortho_reg}, this gives us another proof of Theorem \ref{thm_distal_reg}.

\subsection{Equipartitions}
In classical logic, by \cite[Corollary 5.14]{distal_reg} and \cite[Proposition 3.3]{simon_distal_reg}, distal regularity partitions can be chosen to be \emph{equipartitions}, where the measures of each piece are approximately equal.
In the case of $[0,1]$-valued partitions of unity, it is trivial to split a partition of unity into pieces of approximately equal integral.
However, it is not so easy to repartition a partition into sets of approximately equal measure, and this will require \emph{uniform cutting of generically stable measures}.
In this subsection, we check that we can modify our results about constructible partitions to work with equipartitions in a uniformly constructible way.

\begin{lem}\label{lem_pos_measure_gs}
In a distal structure $M$, if $\mu \in \mathfrak{M}_x(M)$ is a generically stable measure and $p \in S_x(M)$ a type with
$\mu(\{p\})> 0$, then $p$ is realized in $M$.
\end{lem}
\begin{proof}
    By Theorem \ref{lem_gs_smooth}, $\mu$ is smooth, so by Corollary \ref{cor_smooth_local}, the localization measure $\mu_{\{p\}} = \delta_p$ is smooth as well.
    A type that is smooth over $M$ as a measure is realized in $M$, because any nonrealized type has multiple realizations, each of which would be a valid extension.
\end{proof}

\begin{lem}\label{lem_cut_gsm}
Any distal structure $M$ \emph{uniformly cuts generically stable measures}. That is, for every definable predicate $\phi(x;y)$ and $\varepsilon > 0$, there is a definable predicate $\chi(x;z)$ such that if $\mu \in \mathfrak{M}_x(M)$ is a measure such that for all $a \in M^x$, $\mu(\{a\}) = 0$, and $0 \leq r \leq \mu(\phi(M))$, then there exists $c \in M^z$ with $|\mu(\phi(M)\cap \chi(M;c)) - r|\leq \varepsilon$.
\end{lem}
\begin{proof}
    It suffices to show this for the trivial predicate $\phi(x) = 0$, because we can simply replace $\mu$ with its localization to $\phi(M)$, and replace $r$ with $\frac{r}{\mu(\phi(M))}$. This works unless $\mu(\phi(M)) = 0$, when this is still trivial.

    Now fix $\mu$. Let $p \in S_x(M)$ be a type. Assume for contradiction that $\mu(\{p\})> 0$. Then by Lemma \ref{lem_pos_measure_gs}, $p$ is realized, contradicting our assumption on $\mu$, so $\mu(\{p\}) = 0$, and $\mu$ is atomless. As $\mu$ is also regular, if $p \in S_x(M)$, then there must be an open set $U \subseteq S_x(M)$ containing $p$ such that $\mu(U)<\varepsilon$. We can express $U$ as $\psi(x)<\delta$ for some $M$-definable predicate $\psi(x)$ and $\delta \in [0,1]$. As $\psi(p) < \delta$, there is some $\psi(p) < \delta' < \delta$, so if we let $U_p$ be the open set defined by $\psi(x) < \delta'$, and let $F_p$ be the closed set defined by $\psi(x) \leq \delta'$, we find that $p \in U_p \subseteq F_p$ and $\mu(F_p) \leq \mu(U) \leq \varepsilon$.

    By compactness, $S_x(M)$ can be covered with finitely many open sets $U_p$. In fact, there is some $K$ where in can be covered with at most $K$ many open sets $U_p$, where the sets $F_p$ are uniformly definable as $\chi(M;c)$ for various parameters $c \in M^z$. We show this by contradiction. For each finite set $F$ of pairs $(K,\chi(x;z))$, find a generically stable measure $\mu_{F}$ such that this fails for each $(K, \chi) \in F$. By taking an ultraproduct of these counterexamples according to an appropriate ultrafilter, as in the proof of Lemma \ref{lem_ortho_uniform}, we find a generically stable measure $\mu$ such that this fails for every $K$ and every definable predicate $\chi_1(x;z)$. This gives a contradiction, as for every $\mu$, there is some finite cover of open sets $U_p$ contained in closed sets $F_p$, and by the standard coding tricks, a single formula $\chi_1(x;z)$ can be used for each $F_p$ in the finite cover.

    We can then find a formula $\chi(x;z)$ such that for any $k \leq K$ and any $c_1,\dots,c_k$, there is some $c$ such that $\chi(M;c) = \bigcup_{i \leq k} \chi_1(M;c_i)$. We can cover $S_x(M)$ with closed sets $\chi_1(M;c_1),\dots, \chi_1(M;c_K)$, each of measure at most $\varepsilon$, and assume that $c_1,\dots,c_k$ form a minimal subset such that $\mu(\bigcup_{i = 1}^k \chi(M;c_k)) \geq r$. By minimality, we have that $|\mu(\bigcup_{i = 1}^k \chi(M;c_k)) - r| \leq \varepsilon$.
\end{proof}

\begin{lem}\label{lem_cut_fin}
    Any distal structure $M$ \emph{uniformly cuts finite sets}.
    That is, for every definable predicate $\phi(x;y)$ and $\varepsilon > 0$,
    there is a definable predicate $\chi(x;z)$
    such that for any sufficiently large finite set $A \subseteq M^x$, any $b \in M^y$, and any $0 \leq m \leq |\phi(A;b)|$,
    either $|\phi(A;b) = 0|$ or there is some $c \in M^z$ such that
    $$\left|\frac{|\phi(A;b) \cap \chi(A;c)|}{|\phi(A;b)|}-\frac{m}{|\phi(A;b)|}\right|\leq \varepsilon.$$
\end{lem}
\begin{proof}
    It is enough to show this for $\phi(x;y)$ which is uniformly 0.
    Specifically, we will show that for all $\varepsilon > 0$,
    there is a definable predicate $\chi(x;z)$
    such that for any sufficiently large finite set $A \subseteq M^x$, any $b \in M^y$, and any $r \in [0,1]$,
    there is some $c \in M^z$ such that
    $$\left|\frac{|\chi(A;c)|}{|A|}-r\right|\leq \varepsilon.$$
    We can then apply this with $\phi(A;b)$ in place of $A$, and $r = \frac{m}{|\phi(A;b)|}$.

    Assume for contradiction that this does not hold.
    Let $(A_n, r_n) : n \in \N$ be a sequence of counterexamples, with $|A_n| \geq n$ for each $n$.
    That is, for any predicate $\chi(x;z)$, any $c \in M^z$ and any $n \in \N$,
    $$\left|\frac{|\chi(A_n;c)|}{|A_n|}-r_n\right|>\varepsilon.$$

    Now let $\mu_n$ be the uniform measure on $A_n$ for each $n$.
    Fix a nonprincipal ultrafilter on $\N$, and let $\mu$ be the ultraproduct of the $\mu_n$s, and $r \in [0,1]$ be the ultralimit of the $r_n$s.
    Then $\mu$ is a generically stable measure with $\mu(\{c\}) = 0$ for each $c$.
    Thus Lemma \ref{lem_cut_gsm} applies. Let $\chi(x;z)$ be as given by that lemma, but with $\frac{\varepsilon}{2}$ substituted for $\varepsilon$.
    Then there exists some sequence $(c_n : n \in \N)$ such that if $c$ is the element of the ultraproduct representing that sequence,
    $\left|\mu(\chi(M;c)) - r \right|\leq \frac{\varepsilon}{2}$,
    and thus also on a large set of indices $n$,
    $\left|\mu(\chi(M;c_n)) - r_n \right|< \varepsilon.$
    However, for each such $n$,
    $$\left|\mu(\chi(A_n;c_n)) - r_n \right| = 
        \left|\frac{|\chi(A_n;c_n)|}{|A_n|} - r_n \right|,$$
    contradicting the choice of $(A_n,r_n)$.
\end{proof}

Having seen that distal structures uniformly cut finite sets,
we can make the partition in Corollary \ref{cor_distal_reg} a uniformly constructible equipartition, as in \cite[Corollary 5.14]{distal_reg}.

\begin{cor}\label{cor_distal_equi}
    Let $\phi(x_1,\dots,x_k)$ be a definable predicate where each $x_i$ is a copy of the same variable tuple $x$, and fix $\varepsilon > 0$.
    Then there is a constructible predicate $\psi(x;z)$ and some $C > 0$ such that the following holds:
    
    For any generically stable measure $\mu \in \mathfrak{M}_{x}(M)$ such that $\mu(\{a\}) = 0$ for all $a \in M^x$, and any $\gamma,\delta > 0$,
    $\psi$ defines a constructible $(\varepsilon,\delta)$-distal regularity partition $P$ of $M^x$ of size at most $\mathcal{O}(\delta^{-C})$,
    such that each cell in $P$ is uniformly (in terms of $\phi,\varepsilon,\delta, \gamma$) constructible over a set of parameters of size $\mathcal{O}(\delta^{-C})$,
    such that for any two sets $A,B \in P$, $|\mu(A) - \mu(B)|\leq \gamma$.
\end{cor}
\begin{proof}
    We start with $\psi(x;z)$ and $C > 0$ as given by Theorem \ref{thm_distal_reg}, with $\frac{\varepsilon}{2}$ playing the role of $\varepsilon$.
    (This $\psi$ and $C$ will not be the final $\psi$ and $C$.)
    Fix $\mu, \gamma,\delta$.

    We can use the same repartitioning argument from the proof of \cite[Corollary 5.14]{distal_reg} to form an equipartition, using the predicate $\chi$ from Lemma \ref{lem_cut_gsm} to cut the measure $\mu$.
    The resulting equipartition will consist of boolean combinations of pieces from the previous partition and $\chi$-zerosets, and thus with the usual coding tricks, are uniformly constructible.
\end{proof}

\bibliographystyle{plainurl}
\bibliography{ref.bib}

\end{document}